\documentclass[reqno]{amsart}
\usepackage{graphicx} 
\usepackage[margin=1.5in]{geometry}
\usepackage{setspace}
\usepackage{amsthm}
\usepackage[dvipsnames]{xcolor}
\usepackage{amssymb}
\usepackage{amsmath}
\usepackage[mathscr]{euscript}
\usepackage{amsfonts}
\usepackage{mathtools}
\usepackage{enumitem}
\usepackage{hyperref}

\newtheorem{theorem}{Theorem}[section]
\newtheorem{lemma}[theorem]{Lemma}
\newtheorem{proposition}[theorem]{Proposition}

\newtheorem{definition}[theorem]{Definition}
\newtheorem{remark}[theorem]{Remark}

\newtheorem{thm}[theorem]{Theorem}

\newtheorem{prop}[theorem]{Proposition}
\newtheorem{cor}[theorem]{Corollary}
\newtheorem{defn}[theorem]{Definition}

\numberwithin{equation}{section}

\usepackage[normalem]{ulem}
\newcommand{\mydash}{\hbox{\sout{ }}}

\newcommand{\mc}[1]{{\mathcal #1}}
\newcommand{\mf}[1]{{\mathfrak #1}}

\newcommand{\bb}[1]{{\mathbb #1}}

\newcommand{\ms}[1]{{\mathscr #1}}

\newcommand{\rom}[1]{%
  \textup{\uppercase\expandafter{\romannumeral#1}}%
}

\renewcommand{\Cap}{{\rm cap}}

\DeclareMathOperator\supp{supp}

\title[A \emph{$\Gamma$}-convergence of level-two large deviation for metastable systems]{A \emph{$\Gamma$}-convergence of level-two large deviation for metastable systems: The case of zero-range processes}
\author{Kyuhyeon Choi}
\thanks{Department of Mathematical Sciences, Seoul National University, 1, Gwanak-ro, Gwanak-gu, 08826, Seoul, Republic of Korea.}
\thanks{Email: efgd301@snu.ac.kr}

\begin{document}

\begin{abstract}

This study explores the relationship between the precise asymptotics of the level-two large deviation rate function and the behavior of metastable stochastic systems. Initially identified for overdamped Langevin dynamics (Ges{\`u} et al., SIAM J Math Anal 49(4), 3048-3072, 2017), this connection has been validated across various models, including random walks in a potential field. We extend this connection to condensing zero-range processes, a complex interacting particle system.

We investigate a certain class of zero-range processes on a fixed graph $G$ with $N > 0$ particles and interaction parameter $\alpha > 1$. On the time scale $N^2$, this process behaves like an absorbing-type diffusion and converges to a condensed state where all particles occupy a single vertex of $G$ as $N$ approaches infinity. Once condensed, on the time scale $N^{1+\alpha}$, the condensed site moves according to a Markov chain on $G$, showing metastable behavior among condensed states.
The time scales $N^2$ and $N^{1+\alpha}$ are called the pre-metastable and metastable time scales. This behavior is conjectured to be encapsulated in the level-two large deviation $\mc{I}_N$ of the zero-range process. Precisely, the $\Gamma$-expansion of $\mc{I}_N$ is expected to be:
$$
\mc{I}_N = \frac{1}{N^2} \mc{K} + \frac{1}{N^{1+\alpha}} \mc{J},
$$
where $\mc{K}$ and $\mc{J}$ are the level-two large deviations of the absorbing diffusion processes and the Markov chain on $G$, respectively. We rigorously prove this $\Gamma$-expansion by developing a methodology for $\Gamma$-convergence in the pre-metastable time scale and linking the resolvent approach to metastability (Landim et al., J Eur Math Soc, 2023. arXiv:2102.00998) with the $\Gamma$-expansion in the metastable time scale.
\end{abstract}

\maketitle

\section{Introduction} \label{section1}

Metastability, a prevalent phenomenon in various stochastic dynamical systems, refers to the tendency of systems to persist in one locally stable state before transitioning to other stable configurations.
Examples of metastability in various contexts, such as random perturbations of dynamical systems, low-temperature spin systems with ferromagnetic properties, and stochastic models like interacting particle systems, illustrate its widespread occurrence.
This phenomenon is necessary in a stochastic dynamical system with two or more locally stable sets and this stability is parameterized by certain parameters.
Explicitly, we usually consider a sequence of Markov processes of the form $(X_t^N)_{N=1}^{\infty}$ or $(X_t^\beta)_{\beta\geq 0}$ and consider the limiting behavior of this sequence in the regime $N\rightarrow \infty$ or $\beta \rightarrow \infty$.
Here $N$ usually stands for the mesh size of the discretized space (e.g., random walks in a potential field, Ising or Potts model \cite{Ising1, Potential1, L23G}) or the number of particles in the condensing interacting particle systems
(e.g., zero-range processes or inclusion processes \cite{BL09, SIP2, SIP1, SZRP}), while $\beta$ refers to the inverse temperature of the system (e.g., Ising model or Langevin dynamics \cite{Bov2, Ising_beta, Lange}).
In this paper, we discuss an example of zero-range processes for which the Markov processes are parametrized by the number of particles.
Hence, for the simplicity of the discussion, we focus on a sequence of Markov processes of the form $(X_t^N)_{N=1}^{\infty}$.

Metastability encompasses two distinct temporal behaviors: metastable and pre-metastable dynamics.
We characterize a sequence of processes $(X_{t\theta_N}^N)_{N=1}^\infty$ as exhibiting metastable behavior at the time scale $\theta_N$
if the accelerated process remains within specific regions of the state space for an extended period before transitioning almost instantaneously to another region.
These regions are referred to as metastable valleys.
Conversely, the pre-metastable time scale $\sigma_N$ precedes the metastable time scale, during which the accelerated process $(X_{t\sigma_N}^N)_{N=1}^\infty$ tends to be absorbed into these valleys.

The foundation for analyzing metastable behavior was laid by the potential theoretic approach \cite{Bov1, Bov2}, which calculates quantities from probability theory through potential theoretic notions.
This method, in particular, enabled the analysis of metastability in numerous models by effectively dissecting reversible Markov processes.
Building on this, to analyze metastable transitions in models with several metastable valleys, Beltr{\'a}n and Landim \cite{Turnel, Turnel2} introduced the martingale approach.
This methodology essentially demonstrates that under certain conditions, transitions between valleys occur as a Markov process when the model is accelerated on a specific time scale.
It has been effectively applied to interacting particle systems, notably zero-range processes \cite{BL09, Landim2014, SZRP}, establishing that $\theta_N = N^{1+\alpha}$ is the metastable time scale for such processes.

Despite these advances, a connection between metastable and pre-metastable behaviors remained unexplored.
In the context of condensing zero-range processes, it has been demonstrated in \cite{MZRP} that condensation occurs at the pre-metastable time scale $\sigma_N = N^2$,
largely using techniques from Feller processes and martingale problems.
This time scale corresponds to the process achieving limiting diffusion in the Skorokhod topology, which has absorption behavior into the condensation.

A novel perspective through level-two large deviation rate functions has recently integrated metastable and pre-metastable behaviors.
Gesù and Mariani \cite{Fisher} initially introduced the $\Gamma$-expansion of level-two large deviation rate functions in order to merge several time scales in overdamped Langevin dynamics.
Later, Bertini, Gabrielli, and Landim \cite{Bertini} showed that metastable time scales in finite state reversible Markov processes are delineated through the $\Gamma$-expansion of level-two large deviation rate functions.
This approach has been generalized to finite non-reversible cases \cite{L22G}.
Additionally, Landim, Misturini, and Sau \cite{L23G} extended this to models, in which the state space is a discretization of Euclidean space,
presenting that the $\Gamma$-expansion of level-two large deviation rate functions provide both metastable and pre-metastable time scales.
The following sections provide a brief overview of the large deviation theory and a notion of $\Gamma$-convergence.

\subsection*{Level-two large deviation}

In this paper, a large deviation rate function is a prior object of studying the metastability of a Markov process as in \cite{L23G}. We first recall its definition.

Let $S$ be a finite state space and $\ms L$ be a generator of a Markov process $X_t$ on $S$. Let $L_t$ be the empirical measure of the process $X_t$ defined as
\begin{equation*}
    L_t = \frac{1}{t}\int_{0}^{t}\delta_{X_s}ds,
\end{equation*}
where $\delta_x$ is the Dirac measure at $x\in S$.
Therefore, we may consider $L_t$ as a random variable on the space of probability measures on $S$, denoted by $\ms P(S)$.
Under an assumption that the process $X_t$ is irreducible, for any starting point $x\in S$, as $t\rightarrow \infty$,
the empirical measure $L_t$ converges to a unique stationary probability measure $\pi$ on $S$.

Donsker and Varadhan \cite{DV75a} proved the associated large deviation principle. To be specific, for any subset $A$ of $\ms P(S)$,
\begin{equation*}
    -\inf_{\mu\in \mathring{A}} I(\mu) \leq \liminf_{t\rightarrow \infty} \min_{x\in S} \frac{1}{t}\log \bb P_x(L_t\in A)
    \leq \limsup_{t\rightarrow \infty} \max_{x\in S} \frac{1}{t}\log \bb P_x(L_t\in A) \leq -\inf_{\mu\in \bar{A}} I(\mu).
\end{equation*}
Here, $\bb P_x$ stands for the probability measure of the process $X_t$ starting at $x\in S$,
and $\mathring{A}$ and $\bar{A}$ are the interior and the closure of $A$ in $\ms P(S)$, respectively. Also, $I$ is the large deviation rate function defined by
\begin{equation*}
    I(\mu) = \sup_{u>0} -\sum_{x\in S} \frac{\ms L u(x)}{u(x)} \mu(x).
\end{equation*}
This rate function is often called the level-two large deviation.
Let $\pi$ be a stationary measure of the process $X_t$. For the case that $\ms L$ is reversible, by \cite[Theorem 5]{DV75a},
\begin{equation}\label{sqrtcalc}
    I(\mu) = \langle \sqrt f, (-\ms L) \sqrt f \rangle_{\pi}
\end{equation}
for all $\mu\in \ms P(S)$, where $f = \frac{d\mu}{d\pi}$.

Even when $S$ is infinite, a similar theory can be developed if $S$ is compact and $\ms L$ serves as the generator of a Feller process on $S$, given that it satisfies certain nice properties. Precise descriptions can be found in \cite{DV75a}.
In this case, the large deviation rate function $\mc K$ is defined by
\begin{equation*}
    \mc K(\mu) = \sup_{u>0} -\int_{S} \frac{\ms L u}{u} d\mu,
\end{equation*}
where the supremum is taken over all positive functions $u:S\to \bb R$ contained in the domain of $\ms L$.

\subsection*{Metastability and level-two large deviations}

Prior to discussing a full $\Gamma$-expansion, we introduce a notion of $\Gamma$-convergence.

Fix a Polish space $\mc X$ and a sequence $(U_N : N\in\bb N)$ of
functionals on $\mc X$, $U_N\colon \mc X \to [0,+\infty]$.  The
sequence $U_N$ $\Gamma$-converges to the functional
$U\colon \mc X\to [0,+\infty]$ if and only if the following two
conditions are met:

\begin{itemize}
\item [(i)]\emph{$\Gamma$-liminf.} The functional $U$ is a
$\Gamma$-liminf for the sequence $U_N$: for each $x\in\mc X$ and each
sequence $x_N\to x$, we have that
\begin{equation}
\label{ginf}
\liminf_{N\to \infty} U_N(x_N) \;\ge\; U(x) \;.
\end{equation}

\item [(ii)]\emph{$\Gamma$-limsup.} The functional $U$ is a
$\Gamma$-limsup for the sequence $U_N$: for each $x\in\mc X$ there
exists a sequence $x_N\to x$ such that
\begin{equation}
\label{gsup}
\limsup_{N\to \infty} U_N(x_N) \;\le\; U(x)\;.
\end{equation}
\end{itemize}

The application of $\Gamma$-convergence to large deviation theory is concisely summarized in \cite{Mariani}.
Supppose a sequence of functionals $\ms I_N : \ms P(\Xi) \to [0,+\infty)$ are given.
Using $\Gamma$-convergence, we define a full $\Gamma$-expansion of level-two large deviation rate functions, which gives the asymptotics of functional $\ms I_N$ in a precise way in the regime $N\to \infty$.
For two positive sequences $(\alpha_N, N\ge 1)$ and $(\beta_N, N\ge 1)$, we denote by $\alpha_N \prec \beta_N$ if $\alpha_N/\beta_N \to 0$ as $N\to \infty$.
Then, we have the following definition.

\begin{defn}
\label{def1}
A \emph{full $\Gamma$-expansion of $\ms I_N$} is given by the speeds
$(\theta^{(p)}_N, N\ge 1)$, $1\le p\le \mf q$, and the functionals
$\ms I^{(p)}\colon \ms P(\Xi) \to [0, +\infty]$, $0\le p\le \mf q$, if:
\begin{itemize}
\item [(a)] The speeds $\theta^{(1)}_N, \dots, \theta^{(\mf q)}_N$ are
sequences such that $\theta^{(1)}_N \to\infty$,
$\theta^{(p)}_N \prec \theta^{(p+1)}_N$, $0\le p<\mf q$;
\item [(b)] Denote the zero set of $\ms I^{(p)}$ by $\ms N^{(p)}$. Then $\ms I_N$ $\Gamma$-converges to $\ms I^{(0)}$, and for each
$1\le p\le \mf q$, $\theta^{(p)}_N \, \ms I_N$ $\Gamma$-converges
to $\ms I^{(p)}$ on $\ms N^{(p-1)}$;
\item [(c)] For $1\leq p \leq \mf q$, $\ms N^{(p)} \subsetneq \ms N^{p-1}$ and $\ms N^{(\mf q)}$ is a singleton;
\item [(d)] For $0\le p<\mf q$, $\ms I^{(p+1)} (\mu)$ is finite if, and only if $\mu$ belongs to the $0$-level set of $\ms I^{(p)}$;
\end{itemize}

Then, the full $\Gamma$-expansion of $\ms I_N$ reads as
\begin{equation}
    \label{expansion}
    \ms I_N = \ms I^{(0)} + \sum_{p=1}^{\mf q} \frac{1}{\theta^{(p)}_N} \ms I^{(p)}.
\end{equation}
\end{defn}

For the case when $\ms J_N$ represents the level-two large deviation of a Markov process, the functional measures the improbability of a particular measure emerging from an empirical measure.
According to definition (c), we have a nested sequence of sets: $\ms N^{(0)} \supsetneq \ms N^{(1)} \supsetneq \cdots \supsetneq \ms N^{(\mf q)}$.
Additionally, definition (d) suggests that $ \theta^{(p+1)}_{N} $ is the first time scale at which a non-trivial improbability of the set of measures in $ \ms N^{(p)} $ becomes apparent.
Heuristically, we can interpret $\ms N^{(p)} $ as a set of common empirical measures at the time scale $ \theta^{(p)}_N $, where an increasing sequence of time scales establishes a hierarchy of these sets.

If there exists a $\tau$ satistying the following conditions:
\begin{itemize}
\item [(e)] There exists a finite set $M = \{x_1,\dots,x_\kappa\}$ such that $\ms N^{(\tau)}= \{ \sum_{j\in S} w_j \delta_{x_j} : w \in \ms P(S) \}$, where $S = \{1,\dots,\kappa\}$,
\end{itemize}
choose the smallest such $\tau$. Then we refer to the time scales $\theta^{(1)}_N, \dots, \theta^{(\tau)}_N$ as the pre-metastable time scales and $\theta^{(\tau+1)}_N, \dots, \theta^{(\mathfrak{q})}_N$ as the metastable time scales.
Heuristically, the point set $M$ represents the set of metastable valleys, and the full $\Gamma$-expansion implies that meaningful behaviors about $M$ cannot be observed before the scale $\theta^{(\tau+1)}_N$;
this is the reason for the term ``pre-metastable time scales".
Thus, pre-metastable time scales can be understood as the duration needed for the process to concentrate on $M$,
while the metastable time scales are those during which macroscopic transitions over $M$ occur.
Remark from the definition that there may exist multiple metastable, pre-metastable time scales.
This hierarchy of time scales is well demonstrated in \cite{L23G}: two pre-metastable time scales and numerous metastable time scales are identified for the random walk in a potential field.
Also, a similar structure has been clarified in \cite{LLS1, LLS2} for some types of diffusion processes, generalizing Langevin dynamics.
For our case, we proved that the zero-range process has one pre-metastable time scale $N^2$ and one metastable time scale $N^{1+\alpha}$.

In this paper, we aim to calculate the full $\Gamma$-expansion of the large deviation rate function for a zero-range process under a uniform measure condition.
To derive this full $\Gamma$-expansion, we introduce a general framework for proving the $\Gamma$-convergence of the level-two large deviation rate function.
We divide our strategy into two parts: a metastable time scale and a pre-metastable time scale.
In the metastable part, we link the $\Gamma$-convergence to a resolvent approach to metastability, first introduced in \cite{RES}.
Using the resolvent equation, we developed a general tool to handle the $\Gamma$-convergence for the metastable time scales.
This result is discussed in Section \ref{gen_metastable_proof}.

For the pre-metastable part, it is much harder to calculate the $\Gamma$-convergence because the support of a limiting rate function is significantly larger.
Furthermore, due to condensation occurring at the time scale, the limiting dynamics at the pre-metastable time scale also exhibit a form of absorption behavior, complicating the analysis of $\Gamma$-convergence.
Despite these challenges, assuming the existence of a certain limiting diffusion and the reversibility of the process,
we extended the results in \cite{DV75a} to calculate the level-two large deviation rate function associated with the limiting diffusion.
This approach enabled us to obtain the $\Gamma$-convergence of the rate functions for the pre-metastable time scales.
This result is discussed through Sections \ref{slowgammainf}--\ref{slowlimsup}.
Finally, throughout Sections \ref{prelimmeta}--\ref{finalsection}, we apply this general framework to zero-range processes and calculate the full $\Gamma$-expansion of the large deviation rate function.

\section{Main Results for Zero-Range Processes} \label{section2}
In this section, we illustrate a notion of zero-range processes and their metastable and pre-metastable behavior. Precisely, we introduce the level-two large deviation rate function of a zero-range process and demonstrate the full $\Gamma$-expansion of it by analyzing the $\Gamma$-convergence at each time scale.
Most of the prerequisites refer to \cite{MZRP, BL09, SZRP}.

\subsection{Condensing zero-range process} \label{thezrp}

A zero-range process is a system of interacting particles, which is trying to model a stickiness of particles.
In this model, the transition of each particle to another state is governed entirely by the number of particles in its current site and the probability of
transitioning to a new site, which is characterized by a random walk on the site space.
This section introduces the latter mechanism, a Markov process that captures the site transition of the particles.

Let $S$ be a finite set with $|S|=\kappa\geq 2$, and let $X(t):t\geq 0$ be a continuous-time irreducible Markov process on $S$,
so that the jump rate from a site $x\in S$ to a different site $y\in S$ is given by $r(x,y)$ for some $r:S\times S \rightarrow [0,\infty)$.
For convention, we take $r(x,x)=-\sum_{y\neq x}r(x,y)$ for all $x\in S$. The invariant measure of the Markov process $X(\cdot)$ is denote by $m(\cdot)$, namely,
\begin{equation*}
    \sum_{y\in S} m(y)r(y,x) = 0 \;\;\text{for all } x\in S.
\end{equation*}

Here, $m(x), x\in S$ represents the likelihood of particle accumulation at state $x$. 
To simplify and bypass intricate computations, we adopt a uniform measure condition which implies that
\begin{equation}\label{uniformmeasurecond}
    m(x)=\kappa^{-1} \text{ for all } x\in S.
\end{equation}
The reason for this assumption is explained in Section \ref{limitingdiff}.

For $f:S\rightarrow \bb R$, the generator $\mc L_S$ associated with the Markov process $X(\cdot)$ can be written as
\begin{equation*}
    (\mc L_S f)(x) = \sum_{y\in S} r(x,y)(f(y)-f(x)), \:x\in S.
\end{equation*}

Now, we choose the stickiness function which is determined by the number of particles in a site. A parameter $\alpha$ represents the stickiness of a cluster of particles.
In this paper, we assume that $\alpha > 1$. Let $a:\bb N \rightarrow \bb R$ be a function defined by
$$
a(n)=\begin{cases}
    1 &\text{if } n=0, \\
    n^\alpha & \text{if } n\geq 1.
\end{cases}
$$
Additionally, we introduce a function $g:\bb N \rightarrow \bb R$ defined as
$$
g(n)=\begin{cases}
    1 &\text{if } n=0, \\
    a(n)/a(n-1) & \text{if } n\geq 1.
\end{cases}
$$
For $N\in \bb N$, the set $\mc H_N\subset \bb N^S$, representing the set of configuration of $N$ particles on $S$, is defined by
\begin{equation}\label{mchn}
    \mc H_N = \{ \eta = (\eta_x)_{x\in S}\in\bb N^S:\sum_{x\in S} \eta_x = N\}.
\end{equation}
The zero-range process $\{\eta_N(t) : t\geq 0\}$ is defined as a continuous-time Markov process on $\mc H_N$ associated with the generator
\begin{equation} \label{zrpgen}
(\ms L_N f)(\eta) = \sum_{x,y\in S} g(\eta_x)r(x,y)(f(\eta^{x,y})-f(\eta))\; ; \; \eta \in \mc H_N,
\end{equation}
for $f\in \mc H_N \rightarrow \bb R$, where $\eta^{x,y}\in \mc H_N$ is the configuration obtained from $\eta$ by sending a particle from site $x$ to $y$. More precisely, $\eta^{x,y} = \eta$ if $\eta_x = 0$, and if $\eta_x\geq 1$,
$$
(\eta^{x,y})_z = \begin{cases}
    \eta_z-1 & \text{if } z=x \\
    \eta_z+1 & \text{if } z=y \\
    \eta_z & \text{otherwise}.
\end{cases}
$$

\subsection{Stationary measure}

For $\eta\in\mathbb{N}^{S}$, let
\begin{equation}
a(\eta)=\prod_{x\in S}a(\eta_{x})\;.\label{rec1}
\end{equation}
Then the unique stationary probability $\rho_N(\cdot)$ of the zero-range process defined above on $\mc H_N$ is given by
\begin{equation} \label{stationaryprob}
    \rho_N(\eta)=\frac{N^{\alpha}}{Z_{S,N}}\frac{1}{a(\eta)}\;;\; \eta \in \mc H_N\;,
\end{equation}
where $Z_{S,N}$ represents the partition function that make $\rho_N$
a probability measure, that is,
\[
Z_{S,N}=N^{\alpha}\sum_{\eta\in\mc H_N}\frac{1}{a(\eta)}\;.
\]
Define constants $\Gamma(\alpha)$ and $Z_S$ by
\begin{equation}
    \label{gammaalpha}
    \Gamma(\alpha)=\sum_{j=0}^{\infty}\frac{1}{a(j)}\;,\;Z_S \coloneq \kappa\Gamma(\alpha)^{\kappa-1}.
\end{equation}
The series converges because $\alpha>1$.
\begin{prop}\cite[Proposition 2.1]{BL09}.
\label{e21}
The positive constant $Z_S$ depending only on $\kappa$ and $\alpha$ satisfies
\begin{equation} \label{normalizationconst}
    \lim_{N\rightarrow\infty}Z_{S,N}=Z_S\;.
\end{equation}
\end{prop}

\subsection{Euclidean embedding} \label{euclideanembedding}

Let $\Xi \subset \bb R^S$ be a set of non-negative coordinates whose sum is 1, that is,
$$ \Xi = \{(\xi_x)_{x\in S} \in \bb R^S : \xi_x\geq 0 \text{ for all }x\in S \text{ and  } \sum_{x\in S}\xi_x = 1 \}. $$
Noting that $\mc H_N$ consists of coordinates summing to $N$, we embed $\mc H_N$ into $\bb R^S$ as the following:
\begin{equation} \label{iotaembed}
    \iota_N:\mc H_N \rightarrow \Xi, \;(\iota_N(\eta))_x \coloneq \frac{\eta_x}{N}.
\end{equation}
Let $\Xi_N$ be the image of $\mc H_N$ under $\iota_N$, that is,
$$ \Xi_N \coloneq \iota_N(\mc H_N). $$
Consequently, $\Xi_N$ becomes a subset of $\Xi$, characterized by coordinates that are rational numbers with $N$ in their denominators.

Let $\ms P(\Xi)$ be the space of probability measures on $\Xi$. We introduce the rate functional $\mc I_N$ on $\ms P(\Xi)$ using the embedding $\iota_N$.
We first define the level-two large deviation rate function $I_N:\ms P(\mc H_N)\rightarrow [0,+\infty]$ as

\begin{equation}
    \label{plainratefunc}
    I_N(\nu) \coloneq \sup\limits_{u:\mc H_N\rightarrow \bb R_{>0}}\int_{\ms H_N} -\frac{\ms L_N u}{u} d\nu,\;\; \nu\in \ms P(\mc H_N).
\end{equation}
Now, let $\mc I_N$ be the rate functional on $\ms P(\Xi)$ defined by the following:
\begin{equation}
    \label{rate}
    \mc I_N(\mu) \coloneq \begin{cases}
        I_N(\nu), \;\;\mu= \nu\iota_N^{-1} & \text{if } \mu(\Xi_N)=1, \\
        +\infty & \text{otherwise}.
    \end{cases}
\end{equation}
In the first case, $\nu$ is uniquely determined since $\iota_N$ is injective. We understand $\mc I_N$ as a pushforward of $I_N$ under $\iota_N$.

\subsection{Metastablility} \label{metastability24}

In this section, we aim to demonstrate the metastable behavior of the zero-range process. To achieve this, we give a concise description of metastability.

Let $(\ell_N)_{N\in \bb N}$ be a sequence of positive integer satisfying $\ell_N\leq \frac{N}{2}$. Let us define $\mc{E}_{N}^{\ell, x}$, for each $x\in S$,
as a family of disjoint subsets of $\mc H_{N}$ given as
\begin{equation*}
    \mc{E}_{N}^{\ell, x} \coloneq \{\eta\in\mc H_{N}:\eta_{x}\geq N-\ell_N\}\;.
\end{equation*}
Define sets
\begin{equation} \label{plain}
\mc{E}^{\ell}_{N}\,=\,\bigcup\limits _{x\in S}\mc{E}_{N}^{\ell,x}
\;\;\;\;\text{and\;\;\;\;}
\Delta^{\ell}_{N}\,=\,\mc H_{N}\,\setminus\,
\Big(\,\bigcup_{x\in S}\mc{E}_{N}^{\ell,x}\,\Big)\;.
\end{equation}
For briefness, we omit the superscript $\ell$ when its meaning is evident from the context.
The sets $\mathcal{E}_N^x$, $x\in S$, represent the metastable states within the dynamics $\eta_N(\cdot)$, often referred to as metastable valleys.

To define metastability, we need the following condition from \cite{RES}.
For a finite set $S$ with a generator $L$, suppose a sequence of finite state spaces $(X_N)_{N\in \bb N}$ and a sequence of disjoint valleys $(\{\mc E_N^x : x\in S\})_{N\in \bb N}$ of $X_N$ are given.
For a sequence of generators $(L_N)_{N\in \bb N}$ of Markov processes on $X_N$ for each $N$, we introduce the following condition.
\smallskip

\noindent \textbf{Condition $\mf{R}_{L}$}. For all $\lambda>0$ and
$g\colon S\to\bb{R}$, the unique solution $F_{N}:X_N\rightarrow \bb R$ of the resolvent
equation \eqref{1f01} is asymptotically constant in each set
$\mc{E}_{N}^{x}$:
\begin{equation}
\label{1f01}
(\, \lambda \,-\, L_N\,) \, F_N \;=\; G_N\; \coloneqq \;\sum_{x\in S} g(x)\,\chi_{\mc{E}_{N}^{x}}\;,
\end{equation}
\begin{equation*}
\lim_{N\to\infty}
\sup_{\eta\in\mc{E}_{N}^{x}}\,\big|\,F_{N}(\eta)\,-\,f(x)\,\big|
\,\;=\;0\;, \quad x\in S\;,
\end{equation*}
where $f:S\rightarrow\bb{R}$ is the unique solution of the reduced
resolvent equation
\begin{equation*}
(\lambda - L) \, f \,=\, g \;.
\end{equation*}

Now, we begin to describe the properties and metastable behaviors of the zero-range process.
For the zero-range process, the measure concentration property is presented in \cite[Theorem 3.1]{SZRP}.
Recall that $\rho_N$ is the stationary measure of the zero-range process.

\begin{thm}
    \label{meas}
    For any $\ell_N \rightarrow \infty$ with $\ell_N\leq \frac{N}{2}$,
    $$\lim_{N\rightarrow \infty}\rho_N(\Delta_N) = 0,\;\; \lim_{N\rightarrow \infty}\rho_N(\mc E^x_N) = \frac{1}{\kappa}$$ for all $x\in S$.
\end{thm}

The metastable behavior of the zero-range process occurs at the time scale $\theta_N = N^{1+\alpha}$.
To address condition $\mf{R}_{L}$, it is necessary to define a limiting Markov process, which is denoted by $L$ in condition $\mf R_{L}$.
This concept of the limiting process was first introduced in \cite{BL09}. To describe it, we begin with a notion of capacity.

\begin{defn}
    Let $S$ be a finite set and $L$ be a generator of an irreducible Markov process on $S$.
    Suppose that $L$ is given by a rate $r:S\times S\rightarrow [0,\infty)$.
    Let $m$ be the stationary probability of the process generated by $L$.
    The Dirichlet form $D$ associated with $L$ is defined by
    \begin{equation*}
        D(f,g) \coloneq \frac{1}{2}\sum_{x,y\in S} m(x)r(x,y)(f(y)-f(x))(g(y)-g(x)),
    \end{equation*}
    For any disjoint subsets $A,B\subset S$, the capacity $\Cap(A,B)$ is defined by
    \begin{equation*}
        \Cap_S(A,B) \coloneq D(f,f),
    \end{equation*}
    where $f$ is the unique solution of the Dirichlet problem
    \begin{equation*}
        \begin{cases}
            Lf = 0, \; x\in S\setminus(A\cup B), \\
            f = 1, \; x\in B, \\
            f = 0, \; x\in A.
        \end{cases}
    \end{equation*}
\end{defn}

\begin{defn}
    Define constants $\Gamma(\alpha)$ and $I_\alpha$ by
    $$ \Gamma(\alpha) \coloneq 1 + \sum_{n=1}^\infty \frac{1}{n^{\alpha+1}}, 
    \quad I_\alpha \coloneq \int_{0}^{1} u^\alpha (1-u)^\alpha du.$$
    Define a generator of Markov process $\bb L$ on $S$ given by
    \begin{equation}
        \label{limiting}
        (\bb L f)(x) = \frac{\kappa}{\Gamma(\alpha)I_\alpha} \sum_{y\in S} \Cap_S(x,y)(f(y)-f(x)), \; x\in S,
    \end{equation}
    where $\Cap_S(x,y)$ is a capacity between $\{x\}$ and $\{y\}$ in the Markov process generated by $\mc L_S$.
\end{defn}

Using the limiting Markov process $\bb L$, we may consider condition $\mf R_{\bb L}$. The following theorem teaches us that condition $\mf R_{\bb L}$ holds
for a certain $(\ell_N)_{N\in \bb N}$. The proof is presented in Section \ref{prelimmeta}.

\begin{thm}
    \label{mszrp}
    For $\ell_N = \lfloor N^{\frac{1}{2(\kappa-1)}} \rfloor$, 
    condition $\mf R_{\bb L}$ holds for a generator $N^{1+\alpha} \ms L_N$ and $(\mc E^{x}_N)_{x\in S}$.
\end{thm}

\subsection{\texorpdfstring{$\Gamma$}{Gamma}-convergence on metastable time scale}

For $x\in S$, define $\xi^x \in \Xi$ such that $(\xi^x)_y = \delta_{x,y}$ for $y\in S$. Thus $\xi^x$ represents a configuration with a condensation at $x$.
This notion leads to an embedding $\iota_S : S\rightarrow \Xi$ of $S$ into $\Xi$, defined by
\begin{equation} \label{iotas}
    \iota_S(x) = \xi^x \text{ for } x\in S.
\end{equation}
Define a functional $\mc J$ on $\ms P(\Xi)$ by

\begin{equation}
    \label{rate2}
    \mc J(\mu) \coloneq \begin{cases}
        \sup\limits_{u:S \rightarrow \bb R_{>0}}\int_{S} -\frac{\bb L u}{u} d\nu, \;\; \mu = \nu \iota_S^{-1} & \text{if } \sum\limits_{x\in S} \mu(\xi^x)=1, \\
        +\infty & \text{otherwise}.
    \end{cases}
\end{equation}

\begin{thm}
    \label{27}
    A sequence of rate functions $N^{1+\alpha} \mc I_N$ $\Gamma$-converges to $\mc J$.
\end{thm}

According to \eqref{ginf} and \eqref{gsup}, it is noted that the $\limsup$ part of the previous theorem
requires arguments for existence, whereas the $\liminf$ part necessitates conditions that are universally applicable, irrespective of the choice of converging sequences.
Because of this difference, the $\limsup$ analysis requires only a single representation of metastable valleys, while 
the $\liminf$ analysis demands a further argument on converging sequences of measures. The detail is presented in Section \ref{metatimescale}.

\subsection{Limiting diffusion of zero-range process at diffusive time scale} \label{limitingdiff}

In this section, referencing results from \cite{MZRP}, we are going to define a diffusion representing
limiting behavior of the zero-range process in the diffusive time scale $N^2$.
The constant $b$ in \cite[display (2.1)]{MZRP} is set as $\alpha$.
To introduce the generator of this diffusion, the domain must first be defined. We denote by $C^n(\Xi), n\geq 1$ the set of functions 
$F:\Xi \rightarrow \bb R$ which are $n$-times continuously differentiable in the
interior of $\Xi$ and have continuous derivatives up to order $n$ on the boundary of $\Xi$.
For a function $F\in C^1(\Xi)$, $\partial_{x}F$ denotes the partial derivative with respect to the variable $x$,
and $\nabla F$ stands for the gradient of $F$.

Let $\{\mathbf{v}_x\in \bb R^S : x\in S\}$ be the vectors defined by
\begin{equation} \label{vx}
    \mathbf{v}_x \coloneqq \sum_{y\in S} r(x,y) (\textbf{e}_y-\textbf{e}_x), \quad x\in S,
\end{equation}
where $r$ is a jump rate of $\mc L_S$ and
$\{\textbf{e}_x:x\in S\}$ represents the canonical basis of $\bb R^S$. Define the vector field $\mathbf{b}:\Xi \rightarrow \bb R^S$ by
$$\mathbf{b}(\xi)\coloneq \alpha \sum_{x\in S}\mathbf{1}\{\xi_x\neq 0\}\frac{m_x}{\xi_x}\mathbf{v}_x,$$
where $m_x$ is a constant satisfying
\begin{enumerate}
    \item $\lim\limits_{n\rightarrow \infty} n(\frac{g(n)}{m_x}-1) = \alpha \text{ for all } x\in S \label{mcond}$,
    \item $m_x$ is a stationary measure of the Markov process generated by $\mc L_S$,
\end{enumerate}
from \cite[display (2.1)]{MZRP}.
Setting $m_x = 1$, following from \eqref{uniformmeasurecond}, satisfies the above conditions.

Actually, if the stationary probability is not uniform, an additional time scale $N$ appears as the scale for a fluid limit: see \cite{Fluid}. Additionally, it is expected that
$N^2$ still represents the diffusive time scale, although the results in \cite{MZRP} have yet to be generalized to this case.
To avoid this complicated situation, we assume the uniform measure condition \eqref{uniformmeasurecond}.

\begin{defn} \label{def31}
    For each $x\in S$, let $\mc D^S_x$, be the space of functions $H$ in $C^2(\Xi)$ for which the map 
    $\xi \mapsto [\mathbf{v}_x\cdot\nabla H(\xi)]/\xi_x\mathbf{1}\{\xi_x>0\}$ is continuous on $\Xi$,
    and let $\mc D^S_A \coloneq \cap_{x\in A} \mc D^S_x$, $\varnothing\subsetneq A\subset S.$
\end{defn}

Let us now introduce the generator for the limiting diffusion. Let $\mf L : C^2(\Xi)\rightarrow \bb R$ be the second order differential operator given by
\begin{equation}
    \label{lim diff}
    (\mf L F)(\xi) \coloneq \mathbf{b}(\xi)\cdot \nabla F(\xi) + \frac{1}{2}\sum_{x,y\in S} r(x,y) (\partial_x - \partial_y)^2F(\xi).
\end{equation}

Now, we characterize the limiting diffusion as the solution of the martingale problem corresponding to $(\mf L, \mc D^S_S)$.
Denote by $C(\bb R_+, \Xi)$ the space of continuous trajectories with the topology of uniform convergence on bounded intervals.
For $t\geq 0$, we denote by $\xi_t:C(\bb R_+,\Xi)\rightarrow \Xi$ the process of coordinate maps on time $t$ and by $(\ms F_t)_{t\geq 0}$ the generated filtration 
$\ms F_t \coloneq \sigma(\xi_s:s\leq t)$.
We say a probability measure $\bb P$ on $C(\bb R_+,\Xi)$ is a solution for the $\mf L$-martingale problem if, for any $H\in \mc D^S_S$, 
\begin{equation}
    \label{42}
    H(\xi_t) - \int_{0}^{t} (\mf L H)(\xi_s)ds, \quad t\geq 0
\end{equation}
is a $\bb P$-martingale with respect to the filtration $(\ms F_t)_{t\geq 0}$.
Then we have a unique solution to the martingale problem.
\begin{thm} \cite[Theorem 2.2]{MZRP} \label{zrplaw}
    For each $\xi\in \Xi$, there exists a unique probability measure on $C(\bb R_+, \Xi)$, denoted by
    $\bb P_{\xi}$, which starts at $\xi$ and is a solution of the $\mf L$-martinagale problem.
\end{thm}
Furthermore, the process is Feller. Its proof is presented in Section \ref{prelimpremeta}.

\begin{prop} \label{feller}
    The Markov process coming from the solution $(\bb P_{\xi})_{\xi\in \Xi}$ of the $\mf L$-martingale problem is Feller.
\end{prop}

\subsection{Level two large deviation rate function and diffusive time scale}

Since the process is Feller, we may consider its domain.
Let $\mc D^S$ be a domain of $\mf L$ in $C(\Xi)$ and $\mc D^{S}_+$ be a set of positive functions in $\mc D^S$.
We have solved the martingale problem over a set $\mc D^S_S$, so it ensures that $\mc D^S_S$ is a subset of $\mc D^S$.
However, they are not necessarily equal, and actually $\mc D^S$ is larger.
The level-two large deviation function of the limiting diffusion is defined as
\begin{equation} \label{ratedef}
    \mc K(\mu) \coloneq \sup_{u\in \mc D^S_+} \int_{\Xi} -\frac{\mf L u}{u}d\mu.
\end{equation}

From now, we assume the reversibility of the underlying Markov process generated by $\mc L_S$,
\begin{equation}
    \label{reversibility}
    r(x,y)=r(y,x) \text{ for all } x, y\in S.
\end{equation}
This assumption is necessary to calculate the large deviation rate function $\mc K$ explicitly.
There may exist a method to demonstrate a $\Gamma$-convergence of the sequence of rate functions without calculating $\mc K$ explicitly,
but to the author's knowledge, no such method is available. Therefore, 
we prove some inequalities for the rate functions $\mc K$ via explicit calculation using the reversibility assumption
and then construct a sequence of measures to show the convergence.
This is a brief sketch of the proof of the following theorem.
\begin{thm}
    \label{46}
    A sequence of rate functions $N^2 \mc I_N$ $\Gamma$-converges to $\mc K$.
\end{thm}
Note that the proof of the $\liminf$ part of Theorem \ref{46}, as described in Section \ref{slowgammainf}, does not require the reversibility assumption.
 Therefore, reversibility is necessary for the proof of the $\limsup$ part of the $\Gamma$-convergence.

Let us embed the set $S$ into $\Xi$ by $\iota_S$ as in \eqref{iotas} so treat $S$ as a subset of $\Xi$.
From Lemma \ref{zerocond}, we have
$$ \mc K(\mu) = 0 \text{ if and only if } \mu(S) = 1 \text{ for all } x\in S.$$
Therefore, from the point of Definition \ref{def1}, it seems like the time scales $N^2$ and $N^{1+\alpha}$ are the only time scales 
that appear in the full $\Gamma$-expansion of the level-two large deviation rate functions.
To ensure this, we need to prove that there exists no other time scale besides them.

\subsection{Other time scales}

Let us treat $S$ as a subset of $\Xi$ via $\iota_S$ as in \eqref{iotas}.
Define $X_S : \ms P(\Xi) \rightarrow [0,+\infty]$ by
\begin{equation}\label{midconv}
    X_S(\mu) = \begin{cases}
    0 & \text{if } \mu(S) = 1, \\
    +\infty & \text{otherwise}.
\end{cases}
\end{equation}
Also, define $U_S : \ms P(\Xi) \rightarrow [0,+\infty]$ by
\begin{equation}\label{postconv}
    U_S(\mu) = \begin{cases}
    0 & \text{if } \mu(\xi^x) = \frac{1}{\kappa} \text{ for all } x\in S, \\
    +\infty & \text{otherwise}.
\end{cases}
\end{equation}
These functionals are used to represent the $\Gamma$-limit of the rate functions in the time scales which are neither $N^2$ nor $N^{1+\alpha}$.

\begin{prop} \label{otherscales}
    Fix any positive sequence $\theta_N$. Then the following holds.
    \begin{enumerate}
        \item If $\theta_N \prec N^2$, we have $(\Gamma\mydash\lim)_{N\rightarrow \infty}\theta_N \mc I_N = 0$.
        \item If $N^2 \prec \theta_N \prec N^{1+\alpha}$, we have $(\Gamma\mydash\lim)_{N\rightarrow \infty}\theta_N \mc I_N = X_S$.
        \item If $N^{1+\alpha} \prec \theta_N$, we have $(\Gamma\mydash\lim)_{N\rightarrow \infty}\theta_N \mc I_N = U_S$.
    \end{enumerate}
\end{prop}

Remark that any time scale except $N^2$ and $N^{1+\alpha}$ does not appear in the full $\Gamma$-expansion of level-two large deviation rate function $\mc I_N$ in the sense of
Definition \ref{def1}.
Therefore, we have:
\begin{thm}
    The full $\Gamma$-expansion of $\mc I_N$ in \eqref{expansion} is given as
    $$\mc I_N = \frac{1}{N^2}\mc K + \frac{1}{N^{1+\alpha}}\mc J.$$
\end{thm}

In the next section, we present a general framework for $\Gamma$-convergence of rate functions associated with stochastic dynamical systems with metastable behavior, as demonstrated throughout Sections \ref{gen_metastable_proof}--\ref{slowlimsup}.
Following those sections, we apply the theorems presented in Section \ref{generalframework} to the zero-range process and prove the main results in this section.

\section{Main Results for \texorpdfstring{$\Gamma$}{Gamma}-Convergence of Level-Two Large deviations} \label{generalframework}

Throughout this section, we present a general framework for $\Gamma$-convergence of rate functions associated with stochastic dynamical systems with metastable behavior.
We discuss conditions necessary for $\Gamma$-convergence at both a metastable time scale and a diffusive time scale.
We assume that the generator of the Markov process, $A_N$, indexed by $N$, is already accelerated by the appropriate time scale, either $\theta_N$ or $\sigma_N$ from Section \ref{section1}. Consequently,
$A_N$ is expected to exhibit metastable or pre-metastable behavior without the need for multiplying by some additional time scale.
Therefore, the conditions we present for each time scale inherently contain information about whether the generator
possesses (pre-)metastable behavior at that time scale.

A key focus of this section is the $\Gamma$-convergence at the diffusive time scale,
as discussed in Section \ref{subsec33}. In stochastic systems exhibiting metastable behavior,
nucleation into metastable valleys occurs at the diffusive time scale.
This nucleation process is not straightforward; it involves transient states where the system temporarily resides before gradually being absorbed into the valleys.
This phenomenon manifests as the absorption property of the hydrodynamic limit diffusion.
The nucleation process is characterized by a hierarchical structure of transient state spaces, leading to a non-trivial structure in the level-two large deviation of the diffusion.
Calculating the large deviation is challenging due to this complexity. Section \ref{subsec33} demonstrates that, under certain assumptions, it is possible to precisely calculate the level-two large deviation for diffusion, thereby establishing $\Gamma$-convergence.
We anticipate that these results are applicable not only to zero-range processes but also to other systems exhibiting similar absorption properties at the diffusive time scale, such as symmetric inclusion processes.

\subsection{Settings} \label{settings}
Let $(X,d)$ be a compact metric space that represents the macroscopic states of the system.
Let $X_N$ be a finite discretization of $X$ with an immersion $\iota_N:X_N\rightarrow X$.
Let $A_N$ be a generator of a given irreducible Markov process on $X_N$, which is expected to have a (pre-)metastable behavior as $N\rightarrow \infty$.
Let $s_N$ be a stationary measure of $A_N$ on $X_N$.
Let $D(\bb R_+,X)$ be the space of c\`adl\`ag functions from $\bb R_+$ to $X$.
For $\xi_N\in X_N$, let $\bb P^N_{\xi_N}$ be a probability measure on $D(\bb R_+, X_N)$ which is a solution of the martingale problem associated with the generator $A_N$ starting at $\xi_N$.
Let $I_N$ be a rate function on $\ms P(X_N)$ defined by
\begin{equation} \label{rateNI}
    I_N(\nu) = \sup_{u:X_N\rightarrow \bb R_{>0}}\int_{X_N} -\frac{A_N u}{u}d\nu.
\end{equation}
Finally, let $\ms I_N$ be a rate function on $\ms P(X)$ defined by a pushforward of $I_N$ under $\iota_N$:
\begin{equation} \label{rateNII}
    \ms I_N(\mu) = \sup_{\nu\in \ms P(X_N), \mu = \nu\iota_N^{-1}}I_N(\nu).
\end{equation}

\subsection{\texorpdfstring{$\Gamma$}{Gamma}-convergence at a metastable time scale} \label{subsec32}

Recall the notion of metastability from Section \ref{metastability24}.
A set $S$ in the following definition represents the metastable valleys.

\begin{definition} \label{metastable valley}
    A finite set $S$ and a generator of Markov process $A$ over $S$ are given.
    Given $(X,d)$ and $\iota_N: X_N \rightarrow X$, $\{\mc E_N^x\subset X_N, x\in S\}$ are called metastable valleys with limit $A$ if
    \begin{enumerate}
        \item For all $x\neq y \in S$, $\mc E_N^x\cap \mc E_N^y = \varnothing$.
        \item The valley $\mc E_N^x$ and $A$ satisfy condition $\mf R_{A}$, which is defined in Section \ref{metastability24}.
    \end{enumerate}
\end{definition}

For sets $U,V\subset X$, let $m_d(U,V)$ be the maximal distance between $U$ and $V$:
$$m_d(U,V) = \sup_{\zeta\in U, \xi\in V}d(\zeta,\xi).$$
Suppose an immersion $\iota_S: S\rightarrow X$ is given. The two conditions for $\Gamma$-convergence are stated as follows.
\smallskip

\noindent \textbf{Condition (M0)}
There exists a metastable valleys $\{\mc E_N^x:x\in S\}$ with limit $A$ such that $m_d(\iota_N(\mc E_N^x),\iota_S(x))\rightarrow 0$ as $N\rightarrow \infty$ for all $x\in S$.

\smallskip
\noindent \textbf{Condition (M0*)}
Suppose condition \textbf{(M0)} is in force. Fix the metastable valley from condition \textbf{(M0)}.
For all strictly increasing sequence of positive integers $(N_k)_{k\in \bb N}$ and $\nu_{N_k}\in \ms P(X_{N_k})$, the inequality
$$\sup_{k\in \bb N} I_{N_k}(\nu_{N_k}) <\infty$$
implies $\limsup\limits_{k\rightarrow \infty} \nu_{N_k} (\Delta_{N_k}) = 0,$ where $\Delta_{N_k} = X_{N_k}\setminus \bigcup_{x\in S} \mc E_{N_k}^x$.

From the generator $A$ in condition \textbf{(M0)}, we define the rate function $I$ on $\ms P(S)$ by
\begin{equation} \label{ratefunctionI}
    I(\nu) = \sup_{u:S\rightarrow \bb R_{>0}}\int_{S} -\frac{A u}{u}d\nu.
\end{equation}
Using the rate function $I$, we define the rate function $\ms I$ on $\ms P(X)$ by a pushforward of $I$ under $\iota_S$:
\begin{equation} \label{ratefunctionII}
    \ms J(\mu) = \sup_{\nu\in \ms P(S):\mu = \nu\iota_S^{-1}}I(\nu).
\end{equation}
Remark that the conditions do not require the reversibility of the generators $A_N$ and $A$.
We are now ready to present the first main result of this section.

\begin{thm} \label{gen_metastable}
    Suppose that condition \textbf{\textup{(M0)}} is in force. Then the sequence of rate functions $\ms I_N$ and $\ms J$ satisfies the $\Gamma$-$\limsup$ condition.
    Furthermore, if condition \textbf{\textup{(M0*)}} is in force additionally, the sequence satisfies the $\Gamma$-$\liminf$ condition either,
    so the sequence $\ms I_N$ $\Gamma$-converges to $\ms J$.
\end{thm}

\begin{remark}
        Condition \textbf{\textup{(M0*)}} might seem to be too restrictive, but it can be obtained from the following condition: for any $\delta > 0$, we have
        \begin{equation}
        \sup_{\xi_N \in X_N} \bb P^N_{\xi_N}[\xi_{\delta} \in \Delta_N] = o_N(1), \label{quickabsorp}
        \end{equation}
        where $\xi_{\delta}$ is the process at time $\delta$.
        This fact can be proved by the same argument as in the proof of Proposition \ref{boundmeasure}.

        Intuitively, the condition \textbf{\textup{(M0*)}} contains a similar implication to \eqref{quickabsorp}, that is,
        the process quickly absorbs into the metastable valleys. This kind of condition is necessary to ensure the
        $\Gamma$-limit of the rate functions does not blow up.
\end{remark}

\subsection{\texorpdfstring{$\Gamma$}{Gamma}-convergence at a diffusive time scale} \label{subsec33}
Let $D(\bb R_+,X_N)$, $D(\bb R_+,X)$ be the space of c\`adl\`ag functions from $\bb R_+$ to $X_N$, $X$ respectively.
The map $\iota_N:X_N\rightarrow X$ induces a map $\iota^D_N: D(\bb R_+, X_N) \rightarrow D(\bb R_+, X)$ by composition.
Let $A$ be the generator of a Feller process on $X$ which is a candidate for a limit of this dynamical system.
Let $\bb P_\xi$ be a probability measure on $D(\bb R_+, X)$ which is the distribution over the paths of the process generated by $A$ starting at
$\xi\in X$.
Let $(P_t)_{t\geq 0}$ be a Feller semigroup generated by $A$. For simplicity, we abbreviate it as $(P_t)$.
Remark that for $f\in C(X)$, $$P_t f(\xi) = \bb P_\xi[f(\xi_t)]$$ holds for all $t\geq 0$ and $\xi\in X$.
Therefore, $P_t$ extends to a positive Borel function.
Let $\mc D$ be a domain of the Feller process $A$ in $C(X)$.
Let $\ms K$ be a rate function on $\ms P(X)$ defined by
\begin{equation} \label{limdifratefunc}
    \ms K(\mu) = \sup_{\substack{u:X\rightarrow \bb R_{>0},\\ u\in \mc D}}\int_{X} -\frac{A u}{u}d\mu.
\end{equation}
The following conditions are needed to ensure a $\Gamma$-convergence at a diffusive time scale.
\smallskip

\noindent \textbf{Condition (D0)}
For all sequence $(\xi_N\in X_N)_{N\in \bb N}$ with $\iota_N(\xi_N)$ converges to $\xi\in X$,
$\bb P^N_{\xi_N}\circ(\iota^D_N)^{-1}$ converges to $\bb P_\xi$ in the Skorohod topology.
\smallskip

\noindent \textbf{Condition (D1)}
The space $X$ is decomposed into a disjoint union of Borel sets $X = \cup_{\beta \in \mc C} \mathring{X}_\beta$ for
some finite partially ordered index set $(\mc C, \leq)$ equipped with a Radon measure $\lambda_\beta$ on $\mathring{X}_\beta$.
Let $X_\beta = \cup_{\gamma\leq \beta} \mathring{X}_\gamma$.
Let $\mc B(X)$, $\mc B(\mathring{X}_\beta)$ be sets of Borel functions on $X$, $\mathring{X}_\beta$ respectively.
\begin{enumerate}
    \item[\textbf{(D1.1)}] Suppose a probability measure $\mu\in \ms P(X)$ is supported on $X_\beta$. Then for all $t>0$,
    $\mu P_t$ is supported on $X_\beta$.
    \item[\textbf{(D1.2)}] For all $\xi\in X_\beta$ and $t>0$, $\delta_\xi P_t |_{\mathring{X}_\beta} \ll \lambda_\beta$, where $(\cdot)|_{\mathring{X}_\beta}$ is a restrction of a measure to $\mathring{X}_\beta$.
    \item[\textbf{(D1.3)}] For a function $f\in L^2(\lambda_\beta)\cap \mc B(\mathring{X}_\beta)$, let $\bar{f}\in \mc B(X)$ be an
    extension of $f$ to $X$ such that $\bar{f}|_{\mathring{X}_\beta} = f$ and $\bar{f}|_{\mathring{X}_\gamma} = 0$ for $\gamma\neq \beta$. 
    Then the transition kernel $P_t$ induces a bounded linear operator $P_t^{\beta}$ on $L^2(\lambda_\beta)$ by defining
    \begin{equation*}
        P_t^{\beta}f = (P_t \bar{f})|_{\mathring{X}_\beta}.
    \end{equation*}
    \item[\textbf{(D1.4)}] For all $\beta \in \mc C$, the operator $(P_t^{\beta})_{t\geq 0}$ is a self-adjoint strongly continuous contraction semigroup on $L^2(\lambda_\beta)$.
\end{enumerate}
\smallskip

\noindent \textbf{Condition (D1*)} Suppose condition \textbf{(D1)} is in force.
Let $L^\beta$ be the generator of the semigroup $P_t^{\beta}$ on $L^2(\lambda_\beta)$. Since $P_t^{\beta}$ is a self-adjoint contraction, $L^\beta$ is self-adjoint and non-positive.
Thus we may define the operator $\sqrt{-L^\beta}$ using the spectral theorem.
Let $\mc D_\beta$ be the domain of $\sqrt{-L^\beta}$ in $L^2(\lambda_\beta)$.
Define the energy functional $Q^\beta$ on $\mc D_\beta$ by
\begin{equation*}
    Q^\beta(f) = \int_{\mathring{X}_\beta} \left(\sqrt{-L^\beta} f \right)^2 d\lambda_\beta = \int_{\mathring{X}_\beta} -f (L^\beta f) d\lambda_\beta.
\end{equation*}
Then the following holds:
\begin{enumerate}[leftmargin=1.5cm]
    \item[\textbf{(D1*.1)}] For all $\beta \in \mc C$, there exists a subset $\mc D_{\beta,0}\subset \mc D_\beta$ such that $\mc D_{\beta,0}$ is dense in $\mc D_\beta$
    in the graph norm of $Q^\beta$, that is, for all $f\in \mc D_\beta$,
    there exists a sequence $(f_n)_{n\in \bb N}$ in $\mc D_{\beta,0}$ such that
    $f_n\rightarrow f$ in $L^2(\lambda_\beta)$ and $Q^\beta(f_n)\rightarrow Q^\beta(f)$.
    \item[\textbf{(D1*.2)}] For all $(f^\beta \in \mc D_{\beta,0})_{\beta\in \mc C}$, there exist sequences of functions $f^\beta_N : X_N\rightarrow \bb R$ for each $\beta\in \mc C$
    such that for all $\beta \neq \gamma \in \mc C$,
    \begin{align}
        ((f^\beta_N)^2 ds_N)\iota_N^{-1} \xrightarrow{\text{weakly}} (f^\beta)^2 d\lambda_\beta\; \text{ and }&\; \label{D1*21}
        ((f^\beta_N f^\gamma_N ds_N)\iota_N^{-1} \xrightarrow{\text{weakly}} 0, \\
        \lim_{N\rightarrow \infty} \int_{X_N} -f^\beta_N (A_N f^\beta_N) ds_N = Q^\beta(f^\beta),&\;\; \label{D1*22}
        \lim_{N\rightarrow \infty} \int_{X_N} -f^\gamma_N (A_N f^\beta_N) ds_N = 0.
    \end{align}
\end{enumerate}

\begin{thm} \label{gen_diffusive}
    Suppose that condition \textbf{\textup{(D0)}} is in force. Then the sequence of rate functions $\ms I_N$ and $\ms K$ satisfies the $\Gamma$-$\liminf$ condition.
    Assume that condition \textbf{\textup{(D1)}}, \textbf{\textup{(D1*)}} is in force additionally. For $\mu$ on $\ms P(X)$, decompose it as $$\mu = \sum_{\beta\in\mc C} \mu(\mathring{X}_\beta) \mu(\cdot|\mathring{X}_\beta).$$ Then
    \begin{equation}
        \label{11111better}
        \ms K(\mu) = \sum_{\beta\in\mc C} \mu(\mathring{X}_\beta) Q^\beta\left(\sqrt{\frac{d\mu(\cdot|\mathring{X}_\beta)}{d\lambda_\beta}}\right)
    \end{equation}
    holds. Furthermore, the sequence $\ms I_N$ $\Gamma$-converges to $\ms K$.
\end{thm}

\begin{remark}
    To prove \eqref{11111better}, it is sufficient to consider condition \textbf{\textup{(D1)}} alone. However, for the sake of simplicity, this part has been omitted.

    Intuitively, condition \textbf{\textup{(D0)}} says that the process is well approximated by the discretization and the $\Gamma$-$\liminf$ directly follows from this condition.
    Therefore, conditions \textbf{\textup{(D1)}} and \textbf{\textup{(D1*)}} give more information about the structure of the rate function $\ms K$ and this gives the $\Gamma$-$\limsup$ result.
    For the sets $\mathring{X}_\beta$ in condition \textbf{\textup{(D1)}} represent the transient states as the process approaches absorption into the metastable valleys and condition \textbf{\textup{(D1.1)}} directly implies the absorption property of the process.
    Remaining conditions \textbf{\textup{(D1.2)}}–\textbf{\textup{(D1.4)}} are regularities needed to calculate the rate function $\ms K$.
    Also, conditions \textbf{\textup{(D1*.1)}}–\textbf{\textup{(D1*.2)}} are needed to establish an approximating argument in order to obtain the $\Gamma$-$\limsup$ result.

    Note that the generator $A_N$ is not required to be reversible in the conditions, but its reversibility is needed to satisfy condition \textbf{\textup{(D1.4)}}.
    For the zero-range process, reversibility is essential to meet condition \textbf{\textup{(D1.4)}}.
    This is the main reason for assuming the reversibility condition \eqref{reversibility}.
\end{remark}

Through Sections \ref{gen_metastable_proof}--\ref{slowlimsup}, we prove the results in this section.
The remaining sections are devoted to verifying the conditions presented in this section for the zero-range process, in order to achieve the results discussed in Section \ref{section2}.

\section{Proof of Theorem \ref{gen_metastable}} \label{gen_metastable_proof}

We divide the proof of Theorem \ref{gen_metastable} into two parts: the $\Gamma$-$\limsup$ part and the $\Gamma$-$\liminf$ part.
Suppose the generator $A$ over $S$ is given by a jump rate $R:S\times S\rightarrow \bb R_{\geq 0}$.

\subsection{\texorpdfstring{$\Gamma$}{Gamma}-limsup part}

Fix a non-singular measure $\mu_S\in \ms P(S)$, that is,
\begin{equation*}
    \mu_S(x)>0 \text{ for all } x\in S.
\end{equation*}
Let $\mu$ be a measure on $X$ such that $\mu = \mu_S\iota_S^{-1}$.
By Lemma B.4 and B.5 in \cite{L22G}, it is enough to show the $\Gamma$-limsup condition for these types of $\mu\in\ms P(X)$.

From Corollary A.5 in \cite{L22G}, there exists a function $h:S\rightarrow R$ where $\mu_S$ is the stationary measure of tilted generator $\mf M_h$, given by
\begin{equation}
    \label{explicitform}
    [(\mf M_h A)f](x) = \sum_{y\in S} e^{-h(x)}R(x,y)e^{h(y)}[f(y)-f(x)] \; \text{ for } \; f:S\rightarrow \bb R.
\end{equation}
From \eqref{explicitform},
\begin{align*}
    \sum_{y\in S} e^{-h(x)}R(x,y)e^{h(y)}[f(y)-f(x)] = & e^{-h(x)}\sum_{y\in S}R(x,y)e^{h(y)}[f(y)-f(x)] \\
                                                   = & e^{-h(x)}(\sum_{y\in S}R(x,y)e^{h(y)}f(y)-f(x)\sum_{y\in S}R(x,y)e^{h(y)}) \\
                                                   = & e^{-h(x)}(A(e^h f)(x)-f(x)A(e^h)(x)).
\end{align*}
In short, we can write
\begin{equation*}
    (\mf M_h A)f = e^{-h}(A(e^h f) - fAe^h).
\end{equation*}
The direct consequence of this formula is the following lemma.

\begin{lemma}
    \label{tilded}
    Given a probability measure $\mu_S$ on $S$, for all $\phi:S\rightarrow R$, suppose we have an equality
    $$\sum_{x\in S}\left[-\frac{A \phi(x)}{e^h(x)}+\frac{\phi(x)}{e^{2h}(x)}Ae^h(x)\right]\mu_S(x) = 0$$
    for the generator $A$ of an irreducible Markov process. Then $\mu_S$ is the unique stationary probability measure for tilded generator $\mf M_h A$.
    Reversely, if $\mu_S$ is the stationary probability measure for tilded generator $\mf M_h A$, the equality holds for all $\phi: S \rightarrow R$.
\end{lemma}

Take a family of neighborhood $\{\mc E^x_N$, $x\in S\}$ from condition \textbf{(M0)}.
Fix $\lambda>0$ and solve a resolvent equation 
\begin{align}
    (\lambda - A) e^h &= k, \nonumber \\
    (\lambda - A_N) H_N &= K_N \coloneqq \sum_{x\in S} k(x)\,\chi_{\mc{E}_{N}^{x}}, \label{HNKN}
\end{align}
for each $N\in \bb N$. The constant $\lambda$ is chosen large enough in Proposition \ref{weakconvergence}.
Before that, the next lemma ensures the positivity of the solution.

\begin{lemma}\label{5.1}
    There exists $N_0\in \bb N$ such that $H_N$ is strictly positive for all $N\geq N_0$.
\end{lemma}

\begin{proof}
    Let $\mc E_N = \cup_{x\in S}\mc E_N^x$ and $\Delta_N = X_N\setminus \mc E_N$.
    From condition $\mf R_{A}$, we have that $H_N$ approximates $h$ on $\mc E_N$.
    So there exists $N_0\in \bb N$ so that for all $N\geq N_0$, $H_N$ is positive on $\mc E_N$.
    Then $H_N$ satisfies a resolvent equation
    \begin{align*}
    (\lambda-A_N) H_N &= 0 \;\;\;\;\; \text{ on } \; \Delta_N, \\
    H_N &= H_N \; \text{ on } \; \mc E_N.
    \end{align*}
    Let $(\eta^N_t)_{t\geq 0}$ be a process generated by $A_N$ on $X_N$ and $\tau_{\mc E_N}$ be a hitting time of the set $\mc E_N$. Then
    \begin{equation*}
        H_N(\eta) = \bb E_{\eta}[e^{-\lambda \tau_{\mc E_N}}H_N(\eta_{\tau_{\mc E_N}})] \text{ for all } \eta \in X_N.
    \end{equation*}
    This equation implies that $H_N$ is strictly positive for all $N\geq N_0$.
\end{proof}

From Lemma \ref{5.1}, we have $H_N >0$ for a large enough $N$. Therefore, we can take a stationary measure $\nu_N\in \ms P(X_N)$ of tilded generator $\mf M_{\log H_N}$.
Then define $\mu_N\in \ms P(X)$ as a pushforward measure of $\nu_N$, that is, $\mu_N \coloneq \nu_N\iota_N^{-1}$. Remark that $\mu_N$ is depending on $\lambda$.

\begin{prop}
    \label{weakconvergence}
    There exists a large enough positive $\lambda$ so that for all $x\in S$, we have
    $$\lim_{N \rightarrow \infty}\nu_N(\mc E_N^x) = \mu_S(x).$$
    Therefore a sequence $(\mu_N)$ weakly converges to $\mu$.
\end{prop}

\begin{proof}
    Recall a function $h: S\rightarrow R$ from \eqref{explicitform}.
    Fix a large enough $\lambda>0$ so that $(\lambda-2) e^h(x)> |A e^h(x)|$ holds for all $x\in S$. Now, solve two resolvent equations
    \begin{align*}
        (\lambda - A) e^h &= k, & (\lambda-1 - A) e^h &= k', \\
        (\lambda - A_N) H_N &= K_N \coloneqq \sum_{x\in S} k(x)\,\chi_{\mc{E}_{N}^{x}}, &
        (\lambda-1 - A_N) H_N' &= K_N' \coloneqq \sum_{x\in S} k'(x)\,\chi_{\mc{E}_{N}^{x}}. 
    \end{align*}    
    From Lemma \ref{tilded},
    \begin{align*}
        0 = \int_{X_N}\left(-\frac{A_N H'_N}{H_N}+\frac{H_N' A_N H_N}{H_N^2}\right)d\nu_N
          = \int_{X_N}\left(\frac{K'_N-(\lambda-1)H_N'}{H_N}+\frac{H_N' (\lambda H_N - K_N)}{H_N^2}\right)d\nu_N.
    \end{align*}
    So we have
    \begin{equation} \label{43eq1}
        \int_{X_N}-\frac{H_N'}{H_N}d\nu_N = \int_{X_N}\frac{H_N K'_N - H'_N K_N}{H_N^2}d\nu_N.
    \end{equation}
    Computing the right hand side of \eqref{43eq1}, we get
    \begin{align*}
        \int_{X_N}\frac{H_N K'_N - H'_N K_N}{H_N^2}d\nu_N &= \sum_{x\in S}\int_{\mc E_N^x}\frac{H_N K'_N - H'_N K_N}{H_N^2}d\nu_N \\
        &= \sum_{x\in S}\frac{e^h (\lambda- 1 - \ms L)e^h - e^h (\lambda - \ms L)e^h}{e^{2h}}(x)\nu_N(\mc E_N^x) + o_N(1) \\
        &= -\sum_{x\in S} \nu_N(\mc E_N^x) + o_N(1).
    \end{align*}
    On the other hand, computing the left-hand side of \eqref{43eq1} gives
    \begin{align*}
        \int_{X_N}-\frac{H_N'}{H_N}d\nu_N = \int_{\Delta_N}-\frac{H_N'}{H_N}d\nu_N - \sum_{x\in S} \nu_N(\mc E_N^x) + o_N(1).
    \end{align*}
    Therefore,
    \begin{equation}
    \label{5.3.1}
        \int_{\Delta_N}\frac{H_N'}{H_N}d\nu_N = o_N(1).
    \end{equation}
    
    From the condition $(\lambda-2) e^h(x)> |A e^h(x)|$, we have $2k'\geq k$. Therefore, for large enough $N$ which makes $H_N'\geq 0$,
    we have $$2(K_N'+H_N')\geq 2K_N' \geq K_N.$$
    Since $(\lambda - \ms A_N)^{-1}$ is a positive operator, we obtain
    \begin{align*}
        2H_N' = (\lambda-A_N)^{-1}(2K_N' + 2H_N') \geq (\lambda-A_N)^{-1}K_N = H_N.
    \end{align*}
    Therefore, from \eqref{5.3.1}, we get
    \begin{equation}
        \label{concentrated}
        \nu_N(\Delta_N) = o_N(1).
    \end{equation}
    
    Fix a function $\phi:S\rightarrow R$. Now, solve a resolvent equation
    \begin{align*}
        (\lambda - A) \phi &= \psi, \\
        (\lambda - A_N) \Phi_N &= \Psi_N \coloneqq \sum_{x\in S} \psi(x)\,\chi_{\mc{E}_{N}^{x}}. 
    \end{align*}
    From condition $\mf R_{A}$, we have that $\Phi_N$ approximates $\phi$ on $\mc E_N$.
    Then
    \begin{align}
    \label{5.3.2}
        0 = \int_{X_N}\left(-\frac{\ms L_N \Phi_N}{H_N}+\frac{\Phi_N A_N H_N}{H_N^2}\right)d\nu_N
          &= \sum_{x\in S} \int_{\mc E_N^x}(\frac{\Psi_N}{H_N}-\frac{K_N}{H_N^2})d\nu_N \nonumber \\
          &= \sum_{x\in S} \left(-\frac{A\phi}{e^h}+\frac{\phi}{e^{2h}}Ae^h\right)(x)\nu_N(\mc E_N^x) + o_N(1).
    \end{align}
    Suppose $\nu_N(\mc E_N^x)$ does not converges to $\mu_S(x)$ for some $x\in S$. Since $$\sum_{x\in S} \nu_N(\mc E_N^x) = 1 - \nu_N(\Delta_N) = 1+ o_N(1),$$
    we may take a limit point $(\tau(x))_{x\in S}\in \bb R^S$ of $\{(\nu_N(\mc E_N^x))_{x\in S}:N\in \bb N\}$.
    From \eqref{concentrated}, we have that $\tau$ is a probability measure on $S$.
    Then \eqref{5.3.2} implies an equality
    $$\sum_{x\in S} \left(-\frac{A\phi}{e^h}+\frac{\phi}{e^{2h}}Ae^h\right)(x)\tau(x)=0,$$
    which means $\tau$ is the stationary measure of tiled generator $\mf M_h A$.
    Since $\mu_S$ was the unique stationary probability measure of $\mf M_h A$, we have $\tau = \mu_S$, which leads to a contradiction.
\end{proof}

Recall the rate functions $I_N$, $\ms I_N$, $I$ and $\ms J$ from \eqref{rateNI}, \eqref{rateNII}, \eqref{ratefunctionI} and \eqref{ratefunctionII}, respectively.
\begin{prop}
    An inequality $$\limsup_{N\rightarrow \infty} \ms I_N(\mu_N) \leq \ms J (\mu)$$ holds.
\end{prop}

\begin{proof}
    Recall that $\mu_N = \nu_N\iota_N^{-1}$.
    From the definition of $\ms I_N$, we have
    $$ I_N(\nu_N) \geq \ms I_N(\mu_N).$$
    From \cite[Lemma A.2]{L22G}, we have
    $$I_N(\nu_N) = \int_{X_N}\frac{-A_N H_N}{H_N} d\nu_N.$$
    From the direct computation using \eqref{HNKN} and Proposition \ref{weakconvergence}, we get
    \begin{align*}
    \limsup_{N\rightarrow \infty} I_N(\nu_N)
    = -\lambda + \limsup_{N\rightarrow \infty} \int_{X_N} \frac{K_N}{H_N} d\nu_N 
    = \limsup_{N\rightarrow \infty} \sum_{x \in S} \int_{\mc E_N^x} \frac{ - A e^h(x)}{e^h(x)} d\nu_N = I(\mu_S) = \ms J(\mu).
    \end{align*}
\end{proof}

\subsection{$\Gamma$-liminf part}

\begin{prop}
    For any sequence $(\mu_N)$ in $\ms P(X)$ which converges to $\mu \in \ms P(X)$,
    \begin{equation} \label{fastliminf}
        \liminf_{N\rightarrow \infty} \ms I_N(\mu_N) \geq \ms J(\mu).
    \end{equation}
\end{prop}

\begin{proof}
    We first assume that $\ms J(\mu)< \infty$.
    Then, $\mu$ must be concentrated on $\iota_S(S)\subset X$.
    Suppose \eqref{fastliminf} does not hold.
    Then there exists a sequence $(\mu_{N_k})$ in $\ms P(X)$ which converges to $\mu \in \ms P(X)$ such that
    \begin{equation*}
        \ms I_{N_k}(\mu_{N_k})< \infty,\;\; \lim_{k\rightarrow \infty} \ms I_{N_k}(\mu_{N_k}) < \ms J(\mu).
    \end{equation*}
    Fix $\epsilon >0$ satisfying $2\epsilon< \ms J(\mu) - \lim_{k\rightarrow \infty} \ms I_{N_k}(\mu_{N_k})$. For large enough $k$, we may take a sequence $(\nu_{N_k} \in \ms P(X_{N_k}))$ so that
    \begin{equation} 
        \label{kakakaka}
        I_{N_k}(\nu_{N_k}) < \ms J(\mu) - \epsilon, \;\; \mu_{N_k} = \nu_{N_k}\iota_{N_k}^{-1}.
    \end{equation}
    From condition \textbf{(M0)}, we take shrinking valleys $(\mc E_N^x)_{x\in S}$ satisfying condition $\mf R_A$:
    Given a function $f: S\rightarrow R_+$, we have a solution of a resolvent equation
    \begin{align*}
        (\lambda - A) f &= k, \\
        (\lambda - A_{N_k}) F_{N_k} &= K_{N_k} \coloneqq \sum_{x\in S} k(x)\,\chi_{\mc{E}_{{N_k}}^{x}},
    \end{align*}
    satisfying
    $$\sup_{x\in S}\sup_{\xi\in \mc E_{N_k}^x}|F_{N_k}(\xi)-f(x)|=o_k(1).$$
    From condition \textbf{(M0*)} and \eqref{kakakaka}, we have
    $$\lim_{k\rightarrow \infty} \nu_{N_k}(\Delta_{N_k}) = 0.$$

    By Lemma \ref{5.1}, $F_{N_k}$ is strictly positive for large enough $k$. Such $k$, we have
    \begin{equation} \label{eq931}
        I_{N_k}(\nu_{N_k}) \geq \int_{X_{N_k}}\frac{-A_{N_k} F_{N_k}}{F_{N_k}}d\nu_{N_k} = -\lambda + \int_{X_{N_k}}\frac{K_{N_k}}{F_{N_k}}d\nu_{N_k}
        = \sum_{x\in S} -\frac{A f(x)}{f(x)}\nu_N(\mc E_{N_k}^x) + o_{N_k}(1).
    \end{equation}
    From the fact that the valleys are shrinking, we have
    \begin{equation} \label{eq933}
        \liminf_{k\rightarrow \infty} \nu_{N_k}(\mc E_{N_k}^x) \leq \liminf_{k\rightarrow \infty} \mu_{N_k}(\iota_{N_k}(\mc E_{N_k}^x)) \leq \mu(\iota_S(x)).
    \end{equation}
    Adding up for all $x\in S$ gives
    \begin{equation} \label{eq937}
        \liminf_{k\rightarrow \infty} \nu_{N_k}(X_{N_k}\setminus \Delta_{N_k}) \leq 1.
    \end{equation}
    Since $\lim_{k\rightarrow \infty} \nu_{N_k}(\Delta_{N_k})= 0$, \eqref{eq937} and \eqref{eq933} are actually equalities.
    Taking $k\rightarrow \infty$ in \eqref{eq931} gives
    $$\liminf_{k\rightarrow \infty} I_{N_k}(\nu_{N_k}) \geq \sum_{x\in S} -\frac{A f(x)}{f(x)}\mu(\iota_S(x)).$$
    Finally, taking $\limsup$ on positive $f$ gives
    $$\liminf_{k\rightarrow \infty} I_{N_k}(\nu_{N_k}) \geq \ms J(\mu),$$
    which contradicts to \eqref{kakakaka}.

    For the case $\ms J(\mu) = \infty$, suppose that \eqref{fastliminf} does not hold.
    Then there exists a sequence $(\mu_{N_k})$ in $\ms P(X)$ which converges to $\mu \in \ms P(X)$ and satisfies
    $$\lim_{k\rightarrow \infty} \ms I_{N_k}(\mu_{N_k}) < \infty.$$
    Then we can take a sequence $(\nu_{N_k})$ in $\ms P(X_{N_k})$ so that
    $$I_{N_k}(\nu_{N_k}) < \ms I_{N_k}(\mu_{N_k}) + 1, \;\; \mu_{N_k} = \nu_{N_k}\iota_{N_k}^{-1}.$$
    From condition \textbf{(M0*)}, for valleys $(\mc E_N^x)_{x\in S}$, we have
    $$\lim_{k\rightarrow \infty} \nu_{N_k}(\Delta_{N_k}) = 0.$$
    Therefore, $\mu$ must be concentrated on $\iota_S(S)$, so $\ms J(\mu)<\infty$. This leads to a contradiction.
\end{proof}

\section{Proof of Theorem \ref{gen_diffusive}: $\Gamma$-$\liminf$ part} \label{slowgammainf}

In this section, we prove the half part of Theorem \ref{gen_diffusive}: $\Gamma$-liminf assuming condition \textbf{(D0)}.
We need an approximation result for the Feller processes.
The following result is from \cite[Theorem I.6.1]{EK}.

\begin{thm}
    \label{EKapprox}
    A Banach space $E$ is given. For $n\in \bb N$, let $E_n$ be a Banach space. Suppose a bounded linear transformation
    $\pi_n:E\rightarrow E_n$ is given for each $n$. Assume that $\sup_n\|\pi_n\| < \infty$. Write $f_n\rightarrow f$ if $f_n\in E_n$ for each $n$,
    $f\in E$, and $\lim_{n\rightarrow \infty}\|\pi_n f - f_n\| = 0$.
    For $n\in \bb N$, let $\{T_n(t)\}$ and $T(t)$ be strongly continuous contraction semigroups on $E_n$ and $E$ with generators
    $A_n$ and $A$ respectively. Let $D$ be a core for $A$. Then the following are equivalent.
    \begin{enumerate}
        \item[\textbf{\textup{(A1)}}] For all $f \in E$, $T_n(t)\pi_n f \rightarrow T(t)f$ for all $t\geq 0$, uniformly on bounded intervals.
        \item[\textbf{\textup{(A2)}}] For all $f \in E$, $T_n(t)\pi_n f \rightarrow T(t)f$ for all $t\geq 0$.
        \item[\textbf{\textup{(A3)}}] For all $f \in D$, there exists $f_n\in D(E_n)$ for each $n\geq 1$ such that $f_n\rightarrow f$ and $A_n f_n \rightarrow Af$.
    \end{enumerate}
\end{thm}

Take $E_N = C(X_N)$ and $E = C(X)$. Let $\pi_n$ be a restriction of a function on $X$ to $X_N$.
Let $A$, $A_N$ be generators in Section \ref{settings} and $A$ a generator of 
Let $P_N$, $P$ be semigroups generated by $A_N$, $A$ respectively.
Then we have the following result.

\begin{lemma}
    Condition \textup{\textbf{(A2)}} in the theorem is satisfied for $P_N$ and $P$.
\end{lemma}
\begin{proof}
    Fix $t>0$. Suppose that \textbf{(A2)} does not hold. Then there exists $\epsilon>0$ and a sequence $N_k\rightarrow \infty$ with
    $\xi_{N_k}\in X_{N_k}$,
    $$|P_{N_k, t} \pi_{N_k} f (\xi_{N_k}) - P_t f(\xi_{N_k})| \geq \epsilon.$$
    Since $X$ is compact, there exists a subsequence $\xi_{N_{k_l}}$ converges to $\xi\in X$ as $l\rightarrow \infty$.
    By condition \textbf{(D0)}, $\bb P_{\xi_{N_{k_l}}}^N$ converges to $\bb P_{\xi}$ in the Skorohod topology, which leads to a contradiction.
    Therefore, condition \textbf{(A2)} holds.
\end{proof}

This lemma implies that condition \textbf{(A3)} in the theorem is satisfied. We use this to prove the $\liminf$ part of Theorem \ref{gen_diffusive}.

\begin{prop} \label{slowliminf}
    For any $\mu$ on $\ms P(X)$ and any sequence of measures $\mu_N$ in $\ms P(X)$ which weakly converges to $\mu$, we have
    $$\liminf_{N\rightarrow \infty}\ms I_N(\mu_N) \geq \ms K(\mu).$$
\end{prop}

\begin{proof}
    Fix a probability measure $\mu\in \ms P(X)$.
    Take any sequence $\mu_N$ in $\ms P(X)$ which converges weakly to $\mu$.
    Suppose there exists a subsequence $\mu_{N_k}$ which satisfies
    $$\lim_{k\rightarrow \infty} \ms I_{N_k}(\mu_{N_k}) < \ms K(\mu).$$
    Fix $F\in \mc D^S_+$, which is a positive function in the domain of the $A$.
    From condition \textbf{(A3)} in Theorem \ref{EKapprox},
    there exists a sequence $F_N: X_N \rightarrow \bb R_+ $ which satisfies $F_N\rightarrow F$ and $A_N F_N \rightarrow AF$.
    Since $F_N \rightarrow F$, for large enough $N$, we have $$F_N > \frac{1}{2}\inf_{\xi\in \Xi} F(\xi).$$
    We choose
    $\nu_{N_k}\in \ms P(X_{N_k})$ so that $\mu_{N_k} = \nu_{N_k}\iota_{N_k}^{-1}$ with $$\ms I_{N_k}(\mu_{N_k}) +\frac{1}{N_k} \geq I_{N_k}(\nu_{N_k}).$$
    Then we have
    \begin{equation*}
        \ms I_{N_k}(\mu_{N_k}) \geq I_{N_k}(\nu_{N_k}) - \frac{1}{N_k} = \int_{X_{N_k}}\frac{-A_{N_k} F_{N_k}}{F_{N_k}}d\nu_{N_k} - \frac{1}{N_k} = \int_{X}\frac{-AF}{F}d\mu_{N_k} + o_{k}(1).
    \end{equation*}
    Since $\mu_{N_k}$ converges to $\mu$ weakly, taking $\liminf$ as $k \to \infty$ gives
    \begin{equation*}
        \liminf_{N\rightarrow \infty} \ms I_{N_k}(\mu_{N_k}) \geq \int_{X}\frac{-AF}{F}d\mu.
    \end{equation*}
    Taking $\limsup$ on $F\in \mc D^S_+$ gives a contradiction.
\end{proof}

\section{Level two large deviation Rate function of a Feller process} \label{sectionFeller}

This section provides several lemmas which are used to calculate the rate function of limiting diffusion to prove
$\limsup$ part of Theorem \ref{gen_diffusive}.

We introduce general lemmas for a stochastic process on a locally compact Polish space $X$.
These results are exactly from \cite{DV75a} but for the sake of completeness, we give the whole proof of it. Suppose we have a Feller continuous transition kernel $(P_t)$ with a generator $A$.
Let $\mc D$ be a domain of generator in $C(X)$ and $\mc D_+$ be a subset of $\mc D$ which consists of strictly positive functions.
Let $(P_t)$ act on $\ms U$ which is a collection of all bounded Borel functions on $X$.
Define $\ms U_0$ to be a collection of all positive Borel functions that are bounded away from zero, namely
$$\ms U_0 \coloneq \{f \in \ms U : f > c \text{ for some } c>0.\}.$$

\begin{defn} For $h>0$, define a rate function $\ms K_h$ as
    $$\ms K_h(\mu) \coloneq -\inf_{u \in \ms U_0} \int_X \log\left(\frac{P_h u}{u}\right)d\mu.$$
\end{defn}

Recall from \eqref{limdifratefunc} that
    $$\ms K(\mu) \coloneq -\inf_{u \in \mc D_+} \int_X \frac{A u}{u}d\mu.$$
\begin{lemma}
    \label{93}
Suppose for all probability measure $\mu$ on $X$ and $u\in \ms U$, there exists a sequence $(u_n)$ in $\mc D$ satisfying
$$ \inf_{\xi\in X} u(\xi) \leq \liminf_{n\rightarrow \infty} \inf_{\xi\in X} u_n(\xi) \leq \limsup_{n\rightarrow \infty} \sup_{\xi\in X} u_n(\xi) \leq \sup_{\xi\in X} u(\xi)$$
and $u_n\rightarrow u$ 
almost everywhere $\mu$ on $X$. Then for any $h>0$ and $\mu\in \ms P(X)$, we have
$$\ms K_h(\mu) \leq h \ms K(\mu).$$
\end{lemma}

\begin{proof}
    For a function $u\in \ms U_0$, we need to prove that for all $u\in \ms U_0$,
    \begin{equation}
        \label{91}
        \int_{X} \log P_h u d\mu - \int_{X} \log u d\mu \geq -h \ms K(\mu).
    \end{equation}
    We first show that $u\in \mc D_+$ satisfies \eqref{91}.
    Fix $u\ \in \mc D_+$, 
    $$\frac{d}{dh}\int_{X} \log P_h u d\mu = \int_{X} \frac{A P_h u}{P_h u} d\mu \geq -\ms K(\mu).$$
    Therefore, we have \eqref{91} for $u\in \mc D_+$.

    Now, to prove the lemma, we need to obtain \eqref{91} for $u\in \ms U_0$.
    Fix $u\in \ms U_0$ and take a probability measure $\nu = (\mu + \mu P_h)/2$.
    From the assumption, we have a sequence $u_n \rightarrow u$.
    Let
    $$E = \{\xi\in X : \lim u_n(\xi) \nrightarrow u(\xi)\}$$
    so we have $\mu(E) = 0$, $P_h\mu(E)=0$.
    From the given assumption in the lemma, for sufficiently large $n$, $u_n$ uniformly bounded above and below away from 0. Therefore, using the dominated convergence theorem, we finally obtain
    \begin{gather*}
        \int_X \log P_h u_n d\mu = \int_X \log u_n d\mu P_h \xrightarrow{n\rightarrow \infty} \int_X \log u d\mu P_h = \int_X \log P_h u d\mu, \\
        \int_X \log u_n d\mu \xrightarrow{n\rightarrow \infty} \int_X \log u d\mu,
    \end{gather*}
    which imply \eqref{91}.
\end{proof}

\begin{lemma}
    For all probability measure $\mu$ on $X$ and $u\in \ms U$, there exists a sequence $(u_n)$ in $\mc D$ satisfying
    $$ \inf_{\xi\in X} u(\xi) \leq \liminf_{n\rightarrow \infty} \inf_{\xi\in X} u_n(\xi) \leq \limsup_{n\rightarrow \infty} \sup_{\xi\in X} u_n(\xi) \leq \sup_{\xi\in X} u(\xi)$$
    and $u_n\rightarrow u$ almost everywhere $\mu$ on $X$.
\end{lemma}

\begin{proof}
    Since the process is Feller, $\mc D$ is dense in $C(X)$ with the sup-norm topology.
    Applying Lusin's theorem for locally compact Polish space and Tietze extension theorem, we can find a sequence $(u_n)$ in $C(X)$ such that
    $$\inf_{\xi\in X} u(\xi) \leq \inf_{\xi\in X} u_n(\xi) \leq \sup_{\xi\in X} u_n(\xi) \leq \sup_{\xi\in X} u(\xi), \;\; \mu(\{u_n\neq u\})\leq \frac{1}{2^n}.$$
    for all $n$. Then, we have $u_n\rightarrow u$ almost everywhere $\mu$ on $X$.
    Take $v_n\in \mc D$ such that $|v_n-u_n|_{\infty}<\frac{1}{n}$.
    Then we get $v_n\rightarrow u$ with the desired condition.
\end{proof}

\section{Proof of Theorem \ref{gen_diffusive}: $\Gamma$-$\limsup$ part} \label{slowlimsup}

In this section, we prove \eqref{11111better} and the $\limsup$ part of Theorem \ref{gen_diffusive}, assuming conditions \textbf{(D0)}, \textbf{(D1)}, and \textbf{(D1*)}.
Our two objectives are to calculate the rate function of the limiting diffusion and to demonstrate the $\limsup$ part of Theorem \ref{gen_diffusive}.
These objectives are not achieved separately; rather, they are accomplished simultaneously. We begin with analyzing the rate function of the limiting diffusion.

\begin{lemma}
    \label{1102}
    For $\mu\in \ms P(X)$, decompose it as $$\mu = \sum_{\beta\in \mc C} \mu(\mathring{X}_\beta) \mu(\cdot|\mathring{X}_\beta) = \sum_{\beta\in \mc C} \mu|_{\mathring{X}_\beta}.$$ If $\ms K(\mu)<\infty$, then
    $$\mu|_{\mathring{X}_\beta} \ll \lambda_\beta,$$
    where $\lambda_\beta$ is the reference measure on $\mathring{X}_\beta$ in condition \textbf{(D1)}.
\end{lemma}

\begin{proof}
    Define
    \begin{equation*}
        \ms K_h(\mu) \coloneq -\inf_{u \in \ms U_0} \int_X \log\left(\frac{P_h u}{u}\right)d\mu.
    \end{equation*}
    where $\ms U_0$ is a set of positive Borel functions on $X$ that are bounded away from zero.
    According to Lemma \ref{93}, for all $h>0$ we have
    $$\ms K_h(\mu) \leq h \ms K(\mu).$$
    So we get
    \begin{equation} \label{93eq1}
        \ms K_h(\mu) = \sup_{u\in \ms U_0} \int_{X} -\log{\left(\frac{P_h u}{u}\right)}d\mu < \infty.
    \end{equation}
    We claim that for all $\beta \in \mc C$ with $\mu(\mathring{X}_\beta)>0$,
    $\mu(\cdot|\mathring{X}_\beta)$ is absolutely continuous with respect to $\lambda_\beta$.
    Fix such $\beta\subset \mc C$.
    For a Borel set $K$ contained in $\mathring{X}_\beta$ and $a>0$, consider a Borel function $$u_a = (a+1)1_{X \setminus X_\beta}+ 1_{X_\beta}+ a1_K$$ in $\ms U_0$.
    Putting it into the \eqref{93eq1}, we obtain
    \begin{equation*}
        \ms K_h(\mu)\geq \sum_{B\subset S} \int_{\mathring{X}_\beta} -\log\left(\frac{P_h u_a}{u_a}\right)d\mu|_{\mathring{\Xi}_B}.
    \end{equation*}
    For $\beta \neq \gamma \in \mc C$, we claim that
    $$\int_{\mathring{X}_\gamma}-\log\left(\frac{P_h u_a}{u_a}\right)d\mu|_{\mathring{X}_\gamma}\geq 0.$$
    For $\gamma \lneq \beta$, $$P_h u_a = u_a = 1 \text{ on } \mathring{X}_\gamma,$$
    so the claim holds.
    For $\gamma \nleq \beta$, $$u_a = 1+a \;\; \text{ on } \mathring{X}_\gamma.$$
    So $u_a$ achieves its maximal value on every point inside $\mathring{X}_\gamma$ for $\gamma \neq \beta$.
    Therefore, $P_h u_a \leq u_a$ on $\mathring{X}_\gamma$ and the claim holds.
    So we obtain
    $$\ms K_h(\mu)\geq \mu(\mathring{X}_\beta) \int_{\mathring{X}_\beta} -\log\left(\frac{P_h u_a}{u_a}\right)d\mu(\cdot|\mathring{X}_\beta) =
    \mu(\mathring{X}_\beta) \int_{\mathring{X}_\beta} -\log\left(\frac{P_h a1_K + 1}{1 + a1_K}\right)d\mu(\cdot|\mathring{X}_\beta).$$
    Since $\log(x)$ is a concave function, using Jensen's inequality, we have
    \begin{align*}
        \ms K_h(\mu) + \mu(\mathring{X}_\beta)\log\left(1+\frac{a \mu P_h(K)}{\mu(\mathring{X}_\beta)}\right) &\geq \ms K_h(\mu) + \mu(\mathring{X}_\beta)
        \int_{\mathring{X}_\beta} -\log(P_h a1_K + 1)d\mu(\cdot|\mathring{X}_\beta) \\
        &\geq \mu(K) \log(1+a).
    \end{align*}
    
    Therefore,
    $$\ms K_h(\mu) + a \mu P_h(K) \geq \mu(K) \log(1+a).$$
    Subtracting $\mu(K)$ and dividing by $a$, we obtain
    \begin{equation}
        \label{1131}
        \mu P_h(K) -\mu(K) \geq \frac{-\ms K_h(\mu) + \mu(K)(\log(1+a)-a)}{a} \geq -\frac{\ms K_h(\mu)+a-\log(1+a)}{a}.
    \end{equation}
    
    Now, suppose that $\mu|_{\mathring{X}_\beta}$ is not absolutely continuous with respect to the uniform measure $\lambda_\beta$.
    From the Radon--Nikodym theorem, there exists a borel set $K$ contained in $\mathring{X}_\beta$ such that $\mu(K)>0$ and $\lambda_\beta(K) = 0$.
    Let $b = \mu(K)$. Condition \textbf{(D1.2)} gives $\mu P_h(K) = 0$ for all $h>0$.
    Therefore, the inequality \eqref{1131} becomes
    $$ b \leq \frac{\ms K_h(\mu)+(a-\log(1+a))}{a}.$$
    Sending $h\rightarrow 0$ and $a\rightarrow 0$ gives a contradiction.
    Therefore, $\mu|_{\mathring{X}_\beta}$ is absolutely continuous with respect to $\lambda_\beta$.
\end{proof}

Recall the energy functionals $(Q^\beta, \beta\in\mc C)$ from condition \textbf{(D1*)}.

\begin{lemma}
    \label{11111}
    For $\mu$ on $\ms P(X)$, decompose it as $$\mu = \sum_{\beta\in\mc C} \mu(\mathring{X}_\beta) \mu(\cdot|\mathring{X}_\beta).$$ Then
    $$\ms K(\mu) \geq \sum_{\beta\in\mc C} \mu(\mathring{X}_\beta) Q^\beta\left(\sqrt{\frac{d\mu(\cdot|\mathring{X}_\beta)}{d\lambda_\beta}}\right).$$
\end{lemma}

\begin{proof}
    Recall the function $\ms K_h$ from the proof of the previous lemma. Then it satisfies
    $$\ms K_h(\mu) \leq h \ms K(\mu).$$
    Fix $h>0$ and $u\in \ms U_0$.
    Since $-\log(x+1) \geq -x$, we have
    \begin{align*}
        \ms K_h(\mu) = \sup_{u\in \ms U_0} \int_{X} -\log{\left(\frac{P_h u}{u}\right)}d\mu \geq \sup_{u\in \ms U_0} \int_{X} -\frac{P_h u - u}{u}d\mu.
    \end{align*}
    For $\beta\in\mc C$ with $\mu(\mathring{X}_\beta) >0$, define a Borel function $u_\beta$ on $X$ as
    $$u_\beta(X \setminus \mathring{X}_\beta) = 0,\;\; u_\beta|_{\mathring{X}_\beta} = \sqrt{\frac{d\mu(\cdot|\mathring{X}_\beta)}{d\lambda_\beta}}.$$
    For $\mu(\mathring{X}_\beta) = 0$, define $u_\beta = 0$.
    For $\beta\in\mc C$, consider a positive constant $c_\beta$ depending on $\epsilon$ and $n$ which will be determined later.
    Let $u_\beta^n \coloneq \min(u_\beta, n)$.
    Fix $\epsilon > 0$ and take
    $$u^n_\epsilon = \sum_{\beta\in\mc C} c_\beta u_\beta^n + \epsilon.$$
    Note that
    \begin{align*}
        \int_{X} -\frac{P_h u - u}{u}d\mu = \sum_{\beta\in\mc C} \mu(\mathring{X}_\beta) \int_{\mathring{X}_\beta} -\frac{P_h u - u}{u} u_\beta^2 d\lambda_\beta.
    \end{align*}
    For $ \beta\in\mc C$, we have
    \begin{equation}
        \label{111}
        \int_{\mathring{X}_\beta} -\frac{P_h u^n_\epsilon - u^n_\epsilon}{u^n_\epsilon} u_\beta^2 d\lambda_\beta =
        \int_{\mathring{X}_\beta}-\frac{P_h (c_\beta u_\beta^n) - c_\beta u_\beta^n}{c_\beta u_\beta^n + \epsilon}u_\beta^2 d\lambda_\beta + \sum_{\gamma \lneq \beta} \int_{\mathring{X}_\beta} -\frac{P_h c_\gamma u_\gamma^n}{c_\beta u_\beta^n + \epsilon} u_\beta^2 d\lambda_\beta.
    \end{equation}
    Note that we only consider $\gamma\leq \beta$ in the above equation since the absorbing property of the process.
    Define $T_\beta : \bb R_{\geq 0} \rightarrow \bb R_{\geq 0}$ as
    $$ T_\beta(\delta) \coloneq \int_{u_\beta \leq \delta} u_\beta^2 d\lambda_\beta.$$
    Note that
    $$\int_{\mathring{X}_\beta} u_\beta^2 d\lambda_\beta = 1 \quad \text{implies} \quad
    \lim_{\delta\rightarrow 0} T_\beta(\delta) = \lim_{\delta\rightarrow 0}\int_{u_\beta \leq \delta} u_\beta^2 d\lambda_\beta = 0.$$
    Observe that
    \begin{align*}
        \int_{\mathring{X}_\beta} \frac{P_h c_\gamma u_\gamma^n}{c_\beta u_\beta^n + \epsilon} u_\beta^2 d\lambda_\beta
        = \int_{u_\beta \leq \delta} \frac{P_h c_\gamma u_\gamma^n}{c_\beta u_\beta^n + \epsilon} u_\beta^2 d\lambda_\beta + \int_{u_\beta > \delta} \frac{P_h c_\gamma u_\gamma^n}{c_\beta u_\beta^n + \epsilon} u_\beta^2 d\lambda_\beta 
        \leq \frac{n c_\gamma}{\epsilon} T_\beta(\delta) + \frac{n c_\gamma}{c_\beta (\delta\wedge n) + \epsilon}.
    \end{align*}
    Now, we choose $(c_\beta)$ inductively, so $(c_\beta)$ are fully determined if $\epsilon, n$ is fixed.
    Before doing this, we first define a height function $|\cdot|: \mc C \rightarrow \bb R_{\geq 0}$ as
    the maximal length of a process from a minimal element of $\mc C$ to $\beta$.
    So if $\beta$ is a minimal element, $|\beta| = 1$. Note that this is well-defined since we assumed that $\mc C$ is finite.
    
    First, if $\mu(\mathring{X}_\beta)=0$, assign $c_\beta = 0$.
    For the other $\beta$, if $|\beta| = 1$, we choose $c_\beta=1$. For $|\beta| \geq 2$, before choosing a $c_\beta$, we first choose a $\delta_\beta$.
    Precisely, we choose a $\delta_\beta$ small enough such that
    $$\sup_{\gamma \lneq \beta} \frac{n c_\gamma}{\epsilon} T_\beta(\delta_\beta) \leq \epsilon.$$
    Then we select $c_\beta$ large enough such that
    $$\sup_{\gamma \lneq \beta} \frac{n c_\gamma}{c_\beta (\delta_\beta\wedge n) + \epsilon} \leq \epsilon, \;\; c_\beta \geq 1.$$
    Then we have
    $$\int_{\mathring{X}_\beta} \frac{P_h c_\gamma u_\gamma^n}{c_\beta u_\beta^n + \epsilon} u_\beta^2 d\lambda_\beta \leq 2\epsilon.$$
    Therefore, we get
    $$\sum_{\gamma \lneq \beta} \int_{\mathring{X}_\beta} -\frac{P_h c_\gamma u_\gamma^n}{c_\beta u_\beta^n + \epsilon} u_\beta^2 d\lambda_\beta \leq 2(2^{|\beta|}-1)\epsilon.$$
    For the second term of \eqref{111}, when $\mu(\mathring{X}_\beta) > 0$, we have
    \begin{equation*}
        \int_{\mathring{X}_\beta}-\frac{P_h (c_\beta u_\beta^n) - c_\beta u_\beta^n}{c_\beta u_\beta^n + \epsilon}u_\beta^2 d\lambda_\beta
        = \int_{\mathring{X}_\beta}-\frac{P_h u_\beta^n - u_\beta^n}{u_\beta^n + \frac{\epsilon}{c_\beta}}u_\beta^2 d\lambda_\beta.
    \end{equation*}
    Since $c_\beta\geq 1$, from the dominated convergence theorem, we have
    $$\lim_{\epsilon \rightarrow 0} \int_{\mathring{X}_\beta}\frac{u_\beta^n}{u_\beta^n + \frac{\epsilon}{c_\beta}}u_\beta^2 d\lambda_\beta = \int_{\mathring{X}_\beta}u_\beta^2 d\lambda_\beta= 1.$$
    Also,
    \begin{align*}
        \int_{\mathring{X}_\beta} \frac{P_h u_\beta^n}{u_\beta^n + \frac{\epsilon}{c_\beta}}u_\beta^2 d\lambda_\beta
        = \int_{u_\beta\leq n} \frac{P_h u_\beta^n}{u_\beta + \frac{\epsilon}{c_\beta}}u_\beta^2 d\lambda_\beta + \int_{u_\beta > n} \frac{P_h u_\beta^n}{n + \frac{\epsilon}{c_\beta}}u_\beta^2 d\lambda_\beta.
    \end{align*}
    From the dominated convergence theorem,
    $$\lim_{\epsilon\rightarrow 0}\int_{u_\beta\leq n} \frac{P_h u_\beta^n}{u_\beta + \frac{\epsilon}{c_\beta}}u_\beta^2 d\lambda_\beta = \int_{u_\beta\leq n} (P_h u_\beta^n) u_\beta d\lambda_\beta.$$
    Also,
    $$\int_{u_\beta > n} \frac{P_h u_\beta^n}{n + \frac{\epsilon}{c_\beta}}u_\beta^2 d\lambda_\beta\leq \int_{u_\beta>n} u_\beta^2 d\lambda_\beta.$$
    Taking $\epsilon \rightarrow 0$ on equation \eqref{111}, we have
    $$\lim_{\epsilon\rightarrow 0} \int_{\mathring{X}_\beta} -\frac{P_h u - u}{u} u_\beta^2 d\lambda_\beta \geq 1 -\int_{u_\beta\leq n} (P_h u_\beta^n) u_\beta d\lambda_\beta - \int_{u_\beta>n} u_\beta^2 d\lambda_\beta.$$
    Finally, taking $n\rightarrow \infty$, we have
    $$\lim_{n\rightarrow \infty} (1 -\int_{u_\beta\leq n} (P_h u_\beta^n) u_\beta d\lambda_\beta - \int_{u_\beta>n} u_\beta^2 d\lambda_\beta) = \int_{\mathring{X}_\beta} (u_\beta - P_h u_\beta) u_\beta d\lambda_\beta.$$
    Therefore, we have
    $$I_h(\mu) \geq \sum_{\beta\in\mc C} \mu(\mathring{X}_\beta) \int_{\mathring{X}_\beta} (u_\beta - P_h u_\beta) u_\beta d\lambda_\beta.$$
    So,
    $$I(\mu) \geq \sum_{\beta\in\mc C} \mu(\mathring{X}_\beta) \int_{\mathring{X}_\beta} u_\beta\left(\frac{I - P_h}{h}\right)u_\beta d\lambda_\beta.$$
    If $E(\Lambda_\beta)$ is the spectral measure corresponding to $u_\beta$ relative to the operator $-L^\beta$, we get
    $$\sum_{\beta\in\mc C} \mu(\mathring{X}_\beta) \int^{\infty}_0 \frac{1-e^{-h\Lambda_\beta}}{h}E(d\Lambda) \leq I(\mu).$$
    Taking $h\rightarrow 0$ proves that
    $$ \sum_{\beta\in\mc C} \mu(\mathring{X}_\beta) \left\| \sqrt{-L^\beta} u_\beta \right\|^2_{L^2(\lambda_\beta)} =  \sum_{\beta\in\mc C} \mu(\mathring{X}_\beta) \int^{\infty}_0 \Lambda E(d\Lambda) \leq \ms K(\mu).$$
\end{proof}

Now, we use condition \textbf{(D1*)} to prove the $\limsup$ part of Theorem \ref{gen_diffusive}.
To do so, we first define a set of nice measures in $\ms P(X)$, which we call treatable measures.
\begin{defn}
    For $\mu$ in $\ms P(X)$, decompose it as $$\mu = \sum_{\beta \in \mc C} \mu(\mathring{X}_\beta) \mu(\cdot|\mathring{X}_\beta).$$
    We say $\mu$ is a treatable measure if $\frac{d\mu(\mathring{X}_\beta)}{d\lambda_\beta} \in \mc D_{\beta,0}$ for all $\beta \in \mc C$ with
    $\mu(\mathring{X}_\beta)\neq 0$.
\end{defn}

\begin{lemma} \label{treatablelimsup}
    For a treatable probability measure $\mu$ in $\ms P(X)$, there exists a sequence of measures $\mu_N$ in $\ms P(X)$ such that
    \begin{equation*}
        \lim_{N\rightarrow \infty} \ms I_N(\mu_N) \leq \ms K(\mu).
    \end{equation*}
\end{lemma}

\begin{proof}
    For $\mu$ in $\ms P(X)$, decompose it as
    $$\mu = \sum_{\beta \in \mc C} \mu(\mathring{X}_\beta) \mu(\cdot|\mathring{X}_\beta).$$
    Let $f^\beta = \sqrt{\frac{d\mu(\cdot|\mathring{X}_\beta)}{d\lambda_\beta}}$.
    From the definition of treatable measure, we have $f^\beta \in \mc D_{\beta,0}$.
    Using condition \textbf{(D1*.2)}, we take sequences of functions $f^\beta_N:X_N \rightarrow \bb R $ for each $\beta \in \mc C$, $N\in \bb N$
    such that \eqref{D1*21}, \eqref{D1*22} hold.
    We define a sequence of measures $\nu_N$ in $\ms P(X_N)$ as
    $$d\nu_N = \frac{1}{Z_N}\left(\sum_{\beta \in \mc C} \sqrt{\mu(\mathring{X}_\beta)} f^\beta_N\right)^2 d s_N,$$
    where $Z_N$ is a normalizing constant to make $\nu_N$ a probability measure.
    
    We first claim that $Z_N$ converges to 1 as $N\rightarrow \infty$. This is from \eqref{D1*21}, calculating the following integration:
    \begin{align*}
        \lim_{N\rightarrow \infty} Z_N = \lim_{N\rightarrow \infty} \int_{X_N} \left(\sum_{\beta \in \mc C} \sqrt{\mu(\mathring{X}_\beta)} f^\beta_N\right)^2 d s_N 
        =  \sum_{\beta \in \mc C} \mu(\mathring{X}_\beta) \int_{\mathring{X}_\beta} (f^\beta)^2 d \lambda_\beta = 1.
    \end{align*}
    Also from \eqref{D1*22}, we have
    \begin{align*}
        \lim_{N\rightarrow \infty} I_N(\nu_N) = \lim_{N\rightarrow \infty} \frac{1}{Z_N}\int_{X_N} -\left(\sum_{\beta \in \mc C} \sqrt{\mu(\mathring{X}_\beta)} f^\beta_N\right)A_N\left(\sum_{\beta \in \mc C} \sqrt{\mu(\mathring{X}_\beta)} f^\beta_N\right) ds_N
        =  \sum_{\beta \in \mc C} \mu(\mathring{X}_\beta)Q^\beta(f^\beta).
    \end{align*}
    Taking $\mu_N = \nu_N\iota_N^{-1}$, we obtain that $\mu_N$ weakly converges to $\mu$ from \eqref{D1*21} and the fact that $Z_N$ converges to 1.
    This gives the desired result.
\end{proof}
Finally, combining the above lemmas, we obtain \eqref{11111better} and the $\limsup$ part of Theorem \ref{gen_diffusive}.
\begin{proof}[Proof of Theorem \ref{gen_diffusive}]    
    Fix a treatable measure $\mu$. From Lemma \ref{treatablelimsup}, there exists a sequence $\mu_N$ converges to $\mu$ such that
    $$\lim_{N \rightarrow \infty} \ms I_N(\mu_N) = \sum_{\beta\in \mc C} \mu(\mathring{X}_\beta) Q^\beta\left(\sqrt{\frac{d\mu(\cdot|\mathring{X}_\beta)}{d\lambda_\beta}}\right) \leq \ms K(\mu).$$
    On the other hand, from Proposition \ref{slowliminf}, we have
    $$\liminf_{N \rightarrow \infty} \ms I_N(\mu_N) \geq \mc K(\mu),$$
    so we get
    $$\sum_{\beta \in \mc C} \mu(\mathring{X}_\beta) Q^\beta\left(\sqrt{\frac{d\mu(\cdot|\mathring{X}_\beta)}{d\lambda_\beta}}\right) = \ms K(\mu).$$
    
    Take a general measure $\mu$. Suppose $\ms K(\mu)$ is finite. Decompose $\mu$ as
    $$ \mu = \sum_{\beta \in \mc C} \mu(\mathring{X}_\beta) \mu(\cdot|\mathring{X}_\beta). $$
    From Lemma \ref{11111}, we obtain
    $$\sum_{\beta \in \mc C} \mu(\mathring{X}_\beta) Q^\beta\left(\sqrt{\frac{d\mu(\cdot|\mathring{X}_\beta)}{d\lambda_\beta}}\right) < \infty.$$
    For $\beta \in \mc C$ such that $\mu(\mathring{X}_\beta) \neq 0$, let $u_\beta = \sqrt{\frac{d\mu(\cdot|\mathring{X}_\beta)}{d\lambda_\beta}}$.
    From condition \textbf{(D1*.1)}, we may take sequences $(u^n_\beta\in \mc D_{\beta,0})_{n\in \bb N}$ for each $\beta$ such that
    \begin{equation} \label{Qapprox}
        \lim_{n\rightarrow \infty} \int_{\mathring{X}_\beta} |u^n_\beta - u_\beta|^2 d\lambda_\beta = 0, \;\; \lim_{n\rightarrow \infty} Q^\beta(u^n_\beta) = Q^\beta(u_\beta).
    \end{equation}
    Take $v^n_\beta = \frac{1}{(\int_{\mathring{X}_\beta} (u^n_\beta)^2 d\lambda_\beta)^\frac{1}{2}} u^n_\beta$, which is an normalization of $u^n_\beta$.
    We still have \eqref{Qapprox} after we replace $u^N_\beta$ by $v^N_\beta$ since
    $$\lim_{n\rightarrow \infty} \int_{\mathring{X}_\beta} (u^n_\beta)^2 d\lambda_\beta  =  \int_{\mathring{X}_\beta} (u_\beta)^2 d\lambda_\beta = 1.$$
    Now we take a sequence of probability measures 
    \begin{equation} \label{Qapprox2}
        \mu_n = \sum_{\beta \in \mc C} \mu(\mathring{X}_\beta) (v^n_\beta)^2 d\lambda_\beta.
    \end{equation}
    We directly get $\mu_n$ weakly converges to $\mu$ since $L^1$-convergence of a density function implies weak convergence.
    Since the rate functional is lower semi-continuous, we have 
    $$ \ms K(\mu) \leq \liminf_{n\rightarrow \infty} \ms K(\mu_n) = \liminf_{n\rightarrow \infty} \sum_{\beta \in \mc C} \mu(\mathring{X}_\beta) Q^\beta(v^n_\beta)
    = \sum_{\beta \in \mc C} \mu(\mathring{X}_\beta) Q^\beta(u_\beta).$$
    From Lemma \ref{11111}, we already have
    $$\ms K(\mu) \geq \sum_{\beta \in \mc C} \mu(\mathring{X}_\beta) Q^\beta(u_\beta),$$
    so we obtain \eqref{11111better}.

    Sequence in \eqref{Qapprox2} implies that for any $\mu$ in $\ms P(\Xi)$, there exists a sequence of treatable measures $\mu_n$ such that $\mu_n$ weakly converges to $\mu$ and
    $$\lim_{n\rightarrow \infty} \ms K(\mu_n) = \ms K(\mu).$$
    In this situation, applying \cite{L22G}[Lemma B.5],
    $\Gamma$-convergence of $\ms I_N$ to $\ms K$ follows from the $\Gamma$-convergence on the set of treatable measures.    
\end{proof}

\section{Preliminary for metastable time scale} \label{prelimmeta}

From now on, we apply the results in Section \ref{generalframework} to the zero-range process.
The next two sections focus on the metastable time scale $\theta_N = N^{1+\alpha}$.
Most of the results refer to \cite{RES}.

\subsection{Zero-range process and general framework}

Sets $\mc H_N$, $\Xi$ in Section \ref{euclideanembedding} play the role of $X_N$, $X$ in Section \ref{generalframework}, respectively.
So the embedding $\iota_N$ in Section \ref{euclideanembedding} is the embedding $\iota_N$ in Section \ref{generalframework}.
For the metastable time scale, $A_N$ in Section \ref{generalframework} is chosen as $A_N = N^{1 + \alpha} \ms L_N$ and we refer to the result in Section \ref{subsec32}. It is demonstrated through Section \ref{prelimmeta} and \ref{metatimescale}.
For the diffusive time scale, $A_N$ in Section \ref{generalframework} is chosen as $A_N = N^2 \ms L_N$ and we refer to the result in Section \ref{subsec33}. This is illustrated in Section \ref{prelimpremeta}, \ref{scc}, and \ref{p39}.

\subsection{A resolvent approach to metastability}

Recall the definition of $\mc E_N$ and $\Delta_N$ from Section \ref{metastability24}.
Now we define $\breve{\mc E_N^x}$ as
$$\breve {\mc E_N^x}\coloneq \mc E_N \setminus \mc E_N^x.$$
Denote by $r_N(x,y)$ the mean-jump rate between the sets $\mc E_N^x$ and $\mc E_N^y$:
\begin{equation*}
    r_N(x,y) = \frac{1}{\rho_N(\mc E_N^x)}\sum_{\eta\in \mc E_N^x} \rho_N(\eta)\lambda_N(\eta)P_{\eta}^N[\tau_{\mc E_N^y}<\tau^+_{\breve{\mc E_N^y}}],
\end{equation*}
where $\rho_N$ is a stationary measure of the process $(\xi_N)$ generated by $\ms L_N$.
In this formula, $\tau_{\ms A}, \tau^+_{\ms A}, \ms A\subset \mc H_N$, stands for the hitting, return time of $\ms A$, respectively:
\begin{align*}
    \tau_{\ms A} &= \inf\{t>0:\xi_N(t)\in \ms A\}. \\
    \tau^+_{\ms A} &= \inf\{t\geq \sigma_1:\xi_N(t)\in \ms A\}, \text{ where } \sigma_1 = \inf \{t\geq 0 : \xi_N(t)\neq \xi_N(0)\}. \nonumber
\end{align*}
The following conditions are required to describe the metastable behavior of the zero-range process.
\smallskip

\noindent \textbf{Condition (H0)} For all $x\neq y\in S$, the sequence $r_N(x,y)$ converges. Denote its limit by $r(x,y)$:
$$r(x,y) = \lim_{N\rightarrow \infty}r_N(x,y).$$
\textbf{Condition (H1)} For each $x\in S$, there exists a sequence of configurations $(\xi_N^x:N\geq 1)$ such that in $\xi_N^x\in \mc E_N^x \text{ for all } N\geq 1$ and 
$$\lim_{N\rightarrow \infty} \max_{\eta \in \mc E_N^x} \frac{\Cap_N(\mc E_N^x,\breve{\mc E_N^x})}{\Cap_N(\xi_N^x,\eta)}=0.$$

\begin{thm} \cite[Theorem 2.8]{RES}
\label{71}
    Assume that conditions (H0) and (H1) are in force. Then the solution $F_N$ of the resolvent equation (2.1) is asymptotically constant on each well $\mc E_N^x$ in the sense that
    $$\lim_{N\rightarrow \infty}\max_{x\in S}\max_{\eta, \zeta\in \mc E_N^x}|F_N(\eta)-F_N(\zeta)|=0.$$
    Furthermore, let $f_N:S\rightarrow \bb R$ be the function given by
    $$f_N(x) = \frac{1}{\rho_N(\mc E_N^x)}\sum_{\eta\in \mc E_N^x}F_N(\eta)\rho_N(\eta), \quad x\in S,$$
    and let $f$ be a limit point of the sequence $f_N$. Then
    $$[(\lambda-\ms L_Y)f](y) = g(y)$$
    for all $y\in S$ such that $\rho_N(\Delta_N)/\rho_N(\mc E_N^y)\rightarrow 0$, in which $g$ is the function in equality (2.1). In this formula, $\ms L_Y$
    is the generator of the continuous-time Markov process whose jump rates are given by $r(x,y)$, introduced in (H0).
\end{thm}

    Now we prove the Theorem \ref{mszrp}. The proof is based on the Theorem \ref{71} and the results from \cite{SZRP}.
\begin{proof}[Proof of \ref{mszrp}]
    We need to show that conditions (H0), (H1), the condition
    $$\rho_N(\Delta_N)/\rho_N(\mc E_N^y)\rightarrow 0$$
    hold for the sequence $\ell_N = \lfloor N^{\frac{1}{2(\kappa-1)}} \rfloor$.
    Each condition holds from \cite[Proposition 6.22]{SZRP}, \cite[display (6.5)]{SZRP}, Theorem \ref{meas}, respectively.
\end{proof}

\section{Proof of Theorem \ref{27}} \label{metatimescale}
The goal of this section is to prove the $\Gamma$-convergence of $N^{1+\alpha}\mc I_N$ to the rate function $\mc J$ which is presented in \eqref{rate2}.
By Theorem \ref{gen_metastable}, we have the $\Gamma$-convergence if we can show condition \textbf{(M0)} and \textbf{(M0*)}.

\subsection{Conditions (M0) and (M0*)}

Recall the generator $\bb L$ from \eqref{limiting}:
$$
    (\bb L f)(x) = \frac{\kappa}{\Gamma(\alpha)I_\alpha} \sum_{y\in S} \Cap_S(x,y)(f(y)-f(x)), \; x\in S,
$$
for $f:S\rightarrow \bb R$.
For simplicity, define $R: S\times S\rightarrow \bb R$ as
$$R(x,y) \coloneq \frac{\kappa}{\Gamma(\alpha)I_\alpha}\Cap_S(x,y).$$
so that $R$ generates the operator $\bb L$.

From Theorem \ref{mszrp}, we have condition $\mf R_{\bb L}$ for accelerated generator $N^{1+\alpha} \ms L_N$ and $\mc E^{\ell,x}_N$ with $\ell_N = \lfloor N^{\frac{1}{2(\kappa-1)}} \rfloor$.
From this, condition \textbf{(M0)} is immediate.

Recall that $\mc H_N$ is the set of configurations with $N$ particles and
$\Delta_N$ is defined as
$$\Delta_N = \mc H_N \setminus \bigcup_{x\in S} \mc E^{\ell,x}_N.$$
Recall the rate function $I_N$ from \eqref{plainratefunc}.
To check condition \textbf{(M0*)}, we prove the following proposition.

\begin{prop} \label{boundmeasure}
    Assume that the zero range process is reversible and has the uniform measure condition, that is, $r(x,y)=r(y,x)$ for all $x,y\in S$.
    Take any $\ell_N$ such that $\ell_N \prec N$ and $\ell_N \rightarrow \infty$.
    Let $\Delta_N$ be a set in \eqref{plain}.
    Fix an increasing sequence of natural numbers $(N_k)_{k\in \bb N}$.
    For $C>0$, $N_k^{1+\alpha} I_{N_k}(\nu_{N_k})<C$ for all $k\in \bb N$ implies
    $$\lim_{k\rightarrow \infty} \nu_{N_k}(\Delta_{N_k}) = 0.$$
\end{prop}

In order to prove this proposition, we first analyze the case $|S| =2$.
Let $r(x,y) = r(y,x) = r > 0$ for $S = \{x,y\}$.
Let $\eta^x, \eta^y\in \mc H_N$ be configurations representing condensation such that $(\eta^z)_z = N$ for $z\in S$.
Let $B = \{\eta^x,\eta^y\}$.

\begin{lemma}
    \label{boundcap}
    There exists a positive constant $c$ such that for all $\eta \in \mc H_N \setminus B$,
    $$\Cap_N(\eta^x,\eta)^{-1} \leq c \eta_{x}^{1+\alpha}.$$
\end{lemma}

\begin{proof}
    From \cite[display (7.1.60)]{Metastability}, capacity between two points in a one-dimensional nearest-neighbor random walk can be 
    interpreted as an inverse of a sum of inverse conductances.
    Therefore, we have
    $$\Cap_N(\eta^x,\eta)^{-1} = \frac{Z_{N,S}}{N^\alpha} \{ (N-1)^\alpha + \sum_{i=1}^{\eta_x-1} i^\alpha (N-1-i)^\alpha\} \leq c \eta_x^{\alpha + 1}.$$
\end{proof}

For $\eta\in \mc H_N$, let $\bb P^N_{\eta}$ be a probability measure on $D(\bb R_+, \mc H_N)$ generated by $\ms L_N$ starting from $\eta$.

\begin{lemma} \label{boundharmonic}
    For any $1\leq M \leq \frac{N}{2}$, define a function $U_N$ as
    $$U_N(\zeta) =
    \begin{cases}
        \frac{\zeta_x}{M} & \text{if } \zeta_x \leq M, \\
        1 & \text{if } M \leq \zeta_x \leq N-M, \\
        \frac{N-\zeta_x}{M} & \text{if } \zeta_x \geq N-M.
    \end{cases}$$
    Then for $\eta \in \mc H_N \setminus B$ with $ M \leq \eta_x \leq N - M$, we have
    $$\bb P^N_{\zeta}[\tau_B > \tau_\eta] \leq U_N(\zeta) \text{ for all } \zeta\in \mc H_N.$$
\end{lemma}

\begin{proof}
    Observe that $U_N$ is a superharmonic function for the generator $\ms L_N$ which satisfies $U_N(\eta) = 1$ and $U_N(B) = 0$.
    Therefore,
    $$\bb P^N_{\zeta}[\tau_B > \tau_\eta] = \bb E_{\zeta}[u(\tau_{\eta\cup B})] \leq U_N(\zeta),$$
    where $\tau_{\eta\cup B}$ is a hitting time of the set $\eta\cup B$.
\end{proof}

The following lemma gives an upper bound of the expectation of the hitting time of the set $B$.

\begin{lemma} \label{boundhittingtime}
    The following estimation holds:
     \begin{equation*}
        \sup_{\eta\in \mc H_N}\bb E^N_{\eta}[\tau_B] = o_N(N^{1+\alpha}).
    \end{equation*}
\end{lemma}

\begin{proof}
    Let $\overline{\eta}$ be a configuration in $\mc H_N$ such that $\overline{\eta}_x = \eta_y$ and $\overline{\eta}_y = \eta_x$.
    Then from the symmetry, we have
    $$\bb E^N_{\eta}[\tau_B] = \bb E^N_{\overline{\eta}}[\tau_B].$$
    Without loss of generality, assume that $\eta_x \leq \lfloor \frac{N}{2} \rfloor$.
    From \cite[display (7.1.42)]{Metastability}, we have
    \begin{equation} \label{hittingtime}
        \bb E^N_{\eta}[\tau_B] = \frac{1}{\Cap_N(\eta,B)}\sum_{\zeta\in \mc H_N} \rho(\zeta)\bb P_{\zeta}[\tau_B > \tau_\eta].
    \end{equation}

    We first estimate $\Cap_N(\eta,B)$. 
    Since this is a one-dimensional nearest-neighbor random walk, we have
    $$\Cap_N(\eta,B) = \Cap_N(\eta^x, \eta) + \Cap_N(\eta^y, \eta) \geq \Cap_N(\eta^x, \eta)$$
    Fix $0<\epsilon <1$. We divide cases. First, assume that $\eta_x \geq N^{\epsilon}$.
    Consider a constant $M = \lfloor N^{\epsilon} \rfloor$ and a set $\Delta_N = \{\eta\in \mc H_N : M < \eta_x < N-M\}$.
    Applying Lemma \ref{boundharmonic} with $M$, we have
    \begin{align*}
    \sum_{\zeta\in \mc H_N} \rho(\zeta)\bb P^N_{\zeta}[\tau_B > \tau_\eta]
    \leq \sum_{\zeta\in \mc H_N} \rho(\zeta)U_N(\zeta) \leq& 2 \sum_{i=1}^{M} \frac{N^\alpha}{Z_{N,S}} \frac{1}{i^\alpha(N-i)^\alpha}\frac{i}{M} + \sum_{\substack{\eta\in \Delta_N}} \rho(\eta) \\
    \leq& O_N\left(\frac{1}{M} \sum_{i=0}^M \frac{1}{i^{\alpha-1}}\right) + \sum_{\substack{\eta\in \Delta_N}} \rho(\eta) = o_N(1).
    \end{align*}
    From Lemma \ref{boundcap} and \eqref{hittingtime}, we have $\bb E^N_{\eta}[\tau_B] = o_N(N^{1+\alpha})$.

    Next, assume that $\eta_x \leq N^{\epsilon}$. Then Lemma \ref{boundcap} shows
    $$\bb E^N_{\eta}[\tau_B] = \frac{1}{\Cap_N(\eta,B)}\sum_{\zeta\in \mc H_N} \rho(\zeta)\bb P^N_{\zeta}[\tau_B > \tau_\eta] \leq c N^{\epsilon(\alpha+1)} = o_N(N^{\alpha+1}).$$
\end{proof}

Now we begin the proof of Proposition \ref{boundmeasure}.
\begin{proof}[Proof of \ref{boundmeasure}]
    We first handle the case $\kappa = 2$.
    From the Markov property and Lemma \ref{boundhittingtime}, we have
    \begin{align} \label{5501}
        \sup_{\eta\in \mc H_N} \bb P^N_{\eta}[\eta_{\delta N^{1+\alpha}} \in \Delta_N] =&
        \sup_{\eta\in \mc H_N} \{ \bb P^N_{\eta}[\eta_{\delta N^{1+\alpha}} \in \Delta_N, \tau_B < \delta N^{\alpha + 1}] + \bb P^N_{\eta} [\tau_B \geq \delta N^{\alpha+1}] \} \nonumber \\
        \leq& \sup_{0\leq s\leq \delta N^{\alpha+1}} \bb P^N_{\eta^x}[\eta_s\in \Delta_N] + o_N(1).
    \end{align}
    Recall that $\rho_N$ is a stationary measure of the process generated by $\ms L_N$. Since $\rho_N(\eta^x) = \frac{1}{Z_{N,S}}$, we have
    \begin{equation} \label{5502}
        \bb P^N_{\eta^x}[\eta_s\in \Delta_N] \leq Z_{N,S} \bb P^N_{\rho}[\eta_s\in \Delta_N] = Z_{N,S} \rho(\Delta_N) = o_N(1).
    \end{equation}
    Therefore, putting \eqref{5502} into \eqref{5501}, we get
    \begin{equation} \label{5503}
        \sup_{\eta\in \mc H_N} \bb P^N_{\eta}[\eta_{\delta N^{1+\alpha}} \in \Delta_N] = o_N(1).
    \end{equation} 
    Let $P_{N,t}$ be a transition kernel generated by $\ms L_N$.
    From Lemma \ref{93} and \cite[Lemma 4.1]{DV75a}, there exists a function $\phi:\bb R_{\geq 0}\rightarrow \bb R_{\geq 0}$ such that
    $0\leq \phi \leq 2$, increasing, $\phi(l) \rightarrow 0$ as $l\rightarrow 0$, having the property
    \begin{equation} \label{5504}
        \|\nu P_{N,t} - \nu\|_{TV} \leq \phi(t I_N(\nu))
    \end{equation}
    For $\delta >0$, putting $t = \delta N^{\alpha+1}$, $N = N_k$ into \eqref{5504}, we get
    \begin{equation} \label{5505}
        |\nu P_{N_k,\delta N_k^{\alpha+1}}(\Delta_{N_k}) - \nu(\Delta_{N_k}) | \leq \|\nu P_{N_k,\delta N_k^{\alpha+1}} - \nu\|_{TV} \leq \phi(\delta C).
    \end{equation}
    For any $\epsilon >0$, we can take $\delta >0$ so that $\phi(\delta C) < \epsilon$. Taking $k\rightarrow \infty$ and plugging \eqref{5503} into \eqref{5505} gives
    \begin{equation*}
        \limsup_{k\rightarrow \infty} \nu_{N_k}(\Delta_{N_k}) \leq \epsilon + \limsup_{k\rightarrow \infty} \nu P_{N_k,\delta N_k^{\alpha+1}}(\Delta_{N_k}) = \epsilon.
    \end{equation*}
    Therefore, we have the result for the case $\kappa = 2$.

    For the general case, let $ r_0 = \min \{ r(x,y) : x,y\in S, r(x,y)>0 \}$.
    Recall that $\ms L_N$ is determined by the jump rates $r(x,y)$. To clarify the dependence, we denote $\ms L_N^{r}$ as a generator of the zero range process with the jump rates $r(x,y)$.
    Also, we denote $I_N^r$ as a rate function associated with $\ms L_N^{r}$.

    We define a degenerated jump rate $r^{x,y}:S\times S\rightarrow \bb R_{\geq 0}$ as
    $$r^{x,y}(z,w) = r_0 \text{ for } (z, w) = (x,y), (y,x) \text{ and } r^{x,y}(z,w) = 0 \text{ otherwise}.$$
    Take a generator $\ms L_N^{r^{x,y}}$ and a rate function $\ms I_N^{r^{x,y}}$.
    We claim that there exists a positive constant $c$ depending on $r$ such that
    \begin{equation} \label{5506}
    I_N^{r^{x,y}}(\nu) \leq c I^r_N(\nu) \text{ for all } N\in \bb N ,\; \nu\in \ms P(\Xi) \text{ and } x\neq y \in S.
    \end{equation}
    From \eqref{sqrtcalc}, we have
    \begin{align*}
    I_N^r(\nu) = \frac{1}{2}\sum_{z,w\in S}\sum_{(\eta, \eta^{z,w})\in \mc H_N\times \mc H_N} \rho_N(\eta)g(\eta_z)r(z,w)
    \left(\sqrt{\frac{\nu(\eta)}{\rho_N(\eta)}} - \sqrt{\frac{\nu(\eta^{z,w})}{\rho_N(\eta^{z,w})}}\right)^2,
    \end{align*}
    and
    \begin{align*}
    I_N^{r^{x,y}}(\nu) = \frac{1}{2}\sum_{(\eta, \eta^{x,y})\in \mc H_N\times \mc H_N} \rho_N(\eta)g(\eta_x)r_0
    \left(\sqrt{\frac{\nu(\eta)}{\rho_N(\eta)}} - \sqrt{\frac{\nu(\eta^{x,y})}{\rho_N(\eta^{x,y})}}\right)^2.
    \end{align*}
    If $r(x,y)>0$, \eqref{5506} is obvious. If $r(x,y) = 0$,
    we may choose a conneted path $x = x_1, x_2, \cdots, x_n = y$ such that $r(x_i,x_{i+1})>0$ for all $i$.
    For convenience, let $x_{0} = x$ and $x_{n+1} = y$.
    Then there exists a positive constant $c$ such that
    \begin{align*}
        \sum_{i=0}^{n} \rho(\eta^{x_i,x_{i+1}})g(\eta_{x_i})r(x_i,x_{i+1})\left(\sqrt{\frac{\nu(\eta)}{\rho(\eta)}} - \sqrt{\frac{\nu(\eta^{x_i,x_{i+1}})}{\rho(\eta^{x_i,x_{i+1}})}} \right)^2
        \geq c\rho(\eta)g(\eta_x)r_0\left(\sqrt{\frac{\nu(\eta)}{\rho(\eta)}} - \sqrt{\frac{\nu(\eta^{x,y})}{\rho(\eta^{x,y})}} \right)^2,
    \end{align*}
    because of the fact that 
    $$\inf_{N\in \bb N}\inf_{\eta \in \mc H_N }\inf_{z,w\in S} \frac{\rho_N(\eta^{z,w})}{\rho_N(\eta)} > 0 \text{ and } \;\; 1\leq g\leq 2^\alpha \text{ on } \bb N$$
    and Cauchy-Schwarz inequality.

    Fix a positive constant $T$. Let $I^2_N$ be a rate function of a zero-range process for $S = \{x,y\}$ and $r(x,y) = r(y,x) = r_0$.
    Let $\mc H^2_N$ be a set of configurations in \eqref{mchn} for $S = \{x,y\}$.
    We construct a sequence $(\ell^T_n)$ satisfying
    \begin{enumerate}
        \item $0\leq \ell^T_n \leq \frac{N}{2}$,
        \item For all $n\in \bb N$, if a measure $\nu\in \ms P(\mc H^2_n)$ satisfies $n^{1+\alpha} I^2_n(\nu) \leq T$, then $\nu(\Delta^2_n)<T^{-1}$, where
        $$\Delta^2_n = \{\eta\in \mc H^2_n : \ell^T_n < \eta_x < n - \ell^T_n\}.$$
        \item $\ell^T_n \prec \ell_n$ and $\ell^T_n \rightarrow \infty$ as $n\rightarrow \infty$.
    \end{enumerate}
    To construct such a sequence, we take a sequence $(\ell^T_n)$ such that $0\leq \ell^T_n \leq \frac{N}{2}$, $\ell^T_n \prec \ell_n$ and $\ell^T_n \rightarrow \infty$.
    Then there exists $N_0\in \bb N$ such that (2) holds for all $n\geq N_0$ from the result for the case $\kappa = 2$.
    Then we reassign $\ell^T_n = \lfloor \frac{n}{2} \rfloor$ before all $n<N_0$.

    Take a subset $K_N^{x,y}, K_N$ of $\mc H_N$ as
    $$K_N^{x,y} = \{\eta\in \mc H_N : \text{ if } \eta_x + \eta_y = n,\text{ then } \ell_n^T < \eta_x < n - \ell_n^T\},\; K_N = \cup_{x,y\in S} K_N^{x,y}.$$
    Before further the argument, we introduce a notation. Let $S^{x,y} = S\setminus \{x,y\} \cup \{*\}$, identifying $x$ and $y$ as $*$.
    Let $\mc H^{S^{x,y}}_N$ be a subset of ${\bb N}^{S^{x,y}}$ which
    consists of points whose coordinate sum is $N$.
    Now, we may interpret an element $\zeta\in \mc H^{S^{x,y}}_N$
    as a set of points in $\mc H_N$ such that
    \begin{equation} \label{quotientspace}
        \zeta = \{\eta\in \mc H_N : \eta_z = \zeta_z \text{ for } z\in S\setminus \{x,y\}\}.
    \end{equation}
    Note that for $\eta \in \zeta$, $\zeta_* = \eta_x + \eta_y$.
    For $\zeta\in \mc H^{S^{x,y}}_N$, we define a set $\Delta_N^\zeta$ as a set of configurations in $\mc H_N$ such that
    $$\Delta_N^\zeta = \{\eta\in \mc H_N : \eta\in \zeta,\; \ell_{\zeta_*}^T < \eta_x < \zeta_* - \ell_{\zeta_*}^T\} =  K_N^{x,y} \cap \zeta .$$
    To prove the claim, we decompose a measure $\nu_{N_k}\in \ms P(\mc H_{N_k})$ as
    $$\nu_{N_k} = \sum_{\zeta\in \mc H^{S^{x,y}}_{N_k}} \nu_{N_k}(\zeta) \nu_{N_k}(\cdot|\zeta).$$
    Then we have the following identities.
    \begin{enumerate}
        \item $\nu_{N_k}(K_{N_k}^{x,y}) = \sum_{\zeta\in \mc H^{S^{x,y}}_{N_k}} \nu_{N_k}(\zeta) \nu_{N_k}(\Delta_{N_k}^\zeta|\zeta)$,
        \item $I_{N_k}^{r^{x,y}}(\nu_{N_k}) = \sum_{\zeta\in \mc H^{S^{x,y}}_{N_k}} \nu_{N_k}(\zeta) I_{\zeta}^2(\nu_{N_k}(\cdot|\zeta))$,
        where $I_{\zeta}^2$ is a rate function for the zero-range process with two sites on $\zeta$ which has the jump rates $r_0$.
    \end{enumerate}
    Therefore, we have
    \begin{align*}
        &\nu_{N_k}(K_{N_k}^{x,y}) = \sum_{\zeta_*^{1+\alpha}I_{\zeta}^2(\nu_{N_k}(\cdot|\zeta)) \leq T} \nu_{N_k}(\zeta) \nu_{N_k}(\Delta_{N_k}^\zeta|\zeta) + \sum_{\zeta_*^{1+\alpha}I_{\zeta}^2(\nu_{N_k}(\cdot|\zeta)) > T} \nu_{N_k}(\zeta) \nu_{N_k}(\Delta_{N_k}^\zeta|\zeta)\\
        &\leq \sum_{\zeta_*^{1+\alpha}I_{\zeta}^2(\nu_{N_k}(\cdot|\zeta)) \leq T} \frac{\nu_{N_k}(\zeta)}{T} + \sum_{\zeta_*^{1+\alpha}I_{\zeta}^2(\nu_{N_k}(\cdot|\zeta)) > T} \nu_{N_k}(\zeta) \frac{\zeta_*^{1+\alpha}I_{\zeta}^2(\nu_{N_k}(\cdot|\zeta))}{T} \leq \frac{1+ N_k^{1+\alpha}I_{N_k}^{r^{x,y}}(\nu_{N_k})}{T}.
    \end{align*}
    Thus we have $\nu_{N_k}(K_{N_k}) \leq \binom{|S|}{2} \frac{1+c N_k^{1+\alpha}I_{N_k}^r(\nu_{N_k})}{T}$.

    Finally, we claim that for large enough $k$, we have $\Delta_{N_k} \subset K_{N_k}$. For any $\eta\in \Delta_{N_k}$,
    let $x,y$ be a largest and a second largest elements in $S$ coordinate of $\eta$.
    We claim $\eta\in K_{N_k}^{x,y}$. Suppose not. Let $n = \eta_x + \eta_y$.
    Then we have $\eta_y \leq \ell_n^T$. Since $y$ is the second largest element, we have $N - \eta_x \leq (|S|-1)\ell_n^T$.
    Therefore,
    $$N -\eta_x \leq (|S|-1)\ell_n^T \leq \ell_n \leq \ell_N $$ for large $k$ since $\ell_n^T \prec \ell_n$. This leads to a contradiction.
    Therefore, we have $\Delta_{N_k} \subset K_{N_k}$ for large enough $k$.
    Then we have
    $$\limsup_{k\rightarrow \infty} \nu_{N_k}(\Delta_{N_k}) \leq \limsup_{k\rightarrow \infty} \nu_{N_k}(K_{N_k}) \leq \frac{1+c\limsup_{k\rightarrow \infty} N_k^{1+\alpha} I_{N_k}^r(\nu_{N_k})}{T}.$$
    Taking $T\rightarrow \infty$ gives the result.
\end{proof}

Proposition \ref{boundmeasure} implies that condition \textbf{(M0*)} holds for the reversible zero range process.
The following corollary shows that condition \textbf{(M0*)} also holds for the general non-reversible zero range process.

\begin{cor}
    Proposition \ref{boundmeasure} also holds for a non-reversible zero range process.
\end{cor}

\begin{proof}
    Recall the notation $\ms L_N^{r}$ and $I_N^{r}$ from the proof of Proposition \ref{boundmeasure}.
    For a transition $r$, we define a transition $r^*$, $r^s(x,y)$ as
    $$r^*(x,y) = r(y,x), \;\; r^s(x,y) = \frac{1}{2}(r(x,y) + r^*(x,y)).$$
    Directly from \eqref{zrpgen} and \eqref{stationaryprob}, we have that
    $\ms L_N^{r^*}$ is an adjoint of $\ms L_N^{r}$ in $L^2(\rho_N)$.
    Therefore, for any measure $\nu\in \ms P(\mc H_N)$, letting $f = \sqrt{\frac{d\nu}{d\rho_N}}$, we have
    $$I_N^{r}(\nu) \geq \int_{\mc H_N} -f \ms L_N^{r} f d\rho_N = \int_{\mc H_N} -f \ms L_N^{r^s} f d\rho_N
    = I_N^{r^s}(\nu).$$
    Therefore, we have $I_N^{r} \geq I_N^{r^s}$.
    This implies condition \textbf{(M0*)} holds for the non-reversible zero range process.
\end{proof}

\section{Preliminary for pre-metasatble time scale} \label{prelimpremeta}

In this section, we demonstrate some properties of the limiting process on the diffusive time scale which has once been introduced in Section \ref{limitingdiff}.
This plays a crucial role in the proof of the $\Gamma$-limsup part of Theorem \ref{46}.

\subsection{Behavior after absorption} \label{behaviorafterabsorption}
Recall the processes $(\bb P_{\xi}, \xi \in \Xi)$ from Theorem \ref{zrplaw} which are solutions to the $\mf L$-martingale problem \eqref{42}.
According to \cite{MZRP}, the processes exhibit an absorbing behavior.

\begin{prop} \cite[Proposition 2.3]{MZRP}
    For $\xi$ in $\Xi$, let $\ms N(\xi)=\{x\in S: \xi_x=0\}$. For any $\xi \in \Xi$, the set $\ms N(\xi_t)$ is increasing in inclusion relation $\bb P_{\xi}$ almost surely.
\end{prop}

The property in the above proposition is called the absorption property.
Remark that this notion is a reinterpretation of \cite[Proposition 2.3]{MZRP} as the statement is quite different but they are equivalent.
From this property, if a process starts from $\xi \in \Xi_A$, it will remain within $\Xi_A$ in probability 1.
Therefore, the probability measure $\bb P_{\xi}$ on $C(\bb R_+,\Xi)$ can be restricted to $C(\bb R_+,\Xi_A)$.

That restricted probability measure can be described using a family of generators which describes each step of an absorption.
For each $\varnothing\subsetneq A\subset S$, consider the simplex
$$\Xi_A \coloneq \{\xi\in \bb R^A_{\geq 0} : \sum_{x\in A} \xi_x =1 \},$$
and the space $C^2(\Xi_A)$ of functions $f:\Xi_A\rightarrow \bb R$ which are twice-continuously differentiable on the interior of $\Xi_A$ and have continuous second derivatives upto $\Xi_A$.
Denote by
$$\textbf{r}^A=\{r^A(x,y):x,y\in A\}$$
the jump rates of the trace of the Markov process generated by $\mc L_S$ on $A$.
Let $\{\textbf{v}_x^A:x\in A\}$ be the vectors in $\bb R^A$ defined by
$\textbf{v}_x^A \coloneq \sum_{y\in A} r^A(x,y)(\textbf{e}_y-\textbf{e}_x),$
where $\{\textbf{e}_x:x\in A\}$ stands for the canonical basis of $\bb R^A$. and let $\textbf{b}^A:\Xi_A\rightarrow \bb R^A$ be the vector field defined by
$$\textbf{b}^A(\xi)\coloneq \alpha\sum_{j\in A}\frac{1}{\xi_x}\textbf{v}_x^A\textbf{1}\{\xi_x>0\},\quad \xi\in \Xi_A.$$

Similar to $\mc D^S_A$ from Definition \ref{def31}, we define $\mc D^A_B$ for $\varnothing\subsetneq B \subset A$, the space of functions $H$ in $C^2(\Xi_A)$ for which the map
$\xi\mapsto \textbf{v}_x^A\cdot \nabla H(\xi)/\xi_x$ is continuous on $\Xi_A$ for $x\in B$, and let $\mf L^A$ be the operator which acts on functions in $C^2(\Xi_A)$
with the following equation.

\begin{equation*}
    (\mf L^A f)(\xi)\coloneq \textbf{b}^A(\xi)\cdot \nabla f(\xi) + \frac{1}{2}\sum_{x,y\in A} r^A(x,y)(\partial_{x}-\partial_{y})^2f(\xi),
\end{equation*}
for $\xi\in\Xi_A$.

Then the following proposition holds.

\begin{prop} \cite[Proposition 2.4]{MZRP}
    \label{p33}
    Fix $\xi$ in $\Xi$ and assume that $\ms N(\xi)=\{x \in S: \xi_x=0\} \neq \varnothing$. Let $A = \ms N(\xi)^c$. Take a measure $\bb P_{\xi}$ which 
    is a solution $\mf L$-martingale problem with starting point $\xi$. Let $\bb P_{\xi}^A$ denote the restriction of $\bb P_{\xi}$ on $C(\bb R_+,\Xi_A)$.
    Then the measure $\bb P_{\xi}^A$ solves the $\mf L^A$-martingale problem.
\end{prop}

The property described in this proposition is called the recursion property, as its subsimplex dynamics have a generator form identical to that of the original.
Together with the absorption property, the recursion property allows us to confine our interest to the largest simplex, $\Xi = \Xi_S$, in Section \ref{scc}.

\subsection{Extension maps}

Recall that $\mc L_S$ is the generator corresponding to jump rates $r(x,y)$, $x,y\in S$. For each $x\in A$, 
let $u^A_x : S\rightarrow \bb [0,1]$ be the only $\mc L_S$-harmonic extension on $S$ of the indicator function of $x$ on $A$.
In other words, $u^A_x$ is the solution to
\begin{equation*}
    \begin{cases}
        u^A_x(y) = \delta_{x,y} & \text{for } y\in A \\
        (\mc L_S u^A_x)(y) = 0 & \text{for } y\in S\setminus A.
        \end{cases}
\end{equation*}
Take $B = S\setminus A$.
Now, define a projection map $\gamma_A:\Xi\rightarrow \Xi_A$ as
$$[\gamma_A(\xi)]_x = \xi_x +  \sum_{y\in B} u^A_x(y) \xi_y, \;\;x\in A$$

Then $\gamma_A$ can be represented by a $|A|\times |S|$ matrix, with $u^A_x$ comprising the $x$-th row.
Let $\mc L_A$ denote the generator of the trace process whose jump rate is $r^A$.
From equations in \cite[Section 3.B]{MZRP}, we have the following result.

\begin{proposition}
    Consider $\gamma_A$, $\mc L_S$, $\mc L_A$ as $|A|\times |S|$, $|S|\times|S|$, $|A|\times |A|$ matrices, respectively.
    Then
    \begin{equation}
        \label{traceprocess}
        \gamma_A \mc L_S (\gamma_A)^\dagger = \mc L_A.
    \end{equation}
\end{proposition}

Now, define an extension map $\ms E_A:C(\Xi_A)\rightarrow C(\Xi)$ by
$$\ms E_A f(\xi) = f(\gamma_A(\xi)).$$
Then we have the following lemma.

\begin{lemma} \cite[Lemma 3.1]{MZRP} \label{harmext}
    For any function $f\in C^2(\Xi_A)$, the function $\ms E_A f$ is a function in $\mc D^S_B$ where $B=S\setminus A$.
\end{lemma}

This is a natural way to extend a function on $\Xi_A$ to $\Xi$ and it will be used in Section \ref{p39}.

\subsection{Approximation of limiting diffusion}

Let $X_t^N$ be a Markov process on $\mc H_N$ generated by $N^2 \ms L_N$.
Mapping this process to $\Xi_N$ by $\iota_N$ in \eqref{iotaembed}, we obtain a process $\xi_t^N$ on $\Xi_N$.
For $\xi_N\in \Xi_N$, let $\bb P_{\xi_N}^N$ be the probability measure on $D(\bb R_+,\Xi)$
induced by the process $\xi_t^N$ starting at $\xi_N$.
The following theorem shows that $\bb P_{\xi_N}^N$ somehow approximates $\bb P_{\xi}$, which is induced by the limiting diffusion.

\begin{thm} \label{approx}
    \cite[Theorem 2.6]{MZRP}.
    Let $\xi_N\in \Xi_N$ be a sequence converging to $\xi \in \Xi$.
    Then, $\bb P_{\xi_N}^N$ converges to $\bb P_{\xi}$ in the Skorohod topology.
\end{thm}

Additionally, limiting diffusion is Feller continuous.

\begin{prop}\cite[Proposition 7.10]{MZRP} \label{fellercont}
    Let $(\xi_n)$ be a sequence in $\Xi$ converging to $\xi\in \Xi$. Then, $\bb P_{\xi_n}$ converges to $\bb P_{\xi}$ in the Skorohod topology.
\end{prop}

We argue the stronger result that the limiting diffusion is a Feller process. To show this, we begin with the following lemma.
For a set $\varnothing \subsetneq B \subset S$, let
$$\|\xi\|_B \coloneq (\sum_{x\in B}\xi_x^2)^{\frac{1}{2}}.$$

\begin{lemma}
    \label{localftn}
    Fix a positive $\epsilon$. For any $\xi_0 \in \Xi$, there exists a function $\phi^\epsilon_{\xi_0} : \Xi \rightarrow \bb R_+$ such that $\phi^\epsilon_{\xi_0} \in \mc D^S_S$,
    $\phi^\epsilon_{\xi_0}(\xi_0) > 0$ and $\phi^\epsilon_{\xi_0}(\xi) = 0$ if $\|\xi_0 - \xi\|_S > \epsilon$.
\end{lemma}

\begin{proof}
    For $\xi \in \Xi$, let $\ms N(\xi)^c = A$, from Proposition \ref{p33}. Take $B = S \setminus A$.
    From \cite[Lemma 4.1]{MZRP}, there exists a nonnegative, smooth function $I_B:\Xi \rightarrow \bb R$ in $\mc D^S_B$
    and constants $0<c_1<C_1<\infty$ such that for all $\xi \in \Xi$,
    $$c_1\|\xi\|_B \leq I_B(\xi) \leq C_1\|\xi\|_B.$$
    For $\delta>0$, take any function $f_{\delta}$ in $C^2(\Xi_A)$ such that $f_\delta(\xi_0) > 0$ and $f_\delta(\xi) = 0$
    if $\|\xi_0 - \xi\|_S > \delta$.
    Then, we consider a harmonic extension of $\ms E_A f_\delta$.
    From Lemma \ref{harmext}, we get $\ms E_A f_\delta \in \mc D^S_B$.
    Now, consider a smooth function $\chi_{\delta} : \bb R \rightarrow \bb R$ with the property:
    \begin{center}
         $\chi_{\delta}(t) = 1$ if $t\leq 0$, $\chi_{\delta}(t) = 0$ if $t\geq \delta$ and $\chi_{\delta}$ is decreasing.
    \end{center}
    Now define a function $\phi_\delta = (\chi_\delta \circ I_B) (\ms E_A f_\delta$).
    From the fact that $I_B, \ms E_A f \in \mc D^S_B$, we have $\phi_\delta \in \mc D^S_B$.
    We claim that for small enough $\delta$, $\phi_\delta\in \mc D^S_{x}$ for all $x\in A$.
    For $\xi\in \Xi$, suppose $\phi_\delta(\xi)\neq 0$. Then we get
    $$I_B(\xi) < \delta \; \text{ and }  \; \ms E_A f_\delta(\xi) = f(\gamma_A(\xi)) > 0.$$
    This implies
    $$\|\xi\|_B < \frac{\delta}{c_1} \; \text{ and }  \; \|\gamma_A(\xi) - \xi_0\|_S \leq \delta.$$
    Since
    $$[\gamma_A(\xi)]_x = \xi_x + \sum_{y\in B} u^A_x(y)\xi_y,$$
    direct computation using Cauchy inequality gives
    $$\|\gamma_A(\xi) - \xi\|_S \leq (1+ |A|\sqrt{|B|}) \|\xi\|_B.$$
    Therefore, we have
    $$\|\xi - \xi_0\|_S \leq \|\xi - \gamma_A(\xi)\|_S + \|\gamma_A(\xi) - \xi_0 \|_S \leq (1+ \frac{1+|A|\sqrt{|B|}}{c_1})\delta.$$
    Taking $\delta$ small, we have $\phi_\delta$ is supported near $\xi_0$.
    This means that $\phi_\delta$ is zero if $\xi_x$ is small enough for $x\in A$.
    Therefore, we have $\phi_\delta \in \mc D^S_x$, $x\in A$.
\end{proof}

Using this lemma, we have the following result.
\begin{lemma}
    \label{partition}
    For any $\epsilon >0$, there exists a finite collection of functions $\chi_1, \cdots, \chi_m \in \mc D^S_S$ such that
    \begin{enumerate}
        \item [\textup{(1)}] $\sum_{i=1}^m \chi_i = 1,$
        \item [\textup{(2)}] $\supp(\chi_i) \subset B(\xi_i, \epsilon) \text{ for some } \xi_i \in \Xi.$
    \end{enumerate}
\end{lemma}

\begin{proof}
    From Lemma \ref{localftn}, we have a function $\phi_{\xi_0}$ for each $\xi_0\in \Xi$.
    Observe that
    $$\Xi = \bigcup_{\xi_0\in \Xi} \{\xi\in \Xi : \phi^\epsilon_{\xi_0}(\xi) > 0\}.$$
    Since $\Xi$ is compact, we have a finite collection of functions $\phi^\epsilon_{\xi_1}, \cdots, \phi^\epsilon_{\xi_m}$ such that
    $$\Xi = \bigcup_{i=1}^m \{\xi\in \Xi : \phi^\epsilon_{\xi_i}(\xi) > 0\}.$$
    Define a function $\chi_i = \phi^\epsilon_{\xi_i}/\sum_{j=1}^m \phi^\epsilon_{\xi_j}$.
    Clearly, $\chi_i \in C^2(\Xi)$. Now, we prove that $\chi_i \in \mc D^S_x$ for all $x\in S$.
    Fix any $\xi_0\in \Xi$ with $(\xi_0)_x = 0$. From direct computation, we get
    $$\lim_{\xi\rightarrow \xi_0, \xi_x>0} \frac{\textbf{v}_x \cdot \nabla \chi_i(\xi)}{\xi_x} =
    \frac{\textbf{v}_x}{\xi_x}  \cdot \frac{\nabla \phi^\epsilon_{\xi_i}(\xi)(\sum_{j=1}^m \phi^\epsilon_{\xi_j}(\xi)) - \phi^\epsilon_{\xi_i}(\xi)(\sum_{j=1}^m \nabla \phi^\epsilon_{\xi_j}(\xi))}{(\sum_{j=1}^m \phi^\epsilon_{\xi_j}(\xi))^2} = 0.$$
    Therefore, we have $\chi_i \in \mc D^S_x$ for all $x\in S$ from the definition of $\mc D^S_x$ in (\ref{def31}).
\end{proof}

\begin{proof}[Proof of \ref{feller}]
    From Proposition \ref{fellercont}, it is enough to show that the semigroup of this process is strongly continuous.
    Let $P_t$ be a semigroup of the process $\xi_t$.
    For a function $f\in \mc D^S_S$, we have
    $$\lim_{t\rightarrow 0} \|P_t f - f\|_{\infty} = 0$$
    from the martingale equation \eqref{42}. So $P_t$ is strongly continuous on $\mc D^S_S$.
    Therefore, it is enough to show that $\mc D_S^S$ is dense in $C(\Xi)$.
    Fix a function $f\in C(\Xi)$ and $\epsilon >0$.
    Since $\Xi$ is compact, for any $\epsilon >0$, there exists $\delta >0$ such that $\|f(\xi) - f(\xi')\| < \epsilon$ if $\|\xi - \xi'\|_S < \delta$.
    For such $\delta$, take a finite collection of functions $\chi_1, \cdots, \chi_m \in \mc D^S_S$ from Lemma \ref{partition}.
    Take $g = \sum_{i=1}^m f(\xi_i)\chi_i$.
    Then we have for all $\xi\in \Xi$,
    $$|f(\xi) - g(\xi)| \leq \sum_{i=1}^m |\chi_i(\xi) f - \chi_i(\xi) f(\xi_i)| =
    \sum_{\chi(\xi)>0} |\chi_i(\xi) f - \chi_i(\xi) f(\xi_i)| \leq \epsilon.$$
    Since $g\in \mc D^S_S$, we have $\mc D^S_S$ is dense in $C(\Xi)$.
\end{proof}

\section{Condition (D1) for Zero-Range Process} \label{scc}

Throughout this section, we check condition \textbf{(D1)} for the limiting diffusion.

\subsection{Space decomposition of $\Xi$}

In order to handle condition \textbf{(D1)}, we need to decompose the state space $\Xi$ into a finite number of sets indexed
by a finite partial order set. For $\varnothing\subsetneq A\subset S$, let $\mathring{\Xi}_A$ be the subset of $\Xi_A$
defined by
\begin{equation*}
    \mathring{\Xi}_A = \{\xi\in \Xi : \xi_x > 0 \text{ for all } x\in A, \xi_y = 0 \text{ for all } y\in S\setminus A\}.
\end{equation*}
Observe that if $|A|\geq 2$, $\mathring{\Xi}_A$ is the interior of $\Xi_A$.
For convenience, we abbreviate $\mathring{\Xi}_S$ as $\mathring{\Xi}$.

Now, consider a partially ordered set $\mc C$ defined by
$$\mc C = \{A \subset S : \varnothing \subsetneq A \subset S\}$$
with the partial order relation $\subset$ defined by the inclusion relation of sets.

Define a function $ \pi_A(\xi) = (\prod_{x\in A}\xi_x)^{\alpha}$ on $\Xi$.
Define a measure $\lambda_A$ on $\Xi_A$ by
\begin{equation*}
    d\lambda_A= \pi_A^{-1}d m_A,
\end{equation*}
where $dm_A$ denotes the uniform measure on $\Xi_A$.
For convenience, when $A=S$, we abbreviate the function $\pi_S$ as $\pi$ and the measure $\lambda_S$ as $\lambda$.
This $\lambda_A$ is the reference measures on $\Xi_A$ for condition \textbf{(D1)}.
From the absorption property of the limiting diffusion, we have condition \textbf{(D1.1)}.
To check condition \textbf{(D1.2)}, we divide cases. First, for $A \in \mc C$ with $|A| = 1$, the condition is direct from the absorption property.
For $A\in \mc C$ with $|A| \geq 2$, from the recursion property together with the absorption property, it is enough to check the condition for $S\in \mc C$.
The following lemmas show that the condition holds for $S$.

\begin{lemma}
    \label{lem91}
    For $\bold{x} \in \mathring{\Xi}$, let $\delta_\bold{x}$ be a Dirac measure at $\bold{x}$. For $t>0$, let $\delta_\bold{x} P_t|_{\mathring{\Xi}}$ be a
    distribution of $\xi_t$ started at $\bold{x}$ restricted to $\mathring{\Xi}$.
    Then for all $t>0$, $\delta_\bold{x} P_t|_{\mathring{\Xi}}$ is absolutely continuous with respect to uniform measure on $\mathring{\Xi}$.
\end{lemma}

\begin{proof}
    For any convex smooth open subset $\Omega$ of $\mathring{\Xi}$, we consider a killed process of $\xi_t$ which is killed upon its first exit time of $\Omega$.
    Denote the transition kernel of the killed process by $p_\Omega(t,\bold{x},\bold{y})$, $t>0$, $\bold{x},\bold{y}\in \Omega$. To be specific, it satisfies
    $$\int_{\Omega} p_\Omega(t,\bold{x},\bold{y}) f(y) dm(y) = E_\bold{x}[f(\xi_t)1_{\{\xi_s\in \Omega \text{ for all } 0\leq s\leq t \}}],$$
    where $m$ is a uniform measure on $\Omega$.
    From Corollary \ref{heatker2}, we know that the kernel exists.
    Fix any compact subset $K$ of $\mathring{\Xi}$ and $\epsilon>0$. We claim that there exists a constant $C_K>0$ such that
    \begin{equation}\label{111claim}
        \delta_\bold{x} P_t (A) \leq C_K m(A)
    \end{equation}
    for all $A\subset K$.
    Take a convex smooth open subset $\Omega_1$ of $\mathring{\Xi}$ such that $K\cup \{\bold{x}\} \subset \Omega_1$.
    Take a convex smooth open subset $\Omega_2$ of $\mathring{\Xi}$ such that $\bar{\Omega}_1 \subset \Omega_2$ and 
    \begin{equation}
        \label{112cond}
        2\sup_{\xi\in \partial\Omega_2} \phi(\xi) < \inf_{\xi\in \partial \Omega_1} \phi(\xi),
    \end{equation}
    where $\phi(\xi) = (\prod_{x\in S} \xi_x)^{\alpha+1}$.
    For $\bold{z}\in \mathring{\Xi}$, Let $\bb P_\bold{z}$ be a measure given as a solution of \eqref{42}.
    For a Borel set $B\subset \Xi$, denote the first hitting time of $B$ by $\tau_{B}$.
    Fix $A\subset K$. From strong Markov property, we obtain
    \begin{align}
        \label{113eq}
        \bb P_\bold{x}[\xi_t \in A] 
        &= \bb P_\bold{x}[\xi_t \in A, \tau_{\partial \Omega_1} \leq t] + \bb P_\bold{x}[\xi_t \in A, \tau_{\partial \Omega_1} > t] \\
        &\leq \sup_{0\leq s \leq t} \sup_{\bold{z}\in \partial \Omega_1} \bb P_\bold{z}[\xi_{s} \in A] + \int_{A} p_{\Omega_1}(t,\bold{x},\bold{y}) m(d\bold{y}). \nonumber
    \end{align}
    Similarly, we have
    \begin{equation}
        \label{114eq}
        \sup_{0\leq s \leq t} \sup_{\bold{z}\in \partial \Omega_1} \bb P_\bold{z}[\xi_{s} \in A] \leq 
        \sup_{0\leq s \leq t} \sup_{\bold{w}\in \partial \Omega_2} \bb P_\bold{w}[\xi_{s} \in A] + \sup_{0\leq s \leq t}\sup_{\bold{z}\in \partial \Omega_1} \int_{A} p_{\Omega_2}(t,\bold{z},\bold{y}) m(d\bold{y}).
    \end{equation}
    Let $\tau = \tau_{\partial \Omega_1} \wedge \tau_{\partial \Xi}$.
    Since $\phi$ is superharmonic function on $\mathring{\Xi}$, for $\bold{w}\in \partial \Omega_2$, we get
    $$\phi(\bold{w}) \geq \bb E_{\bold{w}}[\phi(\xi_{\tau})] \geq \left(\inf_{\xi\in \partial \Omega_2} \phi(\xi)\right) \bb P_{\bold{w}}[\tau_{\partial \Omega_1} < \tau_{\partial \Xi}].$$
    Then \eqref{112cond} implies that
    $$\sup_{\bold{w}\in \partial \Omega_2} \bb P_{\bold{w}}[\tau_{\partial \Omega_1} < \tau_{\partial \Xi}] < \frac{1}{2}.$$
    Therefore, we get
    \begin{equation}
        \label{115eq}
        \sup_{0\leq s \leq t} \sup_{\bold{w}\in \partial \Omega_2} \bb P_\bold{w}[\xi_{s} \in A] \leq \frac{1}{2} \sup_{0\leq s \leq t} \sup_{\bold{z}\in \partial \Omega_1} \bb P_\bold{z}[\xi_{s} \in A].
    \end{equation}
    Joining \eqref{114eq} and \eqref{115eq} gives
    \begin{equation}
        \label{116eq}
        \sup_{0\leq s \leq t} \sup_{\bold{z}\in \partial \Omega_1} \bb P_\bold{z}[\xi_{s} \in A] \leq
        2\sup_{0\leq s \leq t} \sup_{\bold{w}\in \partial \Omega_2} \int_{A} p_{\Omega_2}(t,\bold{z},\bold{y}) m(d\bold{y}).
    \end{equation}
    Applying Corolloary \ref{heatker2}, there exists a constant $C>0$ such that
    \begin{equation*}
        2\sup_{0\leq s \leq t} \sup_{\bold{w}\in \partial \Omega_2} \int_{A} p_{\Omega_2}(t,\bold{z},\bold{y}) m(d\bold{y}) \leq C m(A).
    \end{equation*}
    Since $\sup_{\bold{y}\in \Omega_1}p_{\Omega_1}(t,\bold{x},\bold{y})<\infty$, \eqref{113eq} and \eqref{116eq} gives \eqref{111claim}.
    Finally, to show that $\delta_\bold{x} P_t|_{\mathring{\Xi}}$ is absolutely continuous with respect to uniform measure on $\mathring{\Xi}$, we take an increasing sequence of compact subsets $K_n$ of $\mathring{\Xi}$ such that $\cup_n K_n = \mathring{\Xi}$.
    For any $A\subset \mathring{\Xi}$ with $m(A) = 0$, decompose it as $A = \cup_n A_n$, $A_n = A \cap K_n\setminus K_{n-1}$ where $K_0 = \emptyset$.
    From \eqref{111claim}, we have
    $$\delta_{\bold{x}}P_t(A_n) = 0.$$
    Therefore, we get $\delta_{\bold{x}}P_t(A) = 0$.
\end{proof}

\begin{lemma}
    For any probability measure $\mu\in \ms P(\Xi)$, let $\mu P_t|_{\mathring{\Xi}}$ be a distibution of $\xi_t$ started at initial distribution $\mu$ restricted to $\mathring{\Xi}$.
    Then for all $t>0$, $\mu P_t|_{\mathring{\Xi}}$ is absolutely continuous with respect to uniform measure on $\mathring{\Xi}$.
\end{lemma}

\begin{proof}
    Using the fact that the process absorbed into the boundary $\partial \Xi$, we have the identity
    $$\mu P_t|_{\mathring{\Xi}} = \int_{\Xi} \delta_{\bold{x}} P_t|_{\mathring{\Xi}} d\mu(\bold{x}) = \int_{\mathring{\Xi}} \delta_{\bold{x}} P_t|_{\mathring{\Xi}} d\mu|_{\mathring{\Xi}}(\bold{x}).$$
    This equation and Lemma \ref{lem91} imply the desired result.
\end{proof}

\subsection{$L^2$ extension of the limiting diffusion} \label{l2extension}

To check conditions \textbf{(D1.3)} and \textbf{(D1.4)}, we need to extend the transition kernel of limiting diffusion to the $L^2$ space.
Recall that it was mentioned that the reversibility assumption is essential to calculate the rate function $I(\mu)$.
This is because the property of $(P_t)$ in Proposition \ref{44}, which is the goal of this section, strongly depends on the reversibility of the underlying Markov process.

We divide cases. For $A \in \mathcal{C}$ with $|A| = 1$, conditions \textbf{(D1.3)} and \textbf{(D1.4)} are directly derived from the absorption property because $(P^A_t)$ in condition \textbf{(D1.3)} turns out to be the identity map on $L^2(\Xi_A, \lambda_A) \simeq \bb R$.
For $A \in \mathcal{C}$ with $|A| \geq 2$, again, due to the recursion property and the absorption property, it is sufficient to check the condition for $S \in \mathcal{C}$. Therefore, the remaining section is devoted to the case where $A = S$.

Throughout the section, $(U_t)_{t\geq 0}$ is the resolvent operator
corresponding to $(P_t)_{t\geq 0}$.
Precisely, it is given by
\begin{equation*}
    P_t f(\xi) = \bb E_{\xi}[f(\xi_t)], \;\; U_t f(\xi) = \int_0^\infty e^{-ts} P_s f(\xi)ds.
\end{equation*}
We consider $\Xi$ as the closed subset of an Euclidean space 
\begin{equation*}
    \bb A \coloneq \{\xi\in \bb R^S : \sum_{x\in S} \xi_x = 1\}
\end{equation*}
with the boundary
$$\partial \Xi = \{\xi\in \Xi: \exists x\in S, \; \xi_x = 0\}.$$
Let $\mathring{\Xi}$ denotes the interior of $\Xi$.

For $k\in N\cup\{\infty\}$, define sets of functions as
\begin{align*}
    C^k(\mathring{\Xi}) &\coloneq \{f:\mathring{\Xi}\rightarrow \bb R : f \text{ is $k$-times continuously differentiable on $\mathring{\Xi}$}\}, \\
    C^k_c(\mathring{\Xi}) &\coloneq \{f:\mathring{\Xi}\rightarrow \bb R : f \in C^{k}(\mathring{\Xi}), f \text{ is compactly supported}\}, \\
    C^k(\mathring{\Xi},\Xi) &\coloneq \{f:\Xi\rightarrow \bb R : f\in C(\Xi), f|_{\mathring{\Xi}}\in C^k(\mathring{\Xi})\} \\
    C^k_c(\mathring{\Xi},\Xi) &\coloneq \{f:\Xi\rightarrow \bb R : f\in C^k(\mathring{\Xi},\Xi), f|_{\mathring{\Xi}} \in C_c(\mathring{\Xi})\}.
\end{align*}

Remark that we omit $k$ when $k=0$.

\begin{defn}
    For $f\in C(\Xi)$, define the domain $\mc D^S$ of $\mf L$ as
    \begin{equation*}
        \mc D^S \coloneq \{f\in C(\Xi) : \lim_{t\rightarrow 0}\frac{P_t f - f}{t} \in C(\Xi) \}.
    \end{equation*}
\end{defn}

Recall $\mc D_S^S$ from Definition \ref{def31}. Note that $\mc D_S^S \subset \mc D^S$, but these sets are not necessarily identical. Therefore, it is unclear how $f \in \mc D^S$ is mapped by the operator $\mf L$.
However, the following lemma ensures that $\mf L f$ can be calculated for $f \in \mc D^S \cap C^2(\mathring{\Xi}, \Xi)$.

\begin{lemma}
    Take $f\in \mc D^S \cap C^2(\mathring{\Xi},\Xi)$. For $\xi\in \mathring{\Xi}$, $\mf L f$ is computed as
    \begin{equation}
        \label{qqqq}
        (\mf L f)(\xi) = \mathbf{b}(\xi)\cdot \nabla f(\xi) + \frac{1}{2}\sum_{x,y\in S} r(x,y) (\partial_x - \partial_y)^2 f(\xi).
    \end{equation}
\end{lemma}

\begin{proof}
    Note that for $f\in \mc D_S^S$, we can calculate $\mf L f$ using \eqref{qqqq}.
    Fix $\xi\in \mathring{\Xi}$ and a open neighborhood $U$ of $\xi$. Take a function $g\in C_c^\infty(\mathring{\Xi},\Xi)$ such that $g=1$ on $U$.
    Then we have $fg \in \mc D_S^S$. To show \eqref{qqqq}, it is enough to show that
    $$\mf L(fg)(\xi) = \mf L(f)(\xi).$$
    Let $h = fg - f$. Take a smaller open neighborhood $V$ of $\xi$ such that $\bar{V}\subset U$.
    Then we take a smooth function $k$ which satisfies $k=0$ on $\bar{V}$ and $k=1$ on $U^{c}$.
    Since $k\in \mc D_S^S$, $\mf L k(\xi) = 0$ by \eqref{lim diff}.
    Let $K = \sup_{\xi \in \Xi}|h(\xi)|$.
    Then we have $ |h|\leq Kk $. Finally, we get
    \begin{align*}
        |\mf Lh(\xi)| = \left|\lim_{n\rightarrow \infty} \frac{P_t h(\xi) - h(\xi)}{t}\right| \leq \lim_{n\rightarrow \infty} \frac{P_t |h|(\xi)}{t}
        \leq \lim_{n\rightarrow \infty} \frac{P_t Kk(\xi)}{t} = 0.
    \end{align*}
\end{proof}

\begin{lemma}
    \label{104}
    For $x,y\in S$, define $v_{x,y}$ as
    \begin{equation} \label{vxy}
        v_{x,y} \coloneq \sqrt{\frac{r(x,y)}{2}}(e_x - e_y).
    \end{equation}
    Then for $f:\mc D^S \cap C^2(\mathring{\Xi},\Xi) \rightarrow \bb R$, we have
    \begin{equation}
        (\mf L f)\pi^{-1} = \sum_{x\neq y}(v_{x,y}\cdot\nabla)(\pi^{-1}v_{x,y}\cdot\nabla f) \; \text{ on } \; \mathring{\Xi}.
    \end{equation}
\end{lemma}

Remark that $v_{x,y} + v_{x,y} = 0$ because of the reversibility assumption on \eqref{reversibility}.

\begin{proof}
    Recall that $e_x$ is a vector in $\mathbb{R}^S$ with $(e_x)_y = \delta_{x,y}, y\in S$.
    For any vector $p$ and $q$ in $\mathbb{R}^S$, we have
    \begin{equation}
        \label{1234321}
        p\cdot\nabla(\pi^{-1}(q\cdot \nabla)f) = (p\cdot\nabla \pi^{-1})(q\cdot \nabla f) + \pi^{-1}(p\cdot \nabla)(q\cdot \nabla)f.
    \end{equation}
    Observe that
    $$(e_x-e_y)\cdot \nabla(\pi^{-1})(\xi) = \alpha\left(\frac{1}{\xi_y}-\frac{1}{\xi_x}\right)\pi^{-1}(\xi).$$
    Now, plug $v_{x,y}$ into $p, q$ in \eqref{1234321} and sum it for all $(x,y)\in S\times S$, $x\neq y$.
\end{proof}

Before we proceed, we refer to a cutoff generator $\mf L^\epsilon$ for $\epsilon >0$ from \cite[Lemma 6.2]{MZRP}.
Let
\begin{equation*}
    \mathbf{b}_\epsilon(\xi) = \alpha \sum_{x\in S} \frac{1}{\epsilon \vee \xi_x} \mathbf{v}_x, \;\; \xi \in \bb A,
\end{equation*}
where $\mathbf{v}_x$ is defined in \eqref{vx}. Thus, for every $\epsilon >0$, $\mathbf{b}_\epsilon : \bb R^n \rightarrow \bb R^n$ is a bounded, continuous vector field which coincides with $\mathbf{b}$ on
\begin{equation*}
    \Lambda_\epsilon = \{ \xi \in \Xi : \min_{x\in S} \xi_x \geq \epsilon\}.
\end{equation*}
Note that for all $\epsilon >0$ and $F\in \mc D_S^S$,
\begin{equation*}
    \mf L^\epsilon F (\xi) = \mathbf{b}_\epsilon(\xi) \cdot \nabla F(\xi) + \frac{1}{2}\sum_{x,y\in S} r(x,y) (\partial_x - \partial_y)^2 F(\xi), \;\; \xi \in \Lambda_\epsilon.
\end{equation*}
Let $(\xi_t)_{t\geq 0}$ be the coordinate maps in the path space $C([0,\infty),\bb A)$.
Let $(\ms F_t)_{t\geq 0}$ be the filtration $\ms F_t = \sigma(\xi_s : s\leq t)$.
Denote by $h_\epsilon$ the exit time from $\Lambda_\epsilon$:
\begin{equation*}
    h_\epsilon \coloneq \inf\{t>0 : \xi_t \notin \Lambda_\epsilon\}, \;\; \epsilon >0.
\end{equation*}
Denote by $\bb Q_\xi^\epsilon$, $\xi \in \Xi$, the unique solution of the $\mf L^\epsilon$-martingale problem starting at $\xi$.
Recall that $\bb P_\xi$ is the unique solution of the $\mf L$-martingale problem starting at $\xi$.
Then \cite[Lemma 6.2]{MZRP} gives that for all $\epsilon >0$, $\bb P_\xi = \bb Q_\xi^\epsilon$ on $\ms F_{h_\epsilon}$.
    Using this cutoff, we can show the following lemma.
\begin{lemma}
    \label{103}
    Let $(U_t)_{t\geq 0}$ be a resolvent operator induced by $(P_t)_{t\geq 0}$. Fix $s>0$. For $f\in \mc D^S \cap C^\infty(\mathring{\Xi},\Xi)$, $f(\partial\Xi)=0$, define $g = U_s f$.
    Then $g\in \mc D^S \cap C^\infty(\mathring{\Xi},\Xi)$.
\end{lemma}

\begin{proof}
    It is enough to prove for $f\geq 0$.
    Consider a Dirichlet problem
    \begin{equation}
        \begin{aligned}
            \label{diriprobn}
                -\mf L^\epsilon u_\epsilon + s u_\epsilon &=\; f \text{ in } \Lambda_\epsilon \\
                u_\epsilon &=\; \frac{f}{s} \text{ on } \partial\Lambda_\epsilon.
        \end{aligned}
    \end{equation}
    From Theorem \ref{probrepn}, there exists a unique solution $u_\epsilon\in C^2(\Lambda_\epsilon)$ of \eqref{diriprobn} and it has a representation
    \begin{equation}
        \label{repn1}
        u_\epsilon(\xi) = \bb E_\xi\left[\int_0^{h_\epsilon} f(\xi_t)e^{-st}dt\right] + \bb E_\xi\left[\frac{f(\xi_{h_\epsilon})}{s}e^{-sh_\epsilon}\right], \;\; \xi \in \Lambda_\epsilon,
    \end{equation}
    where the expectation $\bb E_\xi$ is taken by a unique solution of the martingale problem with initial value $\xi$.
    Since $\bb Q_\xi$ is the unique solution of the martingale problem and $\bb P_\xi = \bb Q_\xi^\epsilon$ on $\ms F_{h_\epsilon}$, we may take the expectation by $\bb P_\xi$.
    
    Now we consider $u_\epsilon$ as a Borel function on $\Xi$ by extending them to 0 on $\Xi \setminus \Lambda_\epsilon$.
    Fix $h\in C_c^{\infty}(\Xi)$.
    For small enough $\epsilon$ such that $\Lambda_\epsilon$ contains $\text{supp}(h)$, we have
    \begin{align*}
        -\int_{\mathring{\Xi}} \mf L h u_\epsilon d\lambda = -\int_{\Lambda_\epsilon} \mf L^\epsilon h u_\epsilon \pi^{-1}dm &= \int_{\Lambda_\epsilon} \sum_{x\neq y} (v_{x,y}\cdot\nabla) (\pi^{-1}(v_{x,y}\cdot\nabla)h) u_\epsilon dm \\
        &= \int_{\Lambda_\epsilon} \sum_{x\neq y} (v_{x,y}\cdot\nabla u_\epsilon)(v_{x,y}\cdot\nabla h)\pi^{-1} dm \\
        &= \int_{\Lambda_\epsilon} \sum_{x\neq y} (v_{x,y}\cdot\nabla) (\pi^{-1}(v_{x,y}\cdot\nabla)u_\epsilon)h dm \\
        &= -\int_{\Lambda_\epsilon} \mf L^\epsilon u_\epsilon h d\lambda = \int_{\Lambda_\epsilon} (f-su_\epsilon) h d\lambda,
    \end{align*}
    where the third and fourth equality is from the divergence theorem.
    Now, consider
    $$u(\xi) = \bb E_\xi\left[\int_0^{\tau} f(\xi_t)e^{-st}dt\right], \;\; \xi \in \Xi,$$
    where $\tau = \inf\{t\geq 0 : \xi_t \in \partial \Xi\}$.
    From the definition of the resolvent operator and the absorption property,
    $$ g(\xi) = \bb E_\xi\left[\int_0^{\infty} f(\xi_t)e^{-st}dt\right] = u(\xi)\;\text{ on }\; \Xi.$$
    Applying the dominated convergence theorem to \eqref{repn1}, we have
    \begin{align*}
        \lim_{n\rightarrow \infty} \int_{\mathring{\Xi}} \mf L h u_\epsilon d\lambda = \int_{\mathring{\Xi}} \mf L h g d\lambda, \;\;
        \lim_{n\rightarrow \infty} \int_{\mathring{\Xi}} (f-su_\epsilon) h d\lambda = \int_{\mathring{\Xi}} (f-sg) h d\lambda.
    \end{align*}
    Therefore, we get
    \begin{align*}
        \int_{\mathring{\Xi}} \pi^{-1} (s-\mf L) h g dm = \int_{\mathring{\Xi}} f h \pi^{-1} dm.
    \end{align*}
    Since $\pi^{-1} (s-\mf L)$ is an elliptic operator acting on $C^\infty_c(\mathring{\Xi})$, the elliptic regularity theorem implies that the smoothness of $g$ comes from the smoothness of $f$.
    Therefore, $g\in \mc D^S \cap C^\infty(\mathring{\Xi},\Xi)$.
\end{proof}

The next lemma shows that non-negative functions in $C(\Xi)$ and $L^p(\lambda)$ can be approximated by non-negative functions in $C^\infty_c(\Xi)$.

\begin{lemma}
    \label{smappr}
    Let $C^{\infty}_{c,+}(\Xi)$ be a set of nonnegative functions in $C^{\infty}_{c}(\Xi)$. For any $f\in L^p(\lambda)$ with $f\geq 0$, there exists a sequence $(f_n)$ in $C^{\infty}_{c,+}(\mathring{\Xi})$ such that $f_n\uparrow f$ in $L^p(\lambda)$.
\end{lemma}

\begin{proof}
    Use denseness of $C_{c}(\mathring{\Xi})$ in $L^p(\lambda)$ to approximate $f$ with nonnegative continuous functions. Then the standard mollification technique gives a smooth approximation.
\end{proof}

Now, we analyze the relation between a boundary decay and the semigroup.

\begin{prop}
    \label{106}
    Let $\mc D_{\alpha+1} \coloneq \{f\in C(\Xi) : |f(\xi)|\leq c\prod_{x\in S}\xi_x^{\alpha + 1}, \exists c>0\}$. Also, define
    $\mc D^S_{\alpha+1} \coloneq \{f\in \mc D^S : |f(\xi)|\leq c\prod_{x\in S}\xi_x^{\alpha + 1}, \exists c>0\}$.
    Then the following holds.
    \begin{enumerate}
        \item[\textup{(1)}] Suppose $f\in \mc D_{\alpha+1}$. Then $P_t f \in \mc D_{\alpha+1}$ for all $t>0$.
        \item[\textup{(2)}] Suppose $f\in \mc D_{\alpha+1}^S \cap C^\infty(\mathring{\Xi},\Xi)$. Then $U_t f \in \mc D_{\alpha+1}^S \cap C^\infty(\mathring{\Xi},\Xi)$ for all $t>0$.
        \item[\textup{(3)}] $P_t:\mc D_{\alpha+1}\rightarrow \mc D_{\alpha+1}$ is a contraction with respect to $L^p(\lambda)$-norm for all $p\geq 1$.
    \end{enumerate}
\end{prop}

\begin{proof}
    To prove (1), let $\phi(\xi) = \prod_{x\in S}\xi_x^{\alpha + 1}$. Then for $\xi \in \mathring{\Xi}$,
    \begin{equation*}
        L\phi(\xi) = \textbf{b}(\xi)\cdot\nabla\phi(\xi) + \frac{1}{2}\sum_{x,y\in S}r(x,y)(\partial_x - \partial_y)^2\phi(\xi)
                   = -\phi(\xi)\sum_{x\neq y} r(x,y)\frac{\alpha + 1}{\xi_x\xi_y} \leq 0.
    \end{equation*}
    Therefore, $\phi$ is superharmonic. So, $P_t$ preserves $\mc D_{\alpha+1}$.

    To prove (2), apply (1) and Lemma \ref{103}.

    To prove (3), we first consider a case when $p=1$. Take a non-negative function $f\in \mc D^S_{\alpha+1} \cap C^\infty(\mathring{\Xi},\Xi)$.
    Take $g = s U_s f$ for some $s>0$.
    From (2), $g\in \mc D^S_{\alpha+1} \cap C^\infty(\mathring{\Xi},\Xi)$.
    
    Fix a positive $\delta>0$. Using Sard's theorem, we can find a $\delta'$ with $0<\delta'<\delta$ such that $\delta'$ is a regular value.
    Therefore, if we consider a compact domain defined by
    $$K_{\delta'} \coloneq \{\xi\in \Xi : g(\xi)\geq \delta'\},$$
    $g^{-1}(\delta')$ is a smooth boundary of it.
    Applying Lemma \ref{104} and the divergence theorem to the domain $K_{\delta'}$,
    \begin{equation*}
        \int_{K_{\delta'}} \mf L g \pi^{-1} dm = \int_{K_{\delta'}} \sum_{x\neq y} (v_{x,y}\cdot \nabla)(\pi^{-1}(v_{x,y} \cdot \nabla)g)dm
        = \int_{\partial K_{\delta'}} v_{x,y}\pi^{-1}(v_{x,y} \cdot \nabla)g \cdot \vec{n} dS,
    \end{equation*}
    where $dS$ is a surface measure on $\partial K_{\delta'}$ and $\vec{n}$ is a unit normal vector on $\partial K_{\delta'}$.
    Since $\delta'$ is regular, $|\nabla g|$ does not vanish on $\partial K_{\delta'}$. Moreover, we know that the gradient is perpendicular to the boundary.
    Therefore, we have $\vec{n} = -\frac{\nabla g}{|\nabla g|}$. Inserting it into the integrand, we obtain
    \begin{align*}
        \int_{\partial K_{\delta'}} v_{x,y}\pi^{-1}(v_{x,y} \cdot \nabla)g \cdot \vec{n} dS 
        = -\int_{\partial K_{\delta'}} \sum_{x\neq y} \frac{(v_{x,y}\cdot\nabla g)^2}{|\nabla g|}\pi^{-1} dS \leq 0.
    \end{align*}
    Since $g - \frac{1}{s} \mf L g = f$, we obtain
    $$\int_{K_{\delta'}} g d\lambda \leq \int_{K_{\delta'}} f d\lambda \leq \int_{\mathring{\Xi}} f d\lambda.$$
    Sending $\delta$ to zero, $\delta'$ going to zero, we obtain that $sU_s$ is a contraction on
    non-negative functions in $\mc D_{\alpha+1}^S \cap C^\infty(\mathring{\Xi},\Xi)$ with respect to $L^1(\lambda)$-norm.

    From Lemma \ref{reslem}.(2), we have
    \begin{equation} \label{ressemi}
        P_t f = \lim_{n\rightarrow \infty} (\frac{n}{t}U_{\frac{n}{t}})^n f.
    \end{equation}
    Using Fatou's lemma, we have
    $$\int_{\mathring{\Xi}} P_t f \pi^{-1} dm = \int_{\mathring{\Xi}} \lim_{n\rightarrow \infty} (\frac{n}{t}U_{\frac{n}{t}})^n f \pi^{-1} dm \leq
    \lim_{n\rightarrow \infty} \int_{\mathring{\Xi}} (\frac{n}{t}U_{\frac{n}{t}})^n f \pi^{-1} dm \leq \int_{\mathring{\Xi}} f \pi^{-1} dm.$$
    So, $P_t$ is a contraction on non-negative functions in $\mc D^S_{\alpha+1}\cap C^\infty(\mathring{\Xi},\Xi)$ with respect to $L^1(\lambda)$-norm.
    Lemma \ref{smappr} gives that $P_t$ is a contraction on $\mc D_{\alpha+1}$.

    For general $p\geq 1$, using Jensen's inequality, we have
    \begin{equation} \label{eq813}
        \int_{\Xi} |P_t f|^p d\lambda \leq \int_{\Xi} P_t |f|^p d\lambda \leq \int_{\Xi} |f|^p d\lambda.
    \end{equation}
    Therefore, $P_t$ is a contraction on $\mc D_{\alpha+1}$ with respect to $L^p(\lambda)$-norm.
\end{proof}

Now we want to define $\bar{P}_t:L^{p}(\lambda)\rightarrow L^{p}(\lambda)$.
The easiest way to do this is using the denseness of $C^\infty_c(\mathring{\Xi},\Xi)$.
For $f\in L^p(\lambda)$, take a sequence $(f_n)$ in $C^\infty_c(\mathring{\Xi},\Xi)$ such that $f_n\rightarrow f$ in $L^p(\lambda)$.
Then we define $\bar{P}_t f$ as $\lim_{n\rightarrow \infty} P_t f_n$. From Proposition \ref{106}.(3), we know that $\bar{P}_t f$ is well-defined and it is a contraction on $L^p(\lambda)$.
The following proposition shows that $\bar{P}_t$ actually coincides with a stochastic kernel $P_t$.

\begin{cor} \label{pbar}
    Fix $t>0$, $p\geq 1$. Take a Borel measurable function $f$ on $\mathring{\Xi}$.
    Let $\bar{f}$ be a Borel measurable function on $\Xi$ such that $\bar{f}|_{\mathring{\Xi}} = f$ and $\bar{f}(\xi) = 0$ for $\xi\in \partial \Xi$.
    If $f\in L^p(\lambda)$, then $\bb E_{\xi}[\bar{f}(\xi_t)]$ is defined for $\lambda$-a.e. $\xi\in \Xi$.
    Moreover, taking $P_t \bar{f}(\xi) = \bb E_{\xi}[\bar{f}(\xi_t)]$, then $\bar{P}_t f = P_t \bar{f}$.
\end{cor}

\begin{proof}
    Using Jensen's inequality as in \eqref{eq813}, we know that it is enough to show for $p=1$.
    Also, it is enough to show for $f\geq 0$. From Lemma \ref{smappr},
    we can take a sequence $(f_n)$ in $C^\infty_{c,+}(\mathring{\Xi},\Xi)$ such that $f_n\uparrow f$ in $L^1(\lambda)$.
    Then we have $P_t f_n(\xi) \uparrow \bb E_{\xi}[f(\xi_t)]$ for all $\xi\in \Xi$.
    On the other hand, we have $P_t f_n \uparrow \bar{P}_t f$ in $L^1(\lambda)$.
    Therefore, by the monotone convergence theorem, we have $$\bar{P}_t f(\xi) = \bb E_{\xi}[f(\xi_t)] \;\;\lambda\text{-a.e.}\;\; \xi\in \Xi.$$
\end{proof}

\begin{cor}
    \label{109}
    For $p\geq 1$, $\bar{P}_t:L^{p}(\lambda)\rightarrow L^{p}(\lambda)$ is a strongly continuous contraction semigroup.
\end{cor}

\begin{proof}
    We already know that $\bar{P}_t$ is a contraction on $L^p(\lambda)$. So it remains to show the strong continuity.
    For $f\in \mc D_{\alpha+1}^S$, we have $P_t f \in \mc D_{\alpha+1}^S$ and $\lim_{t\rightarrow 0} P_t f = f$ pointwisely.
    Precisely, take a positive constant $c$ satisfying
    $$|f(\xi)|\leq c\prod_{x\in S} \xi_x^{\alpha + 1} \;\text{ so }\; |P_t f(\xi)|\leq c\prod_{x\in S} \xi_x^{\alpha + 1} \;\text{ for all }\; t\geq 0.$$
    From the dominated convergence theorem, we have $\lim_{t\rightarrow 0} P_t f = f$ in $L^p(\lambda)$.
    The denseness of $\mc D_{\alpha+1}^S$ in $L^p(\lambda)$ and the fact that $P_t$ is a contraction gives the desired result.
\end{proof}

\begin{cor}
    \label{corelem1}
    Let $L$ be a generator of $\bar{P}_t$ on $L^2(\lambda)$. Then $\mc D_{\alpha+1}^S\cap C^\infty(\mathring{\Xi},\Xi)$ is a core of $L$.
\end{cor}

\begin{proof}   
    To show that $\mc D_{\alpha+1}^S \cap C^\infty(\mathring{\Xi},\Xi)$ is a core of $L$, we apply Lemma \ref{corelem}.
    We know that $C_c^{\infty}(\Xi)$ is dense subset of $L^2(\lambda)$.
    Also, from Proposition \ref{106}, we know that the resolvent preserves $\mc D_{\alpha+1}^S \cap C^\infty(\mathring{\Xi},\Xi)$.
    Therefore, we have $\mc D_{\alpha+1}^S \cap C^\infty(\mathring{\Xi},\Xi)$ is a core of $L$.
\end{proof}

\begin{prop}
    \label{1009}
    For $f\in \mc D_{\alpha+1}^S \cap C^\infty(\mathring{\Xi},\Xi)$, for all $v_{x,y}$ in \eqref{vxy}, $(v_{x,y}\cdot \nabla) f$ is in $L^2(\lambda)$.
    Moreover, for $f,g\in \mc D^S_{\alpha+1}\cap C^\infty(\mathring{\Xi},\Xi)$, the following identity holds.
    \begin{equation}
        \label{dirichlet}
        -\int_{\mathring{\Xi}} f \mf L g d\lambda = \int_{\mathring{\Xi}} \sum_{x\neq y} ((v_{x,y}\cdot \nabla)f) ((v_{x,y}\cdot \nabla)g) d\lambda.
    \end{equation}
\end{prop}

\begin{proof}
    We first consider the case when $f=g\geq 0$.
    From Sard's theorem, a set of singular values of $f$ has Lebesgue measure zero. For a regular value $\delta>0$ of $f$, define a set
    $$K_\delta \coloneq \{\xi\in \Xi : f(\xi)\geq \delta\}.$$
    Then $K_\delta$ is a compact set and $f^{-1}(\delta)$ is a smooth boundary of $K_\delta$. Now, applying the divergence theorem to $K_\delta$,
    \begin{align}
        &-\int_{K_{\delta}} f\mf L f \pi^{-1} dm
        = -\int_{K_{\delta}} \sum_{x\neq y} f(v_{x,y}\cdot \nabla)(\pi^{-1}(v_{x,y} \cdot \nabla)f)dm \nonumber \\
        &= -\sum_{x\neq y} \int_{K_{\delta}} \nabla\cdot (f v_{x,y} (\pi^{-1}(v_{x,y} \cdot \nabla)f))dm 
           +\sum_{x\neq y} \int_{K_{\delta}} (\nabla f) \cdot (v_{x,y} \pi^{-1}(v_{x,y} \cdot \nabla)f)dm \nonumber \\
        &= \sum_{x\neq y}\int_{\partial K_{\delta}} ((v_{x,y} \cdot \nabla)f)^2 \frac{f}{|\nabla f|} \pi^{-1} dS + \sum_{x\neq y}\int_{K_{\delta}} ((v_{x,y} \cdot \nabla)f)^2 d\lambda, \label{810eq1}
    \end{align}
    where $dS$ is a surface measure on $\partial K_{\delta}$.
    Taking $\delta$ to zero, and applying the dominated convergence theorem, we obtain
    $$-\int_{\Xi} f\mf L f \pi^{-1} dm \geq \int_{\Xi} \sum_{x\neq y}((v_{x,y} \cdot \nabla)f)^2 d\lambda.$$

    Now, to prove the equality, we must show that we can choose small enough $\delta$ to make the first term of the right-hand side of \eqref{810eq1} become small enough.
    Assume that there exists $\epsilon>0$ such that for any $\delta>0$, for all regular value $\delta'$ of $f$ with $0<\delta' < \delta$, we have
    $$\int_{\partial K_{\delta'}}\sum_{x\neq y} ((v_{x,y} \cdot \nabla)f)^2 \frac{f}{|\nabla f|} \pi^{-1} dS \geq \epsilon.$$
    Since the set of singular values of $f$ has measure zero, so for almost every $\delta'< \delta,$
    $$\int_{\partial K_{\delta'}}\sum_{x\neq y} ((v_{x,y} \cdot \nabla)f)^2 \frac{1}{|\nabla f|} \pi^{-1} dS \geq \frac{\epsilon}{\delta'}.$$
    Now using the co-area formula, we have
    \begin{align*}
    \int_{\delta'\leq f\leq \delta}\sum_{x\neq y}((v_{x,y}\cdot\nabla) f)^2 \pi^{-1}dm =&\int_{\delta'}^{\delta}\int_{\partial K_{t}}\sum_{x\neq y} ((v_{x,y} \cdot \nabla)f)^2 \frac{1}{|\nabla f|} \pi^{-1} dS dt \\
    \geq & \int_{\delta'}^{\delta} \frac{\epsilon}{t} dt = \epsilon \log\frac{\delta}{\delta'}.
    \end{align*}
    Since the left-hand side is bounded above by
    $$-\int_{\mathring{\Xi}} f\mf L f \pi^{-1} dm \geq \int_{\mathring{\Xi}} \sum_{x\neq y}((v_{x,y} \cdot \nabla)f)^2 d\lambda \geq \int_{\delta'\leq
    f\leq \delta}\sum_{x\neq y}((v_{x,y}\cdot\nabla) f)^2 \pi^{-1}dm,$$
    sending $\delta'\rightarrow 0$ gives a contradiction.
    Therefore, we get \eqref{dirichlet} for $f=g\geq 0$.
    For case $f,g\geq 0$, consider the following identities.
    \begin{align*}
        -\int_{\mathring{\Xi}} f \mf L f d\lambda =& \int_{\mathring{\Xi}} \sum_{x\neq y} ((v_{x,y}\cdot \nabla)f) ((v_{x,y}\cdot \nabla)f) d\lambda.\\
        -\int_{\mathring{\Xi}} g \mf L g d\lambda =& \int_{\mathring{\Xi}} \sum_{x\neq y} ((v_{x,y}\cdot \nabla)g)((v_{x,y}\cdot \nabla)g) d\lambda.\\
        -\int_{\mathring{\Xi}} (f+g) \mf L (f+g) d\lambda =& 
        \int_{\mathring{\Xi}} \sum_{x\neq y} ((v_{x,y}\cdot \nabla)(f+g)) ((v_{x,y}\cdot \nabla)(f+g)) d\lambda.
    \end{align*}
    Subtract first two equations from the last one gives
    \begin{equation}
    \label{diricletineq}
        -\int_{\mathring{\Xi}} f \mf L g d\lambda -\int_{\mathring{\Xi}} g \mf L f d\lambda =
        2\int_{\mathring{\Xi}} \sum_{x\neq y} ((v_{x,y}\cdot \nabla)f) ((v_{x,y}\cdot \nabla)g) d\lambda.
    \end{equation}
    But exactly the same calculation from the above gives
    \begin{align*}
        &-\int_{K_{\delta}} g\mf L f \pi^{-1} dm 
        = -\int_{K_{\delta}} \sum_{x\neq y} g(v_{x,y}\cdot \nabla)(\pi^{-1}(v_{x,y} \cdot \nabla)f)dm \\
        &= \int_{\partial K_{\delta}} \sum_{x\neq y}((v_{x,y} \cdot \nabla)f)^2 \frac{g}{|\nabla f|} \pi^{-1} dS + \int_{K_{\delta}} \sum_{x\neq y}((v_{x,y} \cdot \nabla)g)((v_{x,y} \cdot \nabla)f) d\lambda \\
        &\leq \int_{K_{\delta}} \sum_{x\neq y}((v_{x,y} \cdot \nabla)g)((v_{x,y} \cdot \nabla)f) d\lambda.
    \end{align*}
    So sending $\delta\rightarrow 0$ gives
    $$-\int_{\mathring{\Xi}} g\mf L f d\lambda \leq \int_{\mathring{\Xi}} \sum_{x\neq y}((v_{x,y} \cdot \nabla)g)((v_{x,y} \cdot \nabla)f) d\lambda.$$
    Therefore, from \eqref{diricletineq}, we obtain the desired result for $f,g\geq 0$.

    Finally, for general $f,g\in \mc D^S_{\alpha+1}\cap C^\infty(\mathring{\Xi},\Xi)$, we add a function $C\prod_{x\in S}\xi_x^{\alpha + 1}$
    with positive $C$ to make them non-negative. Then use \eqref{dirichlet} for non-negative functions to obtain the desired result.
\end{proof}

    Using Proposition \ref{1009}, we can finally prove the following proposition which gives condition \textbf{(D1.3)} and \textbf{(D1.4)}.

\begin{prop}
    \label{44}
        The semigroup $(\bar{P}_t)_{t\geq 0}$ on $L^2(\lambda)$ is a self-adjoint strongly continuous contraction.
\end{prop}

\begin{proof}
For $f,g\in \mc D^S_{\alpha+1}\cap C^\infty(\mathring{\Xi},\Xi)$, we get
\begin{equation*}
    -\int_{\mathring{\Xi}} f \mf L g d\lambda = \int_{\mathring{\Xi}} \sum_{x\neq y} ((v_{x,y}\cdot \nabla)f) ((v_{x,y}\cdot \nabla)g) d\lambda = -\int_{\mathring{\Xi}} g \mf L f d\lambda.
\end{equation*}
Fix $t>0$. For $p,q\in \mc D^S_{\alpha+1}\cap C^\infty(\mathring{\Xi},\Xi)$, take $f = tU_t p$ and $g = tU_t q$.
From Proposition \ref{106}, we have $f,g\in \mc D^S_{\alpha+1}\cap C^\infty(\mathring{\Xi},\Xi)$.
Therefore, we have 
$$-\int_{\mathring{\Xi}} f \mf L g d\lambda = -\int_{\mathring{\Xi}} f \mf L g d\lambda.$$
Since $f - \frac{1}{t}\mf L f = p$ and $g - \frac{1}{t}\mf L g = q$, we get
$$\int_{\mathring{\Xi}} p U_t g d\lambda = \int_{\mathring{\Xi}} q U_t p d\lambda.$$
Using \eqref{ressemi}, we have
$$\int_{\mathring{\Xi}} p P_t g d\lambda = \int_{\mathring{\Xi}} q P_t p d\lambda.$$
The denseness of $C^\infty_c(\mathring{\Xi},\Xi)$ in $L^2(\lambda)$ gives the desired result.
\end{proof}

\section{Condition \textbf{(D1*)} for Zero-Range Process} \label{p39}

Throughout this section, we use the following notation.
For a coordinate $\eta \in \mc H_N$ and $\varnothing\subsetneq A\subset S$, let $\eta_A$ is a $A$-coordinates of $\eta$. Precisely, $\eta_A \in \bb Z^A$, $\eta_A(x) = \eta(x)$ for all $x\in A$.
Let $|\eta_A|$ be a sum of coordinates of $\eta_A$, that is,
$$|\eta_A| = \sum_{x\in A} \eta_A(x).$$

For $\varnothing\subsetneq A\subset S$, let $L^A$ be a generator of the induced semigroup by $(P_t)$ on $L^2(\lambda_A)$ in Proposition \ref{44}.
Define an energy functional $Q^A$ on $L^2(\lambda_A)$ as
$$ Q^A(f) = \int_{\Xi_A} f (-L^A) f d\lambda_A.$$
The following notion is needed to handle a graph norm.

\begin{defn}
    [$U$-denseness]
    Suppose a metric space $X$ and a functional $U$ on $X$ is given. A subset $A$ of $X$ is said to be $U$-dense if for all $x\in X$,
    there exists a sequence $(x_n)$ in $A$ such that $x_n\rightarrow x$ and $U(x_n)\rightarrow U(x)$.
\end{defn}

The following lemma says that we can choose $\mc D_{A,0}$ as $C_c^{\infty}(\Xi_A)$ in condition \textbf{(D1*.1)}.

\begin{lemma}
    \label{corelem2}
    For $\varnothing\subsetneq A\subset S$, let $\mc D(\sqrt{-L^A})$ be a domain of $\sqrt{-L^A}$. Then $C_c^{\infty}(\Xi_A)$ is $Q^A$-dense in $\mc D(\sqrt{-L^A})$.
\end{lemma}

\begin{proof}
    From the absorption property, it is enough to show for the case $A=S$.
    Assign $\mc D(\sqrt{-L})$ a norm $\|f\|_{Dom} = \|f\|_{2} + \|\sqrt{-L}f\|_{2}$.
    Remark that a set is $Q$-dense in $\mc D(\sqrt{-L})$ if and only if it is dense in $\mc D(\sqrt{-L})$ with respect to $\|\cdot\|_{Dom}$.
    We have already shown that $\mc D_{\alpha+1}^S\cap C^\infty(\mathring{\Xi},\Xi)$ is a core of $L$.
    Since $\mc D(\sqrt{-L})\subset \mc D(L)$, $\mc D_{\alpha+1}^S\cap C^\infty(\mathring{\Xi},\Xi)$ is dense in $\mc D(\sqrt{-L})$ with respect to $\|\cdot\|_{Dom}$.
    Therefore, it is enough to show that $C_c^{\infty}(\Xi)$ is dense in $\mc D_{\alpha+1}^S\cap C^\infty(\mathring{\Xi},\Xi)$ with respect to $\|\cdot\|_{Dom}$.
    We first define a smooth function $f:\bb R \rightarrow \bb R$ with the following properties.
    \begin{enumerate}
        \item $f(x) = 0$ for $x\leq 1/2$ and $f(x) = 1$ for $x\geq 1$.
        \item $f$ is increasing.
    \end{enumerate}
    Then define $f_n(x) = f(nx)$. Then it satisfies
    \begin{equation} \label{102eq1}
        \sup_{x} xf'_n(x) = \sup_{x} xf'(x) < \infty.
    \end{equation}
    Now, for $u\in \mc D^S_{\alpha+1}\cap C^\infty(\mathring{\Xi},\Xi)$, define $u_n(\xi) = u(\xi) \prod_{x\in S} f_n(\xi_x)$. We claim that $u_n$ converges to $u$ in $\|\cdot\|_{Dom}$.
    It is enough to show that $\lim_{n\rightarrow \infty} Q(u-u_n) = 0$.
    Let $g_n(\xi) = 1 - \prod_{x\in S} f_n(\xi_x)$.
    Then
    \begin{align*}
    Q(u-u_n)
    = \int_{\mathring{\Xi}} \sum_{x\neq y} ((v_{x,y}\cdot \nabla)(ug_n)) ((v_{x,y}\cdot \nabla)(ug_n)) d\lambda = A_n + B_n + C_n,
    \end{align*}
    where $A_n, B_n, C_n$ is defined as
    \begin{align*}
        A_n &= \int_{\mathring{\Xi}} \sum_{x\neq y} ((v_{x,y}\cdot \nabla)u)^2g^2_n \pi^{-1}dm, \\
        B_n &= \int_{\mathring{\Xi}} \sum_{x\neq y} 2((v_{x,y}\cdot \nabla)u)((v_{x,y}\cdot \nabla)g_n)ug_n \pi^{-1}dm, \\
        C_n &= \int_{\mathring{\Xi}} \sum_{x\neq y} u^2((v_{x,y}\cdot \nabla)g_n)^2 \pi^{-1}dm.
    \end{align*}
    The $A_n$ converges to 0 from the dominated convergence theorem.
    Also, by \eqref{102eq1} and the fact that $u \in \mc D_{\alpha+1}$, $(u\pi^{-1})\nabla g_n$ is uniformly bounded.
    Therefore, again by the dominated convergence theorem, $B_n$ converges to 0. For $C_n$,
    $$u^2((v_{x,y}\cdot \nabla)g_n)^2 \pi^{-1} = (u\pi^{-1}(v_{x,y}\cdot \nabla)g_n)^2 \pi \leq c 1_{\{\xi_x \leq \frac{1}{n}, \forall x\in S\}}$$
    for some positive $c$. Therefore, $C_n$ converges to 0.
\end{proof}

The remaining part of this section is devoted to checking condition \textbf{(D1*.2)}.

\begin{prop} \label{123}
    For $\varnothing\subsetneq A\subset S$, consider a nonnegative function $v\in C^\infty_c(\mathring{\Xi}_A,\Xi_A)$.
    Then there exists a sequence of functions $V_N : \mc H_N \rightarrow \bb R_+$ satisfying
    \begin{enumerate}
        \item A sequence of measure $\iota_{N,*}(V_N^2 d\rho_N)$ weakly converges to $v^2 d\lambda_A$,
        \item $\lim_{N\rightarrow \infty} N^2\int_{\mc H_N} V_N (-\ms L_N V_N) d\rho_N = \int_{\Xi} v(-L^A v) d\lambda_A$,
    \end{enumerate}
    where $\rho_N$ is a stationary measure of zero-range process on $\mc H_N$ and $\iota_{N,*}$ is a pushforward map of measures on $\mc H_N$ to measures on $\Xi_N$.
\end{prop}

To find a desired sequence of functions, we need the following auxiliary function $\psi$.
Fix a $0<\gamma<1$ so that $2\gamma < (1-\gamma)(1-\alpha)$. It will be used in \eqref{eqeq400}.
Let $B$ be a complementary set of $A$ in $S$.
Define $\psi_N : \bb R \rightarrow \bb R$ as
$$\psi_N(x) = \begin{cases}
    1 & x\leq N^{1-\gamma} \\
    N^{\gamma -1} x & N^{1-\gamma} < x \leq 2N^{1-\gamma} \\
    0 & x > 2N^{1-\gamma}.
\end{cases}$$
Define $\Psi_N : \mc H_N \rightarrow \bb R$ as
$$\Psi_N(\eta) = \psi_N(|\eta_B|).$$
Then we choose $W_N: \mc H_N\rightarrow \bb R$ as
\begin{equation}
    \label{defv}
    W_N(\eta) = \Psi_N(\eta) \ms E_A v(\eta/N),
\end{equation}
where $\ms E_A$ is a extension map of $v$ on $\Xi_A$ to $\Xi$ which is presented in Lemma \ref{harmext}.

Since $v\in C^\infty_c(\mathring{\Xi}_A,\Xi_A)$, there exists $\epsilon>0$ such that $v(\xi) = 0$ for all $\xi\in \Xi$ with $\xi_x \leq \epsilon$ for some $x\in A$.
We want some regularity lemma on the support of $W_N$.

\begin{lemma}
    \label{reglem}
    Fix a positive real $\epsilon'$ such that $\epsilon' < \epsilon$. Then there exists a natural number $N_0$ such that for all 
    $N\geq N_0$, $W_N(\eta)\neq 0$ implies $\eta_x > \epsilon' N + 1$ for all $x\in A$.
\end{lemma}

\begin{proof}
    From the definition of $W_N$, we have $W_N(\eta) \neq 0$ if and only if $\Psi_N(\eta) \neq 0$ and $\ms E_Av(\eta/N) \neq 0$.
    Therefore $|\eta_B| < 2N^{1-\gamma}$ and $\ms E_Av(\eta/N) \neq 0$.
    Since $\ms E_Av(\eta/N) \neq 0$, we have
    $$v(\gamma_A(\eta/N)) \neq 0,$$
    so $\gamma_A(\eta/N)>\epsilon$ for all $x\in A$. Then for $x\in A$,
    $$[\gamma_A(\eta/N)]_x = \eta_x/N + \sum_{y\in B} u^A_{x}(y)\eta_y/N < \eta_x/N + 2N^{1-\gamma}/N = \eta_x/N + 2N^{-\gamma}.$$
    Therefore, we have
    $$\epsilon < \eta_x/N + 2N^{-\gamma}.$$
    Thus, for sufficiently large $N$, we have $\eta_x > \epsilon' N + 1$.
\end{proof}

Before proceeding to the proof of Proposition \ref{123}, we define some notations.
We define a certain Riemann integration on $\Xi_A$. To do so, let $\mc H_{N}^A$ be a subset of $\mc H_N$ such that
$$\mc H_{N}^A = \{\eta \in \mc H_N : \eta_x = 0 \text{ for all } x\in B\}.$$
Given a bounded continuous function $f:\Xi_A \rightarrow \bb R$, define a Riemann integral of $R(N,f)$ as follows.
\begin{equation*}
    R(N,f) = \frac{1}{|\mc H^A_N|} \sum_{\eta \in \mc H^A_N} f(\iota_N(\eta)).
\end{equation*}
Then we have
\begin{equation*}
    \lim_{N\rightarrow \infty} R(N,f) = \int_{\Xi_A} f dm_A,
\end{equation*}
where $m_A$ is a uniform measure on $\Xi_A$.

Let $\mc H_N^{B,*}$ be a subset of ${\bb N}^{|B|+1}$ which
consists of a points whose coordinate sum is $N$. We understand that this space is obtained from $\mc H_N$ by merging all $A$-coordinates. As in \eqref{quotientspace}, we interpret an element $\zeta\in \mc H_N^{B,*}$
as a set of points in $\mc H_N$ whose $B$ coordinate is $\zeta_B$.

For any coordinate $\eta = (\eta_1,...,\eta_k)$ of natural number, let $a(\eta) = \prod_{i=1}^k a(\eta_i)$ where $a$ is defined in
\eqref{rec1}. Also, let $\min(\eta)$ be a minimus of coordinates of $\eta$
, that is, $\min\{\eta_i : i=1,...,k\}$.

Finally, Let $\iota_A: \bb R^A \rightarrow \bb R^S$ be a map defined as
$$\iota_A(x) = \begin{cases}
    x & x\in A \\
    0 & x\in S\setminus A.
\end{cases}$$

The following two lemmas are the key ingredients to prove Proposition \ref{123}.

\begin{lemma}
    \label{1205}
    Let $\Gamma = \Gamma(\alpha)$ where $\Gamma(\alpha)$ is a constant from \eqref{gammaalpha}. Let $Z_S$ be a constant from \eqref{normalizationconst}.
    For any continuous function $f:\Xi \rightarrow \bb R$, we have
    \begin{equation*}
        \int_{\mc H_N} f(\eta/N) W_N^2(\eta) d\rho_N(\eta) = \frac{N^\alpha N^{|A|-1} \Gamma^{|B|}}{Z_S N^{|A|\alpha}(|A|-1)!} \left(\int_{\mathring{\Xi}_A} f v^2 d\lambda_A +o_N(1)\right).
    \end{equation*}
\end{lemma}

\begin{proof}
    Direct computation gives
    \begin{align}
        \sum_{\eta\in \mc H_N} f\left(\frac{\eta}{N}\right) W_N^2(\eta) \rho_N(\eta) =& \sum_{\zeta \in \mc H_N^{B,*}}
        \sum_{\eta \in \zeta} f\left(\frac{\eta}{N}\right) W_N^2(\eta) \rho_N(\eta) 
        = \frac{N^\alpha}{Z_{N,S}}\sum_{\zeta \in \mc H_N^{B,*}} \sum_{\eta \in \zeta} f\left(\frac{\eta}{N}\right) W_N^2(\eta) \frac{1}{a(\eta)} \nonumber \\
        =& \frac{N^\alpha}{Z_{N,S}}\sum_{\zeta \in \mc H_N^{B,*}} \Psi_N^2(\zeta_B) \frac{1}{a(\zeta_B)} \sum_{\eta \in \zeta} \frac{1}{a(\eta_A)} f\left(\frac{\eta}{N}\right)\ms E_A v^2\left(\frac{\eta}{N}\right) \nonumber \\
        =& \frac{N^\alpha}{Z_{N,S}}\sum_{\substack{\zeta \in \mc H_N^{B,*} \\ |\zeta_B|\leq 2 N^{1-\gamma}}} \Psi_N^2(\zeta_B) \frac{1}{a(\zeta_B)} \sum_{\eta \in \zeta} \frac{1}{a(\eta_A)} f\left(\frac{\eta}{N}\right) \ms E_A v^2\left(\frac{\eta}{N}\right). \label{lemeq1}
    \end{align}

    From Lemma \ref{reglem}, for large enough $N$ and for all $\ms E_A v(\eta)\neq 0$ with $|\eta_B|\leq 2 N^{1-\gamma}$, we have
    $\min(\eta_A) > \epsilon' N$.
    Therefore, for large enough $N$ and $|\zeta_B|\leq 2N^{1-\gamma}$, we get
    \begin{equation}
        \label{lemeq2}
        \sum_{\eta\in \zeta} \frac{1}{a(\eta_A)} f\left(\frac{\eta}{N}\right) \ms E_A v^2\left(\frac{\eta}{N}\right) = \sum_{\eta\in \zeta} \frac{1}{a(\eta_A)} f\left(\frac{\eta}{N}\right) \ms E_A v^2\left(\frac{\eta}{N}\right) 1_{\{m(\eta_A) > \epsilon' N\}}.
    \end{equation}
    Observe that $$\ms E_A v^2\left(\frac{\eta}{N}\right) = v^2\left(\gamma_A\left(\frac{\eta}{N}\right)\right) 
    = v^2\left(\iota_A\left(\frac{\eta_A}{N-|\zeta_B|}\right)\right) + \left[v^2\left(\gamma_A\left(\frac{\eta}{N}\right)\right) - v^2\left(\iota_A\left(\frac{\eta_A}{N-|\zeta_B|}\right)\right)\right].$$
    Then for all $x\in A$,
    $$ \left[\gamma_A\left(\frac{\eta}{N}\right) - \iota_A\left(\frac{\eta_A}{N-|\zeta_B|}\right)\right]_x = \frac{\eta_x}{N} + \frac{\sum_{y\in B}u_x^A(y)\eta_y}{N} - \frac{\eta_x}{N-|\zeta_B|} = O(N^{-\gamma}).$$
    Therefore, we have
    \begin{equation}
        \label{lemeq3}
        \ms E_A v^2\left(\frac{\eta}{N}\right) = v^2\left(\iota_A\left(\frac{\eta_A}{N-|\zeta_B|}\right)\right) + O(N^{-\gamma}).
    \end{equation}
    The same argument gives
    \begin{equation}
        \label{lemeq33}
        f\left(\frac{\eta}{N}\right) = f\left(\iota_A\left(\frac{\eta_A}{N-|\zeta_B|}\right)\right) + o_N(1).
    \end{equation}
    Putting \eqref{lemeq3},\eqref{lemeq33} into \eqref{lemeq2} gives
    \begin{align}
        \label{lemeq4}
        \sum_{\eta\in \zeta} \frac{1}{a(\eta_A)}f\left(\frac{\eta}{N}\right)\ms E_A v^2\left(\frac{\eta}{N}\right) 1_{\{m(\eta_A) > \epsilon' N\}} =
        &\sum_{\eta\in \zeta} \frac{1}{a(\eta_A)}(fv^2)\left(\iota_A\left(\frac{\eta_A}{N-|\zeta_B|}\right)\right) 1_{\{m(\eta_A) > \epsilon' N\}} \nonumber \\
        +& \sum_{\eta\in \zeta} \frac{1}{a(\eta_A)}o_N(1) 1_{\{m(\eta_A) > \epsilon' N\}}. 
    \end{align}
    Note that if $v(\iota_A\left(\frac{\eta_A}{N-|\zeta_B|}\right)) \neq 0$, then $\min(\eta_A) > \epsilon(N-|\zeta_B|)$, so $\min(\eta) > \epsilon' N$ for large enough $N$. Therefore,
    \begin{align}
    \label{lemeq5}
    \sum_{\eta\in \zeta} \frac{1}{a(\eta_A)}(fv^2)\left(\iota_A\left(\frac{\eta_A}{N-|\zeta_B|}\right)\right) 1_{\{m(\eta_A) > \epsilon' N\}}
    =& \frac{|\zeta|}{(N-|\zeta_B|)^{|A|\alpha}}R(N-|\zeta_B|,\pi^{-1}fv^2) \\
    =& \frac{|\zeta|}{(N-|\zeta_B|)^{|A|\alpha}}\left(\int_{\mathring{\Xi}_A} fv^2 d\lambda_A + o_N(1)\right). \nonumber
    \end{align}
    Remark that $|\zeta| = \frac{(N-|\zeta_B|)^{|A|-1}}{N^{|A|\alpha}(|A|-1)!}(1+o_N(1))$. Since $|\zeta_B|\leq 2N^{1-\gamma}$, (\ref{lemeq5}) becomes
    \begin{align}
        \label{lemeq6}
        \sum_{\eta\in \zeta} \frac{1}{a(\eta_A)}(fv^2)\left(\iota_A\left(\frac{\eta_A}{N-|\zeta_B|}\right)\right) 1_{\{m(\eta_A) > \epsilon' N\}}
        = \frac{N^{|A|-1}}{N^{|A|\alpha}(|A|-1)!}(\int_{\mathring{\Xi}_A} fv^2 d\lambda_A+o_N(1)).
    \end{align}
    Now we consider the last term of (\ref{lemeq4}). Similar computation shows that
    \begin{equation} \label{lemeqeq}
        \sum_{\eta\in \zeta} \frac{1}{a(\eta_A)}o_N(1) 1_{\{m(\eta_A) > \epsilon' N\}} = \frac{N^{|A|-1}}{N^{|A|\alpha}(|A|-1)!}R(N-|\zeta_B|,\pi^{-1}1_U)o_N(1)
    \end{equation}
    where $U$ is a open subset of $\Xi_A$ such that $U \coloneq \{\xi \in \Xi_A : \xi_x > \epsilon' \text{ for all } x\in A\}$.
    Therefore, \eqref{lemeq4} becomes
    \begin{equation}
        \label{lemeq7}
        \sum_{\eta\in \zeta} \frac{1}{a(\eta_A)} f(\frac{\eta}{N}) \ms E_A v^2\left(\frac{\eta}{N}\right) 1_{\{m(\eta_A) > \epsilon' N\}} = \frac{N^{|A|-1}}{N^{|A|\alpha}(|A|-1)!}\left(\int_{\mathring{\Xi}_A} fv^2 d\lambda_A+o_N(1)\right).
    \end{equation}
    Putting \eqref{lemeq7} into \eqref{lemeq1}, we have
    \begin{align}
        \label{lemeq8}
        \sum_{\eta\in \mc H_N} f(\frac{\eta}{N})W_N^2(\eta) \rho_N(\eta) &= \frac{N^\alpha N^{|A|-1}}{Z_{N,S}N^{|A|\alpha}(|A|-1)!}\sum_{\substack{\zeta \in \mc H_N^{B,*}, \\ |\zeta_B|
        \leq 2 N^{1-\gamma}}} \Psi_N^2(\zeta_B) \frac{1}{a(\zeta_B)}(1+o_N(1))
    \end{align}
    Finally, we have
    \begin{equation}
        \label{lemeq9}
        \sum_{\substack{\zeta \in \mc H_N^{B,*} \\ |\zeta_B| \leq 2 N^{1-\gamma}}} \Psi_N^2(\zeta_B) \frac{1}{a(\zeta_B)} = \Gamma^{|B|}(1+o_N(1)).
    \end{equation}
    Putting \eqref{lemeq9} into \eqref{lemeq8}, we have the desired result.
\end{proof}

\begin{lemma} Suppose that $Q^A(v)<0$. The following equation holds.
    \begin{equation*}
        -N^2\int_{H_N} W_N (\ms L_N W_N) d\rho_N = \frac{N^\alpha N^{|A|-1} \Gamma^{|B|}}{Z_S N^{|A|\alpha}(|A|-1)!}(Q^A(v)+o_N(1))
    \end{equation*}
\end{lemma}

\begin{proof}
    The left-hand side is calculated as follows.
    \begin{equation}
        \label{eqeq1}
        -N^2\int_{H_N} W_N (\ms L_N W_N) d\rho_N = \frac{N^2}{2} \sum_{\substack{\eta\in \mc H_N \\ x,y\in S}} \rho_N(\eta)g(\eta_x)r(x,y)(W_N(\eta^{x,y}) - W_N(\eta))^2
    \end{equation}
    For any $\eta\in \mc H_N$, let $(\eta,\eta^{x,y})$ be a edge with a starting point $\eta$ and an ending point $\eta^{x,y}$ with direction $(x,y)$.
    Now, give an order on $S$. If an edge $(\eta,\eta^{x,y})$ is given, we say $(x,y)$ is a positive direction if $x<y$.          
    For simplicity, we will use $e$ to refer to an edge. Then $e_0$ and $e_1$ will refer to a starting point and an ending point of $e$ respectively.
    Also, we $dir(e)$ will refer to a direction of $e$. Finally, $\mc H_N^{x,y}$ will refer to a set of possible edges in $\mc H_N$ with direction $(x,y)$.
    
    For a edge $e\in \mc H_N^{x,y}$, define $\rho_N(e)$ as
    $$\rho_N(e) = \rho_N(\eta)g(\eta_x) = \rho_N(\eta^{x,y})g((\eta^{x,y})_y).$$
    Then \eqref{eqeq1} becomes
    \begin{align}
        \label{eqeq2}
        N^2 \sum_{x<y} \sum_{e\in \mc H_N^{x,y}}\rho_N(e)r(x,y)(W_N(e_1) - W_N(e_0))^2.
    \end{align}
    Now, we consider a bijection between $\mc H_N^{x,y}$ and $\mc H_{N-1}$.
    To describe the explicit bijection, all we need to do is assign a $z$-th coordinate to $e\in \mc H_N^{x,y}$ for all $z\in S$.
    For $z\in S\setminus \{x\}$, we define $[e]_z = [e_0]_z$. For $z=x$, we define $[e]_z = [e_0]_x - 1$.
    Then we can treat $e$ as an element of $\mc H_{N-1}$. Moreover, $\rho_N(e)$ is calculated as
    $$\rho_{N}(e) = \rho_N(\eta)g(\eta_x) = \frac{N^\alpha}{Z_{N,S}}\frac{1}{a(\eta)}g(\eta_x) = \frac{N^\alpha}{Z_{N,S}}\frac{1}{a(e)},$$
    where $a(e)$ is calculated as in \eqref{rec1} using $\mc H_{N-1}$ coordinate.
    Then \eqref{eqeq2} becomes
    \begin{align}
        \label{eqeq3}
        \sum_{x<y} \frac{N^\alpha}{Z_{N,S}}  \sum_{e\in \mc H_N^{x,y}}\frac{1}{a(e)}r(x,y)N^2(W_N(e_1)-W_N(e_0))^2.
    \end{align}
    We fix $x<y$.
    As in the proof of Lemma \ref{1205}, we consider $\mc H^{B,*}_{N-1}$.
    So, we consider an element $\zeta\in \mc H^{B,*}_{N-1}$ as a set of edges in $\mc H_{N-1}$ with $B$ coordinate $\zeta_B$.
    Then the inner summation of \eqref{eqeq3} becomes
    \begin{align}
        \label{eqeq40}
        \frac{N^\alpha}{Z_{N,S}} \sum_{\zeta\in \mc H_{N-1}^{B,*}}\frac{1}{a(\zeta_B)}\sum_{e\in \zeta}\frac{1}{a(e_A)}r(x,y)N^2(W_N(e_1)-W_N(e_0))^2
    \end{align}
    Observe that if $W_N(\eta)$ or $W_N(\eta^{x,y})$ is not zero, from Lemma $\ref{reglem}$, we have
    $$\min(\min(\eta_A), \min((\eta^{x,y})_A))>\epsilon'N+1.$$
    Therefore, we have
    $$\min(e_A)>\epsilon'N.$$
    So we only need to consider edges $e$ with $m(e_A)>\epsilon'N$ in the sum of \eqref{eqeq40}.
    Now we say an edge $e\in \mc H_N^{x,y}$ is good if $|(e_0)_B|\vee |(e_1)_B|<N^{1-\gamma}$. Also, we say $e$ is bad if $|(e_0)_B|\vee|(e_1)_B|\geq N^{1-\gamma}$ and $|(e_0)_B|\wedge|(e_1)_B|\leq 2N^{1-\gamma}$.
    Otherwise, we say $e$ is void.
    Then we define that $\zeta$ is good, bad, or void if for all $e\in \zeta$, $e \in \zeta$ is good, bad, or void, respectively.
    Since every element $e$ in $\zeta$ shares $B$ coordinate, we can say that every $\zeta\in \mc H_{N-1}^{B,*}$ is good, bad or void.
    Remark that if $\zeta$ is void, then $W_N(e_1)= W_N(e_0) = 0$ for all $e\in \zeta$.
    Therefore, we need only need to consider good and bad $\zeta$ in the summation of \eqref{eqeq40}.

    For a bad edge $e=(\eta,\eta^{x,y})$, we have
    \begin{align*}
        (W_N(\eta^{x,y}) - W_N(\eta))^2 =& \left(\Psi_N(\eta^{x,y})\ms E_A v\left(\frac{\eta^{x,y}}{N}\right) - \Psi_N(\eta)\ms E_A v\left(\frac{\eta}{N}\right)\right)^2 \\
                                        \leq&\; c \left[\left(\ms E_A v\left(\frac{\eta^{x,y}}{N}\right) - \ms E_A v\left(\frac{\eta}{N}\right)\right)^2
                                        + (\Psi_N(\eta^{x,y})-\Psi_N(\eta))^2\right].
    \end{align*}
    for some positive constant $c$ irrelevant to $N$ and $e$. Thus for a bad $\zeta$,
    \begin{align}
        \label{eqeq4}
        &\sum_{ e\in \zeta} \frac{1}{a(e_A)}r(x,y)N^2(W_N(e_1)-W_N(e_0))^2 \nonumber \\
        &\leq c \sum_{ e \in \zeta} \frac{r(x,y)}{a(e_A)}\bigg[N^2\left(\ms E_A v\left(\frac{\eta^{x,y}}{N}\right) - \ms E_A v\left(\frac{\eta}{N}\right)\right)^2
        + N^2(\Psi_N(\eta^{x,y})-\Psi_N(\eta))^2 \bigg]1_{\{m(e_A)>\epsilon'N\}} \nonumber \\
        &= A_N + B_N,
    \end{align}
    where $A_N$ and $B_N$ are defined as
    \begin{align}
        A_N &= c \sum_{ e \in \zeta} \frac{r(x,y)}{a(e_A)}N^2\left(\ms E_A v\left(\frac{\eta^{x,y}}{N}\right) - \ms E_A v\left(\frac{\eta}{N}\right)\right)^2 1_{\{m(e_A)>\epsilon'N\}}, \\
        B_N &= c \sum_{ e \in \zeta} \frac{r(x,y)}{a(e_A)}N^2(\Psi_N(\eta^{x,y})-\Psi_N(\eta))^2 1_{\{m(e_A)>\epsilon'N\}} \label{12.25}.
    \end{align}
    If $\zeta$ is bad, we have $2N^{1-\gamma}+1\geq |\zeta_B|\geq N^{1-\gamma}-1$ directly from the definition.
    Let $U$ be a open subset of $\Xi_A$ such that $$U \coloneq \{\xi \in \Xi_A : \xi_x > \epsilon' \text{ for all } x\in A\}.$$ Then
    \begin{align}
    A_N &\leq c \sum_{e \in \zeta} \frac{1}{a(e_A)} 1_{\{m(e_A)>\epsilon'N\}} = \frac{c N^{|A|-1}}{N^{|A|\alpha}(|A|-1)!}R(N-1-|\zeta_B|,\pi^{-1}1_U)
    \end{align}
    On the other hand, to calculate $B_N$, we need to consider a difference of $\Psi_N$.
    Put $$|\Psi_N(\eta^{x,y}) - \Psi_N(\eta)|\leq N^{\gamma-1}$$ into \eqref{12.25}, we have
    \begin{align}
        \label{eqeq41}
        B_N &\leq c \sum_{e \in \zeta} \frac{1}{a(e_A)}N^{2\gamma} 1_{\{m(e_A)>\epsilon'N\}} = \frac{cN^{|A|- 1}N^{2\gamma}}{N^{|A|\alpha}(|A|-1)!}R(N-1-|\zeta_B|,\pi^{-1}1_U)
    \end{align}
    For natural number $k$, we have
    \begin{align}
        \label{eqeqt42}
        \sum_{|\zeta_B|=k}\frac{1}{a(\zeta_B)} = \frac{Z_{k,B}}{k^\alpha}
    \end{align}
    from Proposition \ref{e21}.
    Therefore, joining \eqref{eqeq4}, \eqref{eqeq41} and \eqref{eqeqt42},
    \begin{align}
        \label{eqeq400}
        &\frac{N^\alpha}{Z_{N,S}} \sum_{\zeta \text{ is bad}}\frac{1}{a(\zeta_B)}\sum_{e\in \zeta}\frac{1}{a(e_A)}r(x,y)N^2(W_N(e_1)-W_N(e_0))^2 \nonumber \\
        &=\frac{N^\alpha}{Z_{N,S}} \sum_{\zeta \text{ is bad}}\frac{1}{a(\zeta_B)}O\left(\frac{N^{|A|-1}N^{2\gamma}}{N^{|A|\alpha}}\right)
        = O\left(\sum_{k=\lfloor N^{1-\gamma}-1 \rfloor}^\infty \frac{1}{k^\alpha} \frac{N^{|A|-1}N^{2\gamma}}{N^{|A|\alpha}} \right) \nonumber \\
        &=O\left(\frac{N^\alpha N^{(1-\gamma)(1-\alpha)} N^{|A|-1}N^{2\gamma}}{N^{|A|\alpha}}\right) = o\left(\frac{N^\alpha N^{|A|-1}}{N^{|A|\alpha}}\right).
    \end{align}

    Now, we estimate the sum \eqref{eqeq40} for good edges. Let $(x,y)$ be a vector in $\bb R^S$ such that $y$-th coordinate is $1$ and $x$-th coordinate is $-1$.
    Then for good edge $e=(\eta,\eta^{x,y})$, we have
    \begin{equation*}
        \ms E_A v\left(\frac{\eta^{x,y}}{N}\right) - \ms E_A v\left(\frac{\eta}{N}\right) = v\left(\gamma_A\left(\frac{\eta^{x,y}}{N}\right)\right) - v\left(\gamma_A\left(\frac{\eta}{N}\right)\right) = \frac{1}{N}\nabla_{\gamma_A((x,y))}v(\star),  
    \end{equation*}
    for some $\star \in (\gamma_A(\frac{\eta}{N}),\gamma_A(\frac{\eta^{x,y}}{N}))$ from the mean value theorem.
    Then we have
    \begin{align*}
        &\left|\nabla_{\gamma_A((x,y))}v(\star) - \nabla_{\gamma_A((x,y))}v\left(\iota_A\left(\frac{e_A}{N-1-|e_B|}\right)\right)\right| \leq c\left\|\iota_A\left(\frac{e_A}{N-1-|e_B|}\right) - \star\right\| \nonumber \\
        &\leq c \max\left(\left\|\iota_A\left(\frac{e_A}{N-1-|e_B|}\right)-\gamma_A\left(\frac{\eta}{N}\right)\right\|,\left\|\iota_A\left(\frac{|e_A|}{N-1-|e_B|}\right)-\gamma_A\left(\frac{\eta^{x,y}}{N}\right)\right|\right) = O(N^{-\gamma}).
    \end{align*}
    for some positive constant $c$ irrelevant to $N$ and $e$.
    Therefore, for good $\zeta$, we have
    \begin{align}
        \label{eqeq6}
        &\sum_{ e \in \zeta} \frac{1}{a(e_A)}r(x,y)N^2(W_N(e_1)-W_N(e_0))^2 \nonumber \\
        &= \sum_{e\in \zeta} \frac{1}{a(e_A)}r(x,y)\bigg[\left(\nabla_{\gamma_A((x,y))}v\left(\iota_A\left(\frac{e_A}{N-1-|e_B|}\right)\right)\right)^2 + O(N^{-\gamma})\bigg] 1_{\{m(e_A)>\epsilon'N\}} \nonumber \\
        & = C_N + D_N,
    \end{align}
    where $C_N$ and $D_N$ are defined as
    \begin{align*}
        C_N &= \sum_{e\in \zeta} \frac{1}{a(e_A)}r(x,y)\left(\nabla_{\gamma_A((x,y))}v\left(\iota_A\left(\frac{e_A}{N-1-|e_B|}\right)\right)\right)^2 1_{\{m(e_A)>\epsilon'N\}}, \\
        D_N &= \sum_{e\in \zeta} \frac{1}{a(e_A)}r(x,y)O(N^{-\gamma}) 1_{\{m(e_A)>\epsilon'N\}}.
    \end{align*}
    By similar calculations as in the equation (\ref{lemeq5}), for a large enough N, we have
    \begin{align}
        \label{eqeq7}
        C_N &= \sum_{e\in \zeta} \frac{1}{a(e_A)}r(x,y)\left(\nabla_{\gamma_A((x,y))}v\left(\iota_A\left(\frac{e_A}{N-1-|e_B|}\right)\right)\right)^2 1_{\{m(e_A)>\epsilon'N\}} \nonumber \\
        &= \frac{N^{|A|-1}}{N^{|A|\alpha}(|A|-1)!}r(x,y)R(N-1-|\zeta_B|,\pi^{-1}(\nabla_{\gamma_A((x,y))}v)^2)(1+o_N(1)) \nonumber \\
        &= \frac{N^{|A|-1}}{N^{|A|\alpha}(|A|-1)!}r(x,y)\left[\int_{\Xi_A}(\nabla_{\gamma_A((x,y))}v)^2 d\lambda_A + o_N(1)\right].
    \end{align}
    For $D_N$, use the same equation in (\ref{lemeqeq}), we have
    \begin{align}
        \label{eqeq8}
        D_N = o\left(\frac{N^{|A|-1}}{N^{|A|\alpha}(|A|-1)!}\right).
    \end{align}
    Joining \eqref{eqeq6}, \eqref{eqeq7} and \eqref{eqeq8},
    \begin{align}
        \label{eqeq4000}
        &\frac{N^\alpha}{Z_{N,S}} \sum_{\zeta \text{ is good}}\frac{1}{a(\zeta_B)}\sum_{e\in \zeta}\frac{1}{a(e_A)}r(x,y)N^2(W_N(e_1)-W_N(e_0))^2 \nonumber \\        
        &=\frac{N^\alpha}{Z_{N,S}} \sum_{\zeta \text{ is good}}\frac{1}{a(\zeta_B)}\frac{N^{|A|-1}}{N^{|A|\alpha}(|A|-1)!}\left[r(x,y)\int_{\Xi_A}(\nabla_{\gamma_A((x,y))}v)^2d\lambda_A + o_N(1)\right].
    \end{align}
    It is direct from the definition that
    \begin{equation*}
        \{\zeta \in \mc H_{N-1}^{B,*} : \zeta \text{ is good}\} = \{\zeta \in \mc H_{N-1}^{B,*} : |\zeta_B|<N^{1-\gamma}\} \text{ or } \{\zeta \in \mc H_{N-1}^{B,*} : |\zeta_B|<N^{1-\gamma} - 1\}.
    \end{equation*}
    Therefore, \eqref{eqeq4000} becomes
    \begin{equation*}
        \frac{N^\alpha N^{|A|-1} \Gamma^{|B|}}{Z_{N,S}N^{|A|\alpha}(|A|-1)!}\left[r(x,y)\int_{\Xi_A}(\nabla_{\gamma_A((x,y))}v)^2d\lambda_A + o_N(1)\right].
    \end{equation*}
    Therefore, to finish the proof, we need to show that
    $$ \sum_{x<y} r(x,y)\int_{\Xi_A}(\nabla_{\gamma_A((x,y))}v)^2 d\lambda_A = Q^A(u) = -\int_{\Xi_A} v \mf L^A v d\lambda_A.$$
    The following lemma proves it.
\end{proof}

\begin{lemma}
    \label{1206}
    For any $u\in C^\infty_c(\mathring{\Xi}_A,\Xi_A)$, we have
    \begin{equation*}
        \sum_{x<y} r(x,y)(\nabla_{\gamma_A((x,y))}u)^2 = - u \mf L^A u.
    \end{equation*}
\end{lemma}

\begin{proof}
    For $x,y\in A$, let $v^A_{x,y} = \sqrt{\frac{r^A(x,y)}{2}}$.
    From (\ref{1009}), the right-hand side is equal to
    $$ \sum_{x\neq y \in A} ((v^A_{x,y}\cdot \nabla)u) ((v^A_{x,y}\cdot \nabla)u)$$
    Considering $\nabla_{\gamma_A((x,y))}u$ as $\gamma_A((x,y)) \cdot \nabla u$, it is enough to show that for all vector
    $ v \in \bb R^A $,
    $$ \sum_{x<y} r(x,y) (\gamma_A((x,y))\cdot v)^2 = \sum_{x\neq y \subset A} (v^A_{x,y}\cdot v)^2.$$
    The left-hand side is equal to
    $$ \sum_{x<y} r(x,y) (v^\dagger \gamma_A (x,y)  (x,y)^\dagger \gamma_A^\dagger v) = v^\dagger \gamma_A \left[\sum_{x<y} (x,y) r(x,y) (x,y)^\dagger\right] \gamma_A^\dagger v = v^\dagger \gamma_A \mc L_S \gamma_A^\dagger v. $$
    Similarly, the right-hand side is equal to
    $$ v^\dagger \left[\sum_{x\neq y\in A} (x,y) \frac{r^A}{2} (x,y)^\dagger\right] v = v^\dagger \mc L^A v. $$
    By \eqref{traceprocess}, we obtain the desired result. 
\end{proof}

Finally, we have the proof of Proposition \ref{123}.

\begin{proof}[Proof of \ref{123}]
    Now, define $V_N$ as a normalization of $W_N$. Precisely, define $V_N$ as
    $$V_N(\eta) = \frac{W_N(\eta)}{\sqrt{\int_{\mc H_N} W_N^2 d\rho_N}}.$$
    From the above lemmas, we have the desired result.
\end{proof}

Using Proposition \ref{123}, we can check condition \textbf{(D1*.2)}.
Precisely, take $v^A\in C^\infty_c(\mathring{\Xi}_A,\Xi_A)$ for all $\varnothing\subsetneq A\subset S$.
From Proposition \ref{123}, there exists a sequence of functions $V^A_N: \mc H_N \rightarrow \bb R$ satisfying conditions of Proposition \ref{123}.
To check all conditions in \textbf{(D1*.2)}, it remains to check the following:
\begin{align}
    &(V^A_N V^B_N d\rho_N) \rightarrow 0 \text{ weakly}, \label{12.38} \\
    \lim_{N\rightarrow \infty} &\int_{\mc H_N} V^A_N (-\ms L_N V^B_N) d\rho_N = 0. \label{12.39}
\end{align}

Define the support of $V^A_N$ as $$\text{supp} (V^A_N) \coloneq \{\eta\in \mc H_N : V^A_N(\eta) \neq 0\}.$$
Since $V^A_N$ is a normalization of $W^A_N$, we have $V^A_N(\eta) \neq 0$ if and only if $W^A_N(\eta) \neq 0$.
From Lemma \ref{reglem}, $W^A_N(\eta) \neq 0$ implies $\min(\eta_A) > \epsilon' N$.
Also, from the equation \eqref{defv}, $W^A_N(\eta) \neq 0$ implies $|\eta_B| < 2N^{1-\gamma}$.
Therefore, we get
\begin{equation*}
    \text{supp} (V^A_N) \subset \{\eta\in \Xi_N : \min(\eta_A) > \epsilon' N \text{ and } |\eta_B| < 2N^{1-\gamma}\}.
\end{equation*}
From the definition of $\ms L_{N}$, we have
\begin{equation}
    \label{supp2}
    \text{supp} (\ms L_{N} V^A_N) \subset \{\eta\in \Xi_N : \min(\eta_A) > \epsilon' N - 1 \text{ and } |\eta_B| < 2N^{1-\gamma} + 1 \}.
\end{equation}
Denote the set in the right-hand side of \eqref{supp2} by $G_N(A)$.
Then we claim that for a large enough $N$, $\phi \subsetneq A,A'\subset S$, $A\neq A'$,
$$G_N(A) \cap G_N(A') = \emptyset.$$
Without loss of generality, we may assume that $|A| \leq |A'|$. Then we have $x\in A'$ such that $x\notin A$.
Then we have $\eta_x>\epsilon'N-1$ for all $\eta\in G_N(A)$ and $\eta_x<2N^{1-\gamma}+1$ for all $\eta\in G_N(A')$.
Take $N$ large so that
$$\epsilon'N-1 > 2N^{1-\gamma}+1.$$
So we have $G_N(A) \cap G_N(A') = \emptyset$. Therefore, the left-hand side of \eqref{12.38}, \eqref{12.39} vanishes, so the conditions are satisfied.

\section{Proof of Proposition \ref{otherscales}} \label{finalsection}

Define a metric $d$ on $\Xi$ given as
$$d(\xi,\zeta) = \sum_{x\in S} |\xi_x - \zeta_x|.$$
For any functions $f,g:P \rightarrow \bb R_+$ defined on some parameter space $P$, we say $f(x) \sim_{x\in P} g(x)$ if there exists $c,C>0$ such that
$$c < \frac{f(x)}{g(x)} < C.$$

\begin{lemma}\label{l111}
    Suppose $\ell_N \rightarrow \infty$ and $\ell_N \prec N$ given.
    For $\zeta \in \Xi$, let $B_N \coloneq \{\eta \in \mc H_N : d(\frac{\eta}{N},\zeta) \leq \frac{\ell_N}{N}\}$.
    For nonempty $\varnothing\subsetneq A\subset S$, suppose $\zeta \in \mathring{\Xi}_A$.
    Then,
    $$\rho_N(B_N)\left(\frac{N^\alpha}{\ell_N}\right)^{|A|-1} \sim_{N} 1.$$
\end{lemma}
\begin{proof}
    For $i\in \bb N$, let $B_{N,i}$ be a set of $\eta \in B_N$ such that $\sum_{x\in S\setminus A} \eta_x = i$.
    It is easy to check that
    $$|B_{N,i}| \sim_{(N,i)} (\ell_N - i + 1)^{|A|-1}.$$
    Then we have
    \begin{align*}
        \rho_N(B_{N,i}) = \sum_{\eta \in B_{N,i}} \rho_N(\eta) 
        &= \frac{N^\alpha}{Z_{N,S}} \sum_{\eta \in B_{N,i}} \prod_{x\in A} \frac{1}{a(\eta_x)} \prod_{y\in S\setminus A} \frac{1}{a(\eta_y)} \\
        &= \frac{N^\alpha}{Z_{N,S}} \prod_{x\in A} \frac{1}{\zeta_x^\alpha} \sum_{\eta \in B_{N,i}} \prod_{y\in S\setminus A} \frac{1}{a(\eta_y)}(1+o_N(1)) \\
        &= \frac{N^\alpha}{Z_{N,S}} \prod_{x\in A} \frac{1}{\zeta_x^\alpha} \sum_{\eta \in B_{N,i}} \frac{Z_{i,S\setminus A}}{i^\alpha}(1+o_N(1))
        \sim_{(N,i)} \frac{N^\alpha (\ell_N - i + 1)^{|A|-1}}{i^\alpha N^{|A|\alpha}}
    \end{align*}
    Summing $\rho_N(B_{N,i})$ over $0\leq i \leq \ell_N$, we get the desired result.
\end{proof}

\begin{lemma} \label{prescale}
    For $\theta_N \prec N^2$, $(\theta_N \mc I_N)$ $\Gamma$-converges to 0.
\end{lemma}
\begin{proof}
    It is enough to show that the $\limsup$ condition.
    Since the rate function is convex and lower semi-continuous, it is enough to show that for all $\xi_0 \in \Xi$,
    there exists a sequence of measures $\mu_N$ such that $\mu_N$ weakly converges to $\delta_{\xi_0}$ and
    $$\limsup_{N\rightarrow \infty} \theta_N \mc I_N(\mu_N) = 0.$$
    Fix $\xi_0 \in \Xi$. Assume that $\xi_0 \in \mathring{\Xi}_A$ for some nonempty $\varnothing\subsetneq A\subset S$.
    Take $\ell_N \rightarrow \infty$ such that $\sqrt{\theta_N} \prec \ell_N \prec N$.
    Define $\psi_N:\Xi \rightarrow \bb R$ by
    $$\psi_N(\xi) \coloneq \max\left(1 - \frac{d(\xi_0,\xi)N}{\ell_N}, 0\right).$$
    Let $\Psi_N: \mc H_N \rightarrow \bb R, \;\; \Psi_N \coloneq \psi_N \iota_N$.
    Define $Z_{N} = \int_{\Xi} \Psi_N^2 d\rho_N$. By Lemma \ref{l111}, we can show that
    $$Z_{N} \sim_{N} \left(\frac{\ell_N}{N^\alpha}\right)^{|A|-1}.$$
    Take $\nu_N \in \ms P(\mc H_N)$, $\nu_N = \frac{\Psi_N^2}{Z_{N}}\rho_N$. Finally take $\mu_N \in \ms P(\Xi)$,
    $$\mu_N = \nu_N \iota_N^{-1}.$$
    Since the support of $\mu_N$ is shrinking to $\xi_0$, we have $\mu_N$ weakly converges to $\delta_{\xi_0}$.
    
    Now, we calculate $\theta_N \mc I_N(\mu_N)$. Let $B_N$ be a set of $\eta \in \mc H_N$ such that
    $$B_N = \{\eta\in \mc H_N : \Psi_N(\eta) > 0 \text{ or } \Psi_N(\eta^{x,y}) > 0 \text{ for some } x,y\in S \}.$$
    From Lemma \ref{l111}, we can show that
    $$\rho(B_{N}) \sim_{N} \left(\frac{\ell_N}{N^\alpha}\right)^{|A|-1}.$$
    Then we have
    \begin{align*}
        \theta_N \mc I_N(\mu_N) = \frac{\theta_N}{Z_{N}} \int_{\mc H_N} \Psi_N (-\ms L_N \Psi_N) d\rho_N &= 
        \frac{\theta_N}{Z_{N}} \sum_{\eta \in \mc H_N} \sum _{x,y\in S} \rho(\eta)g(\eta_x)r(x,y)(\Psi_N(\eta^{x,y})-\Psi_N(\eta))^2 \\
        &\sim_N \frac{\theta_N}{Z_{N}} \sum_{\eta \in B_N} \sum _{x,y\in S} \rho(\eta)(\Psi_N(\eta^{x,y})-\Psi_N(\eta))^2 \\
        &\leq \frac{\theta_N}{Z_{N}} \sum_{\eta \in B_N} \sum _{x,y\in S} \rho(\eta)\frac{1}{\ell_N^2} \sim_N \frac{\theta_N}{\ell_N^2} = o_N(1).
    \end{align*}
    Therefore, the desired result holds.
\end{proof}

\begin{lemma} \label{zerocond}
    For the rate function $\mc K$ defined in \eqref{ratedef}, $\mc K(\mu) = 0$ if and only if $\mu$ is supported on $\{\xi^x: x\in S\}$.
\end{lemma}
\begin{proof}
    Suppose $\mc K(\mu) = 0$. Then Lemma \ref{93} implies that $\mc K_h(\mu)=0$ for all $h\in \mc H$.
    Decompose $\mu$ as $\mu = \sum_{\varnothing\subsetneq A\subset S} \mu(\mathring{\Xi}_A) \mu(\cdot|\mathring{\Xi}_A)$.
    Then, we must show that $\mu(\mathring{\Xi}_A) = 0$ for all $|A|\geq 2$.
    Suppose there exists $\varnothing\subsetneq A\subset S$ such that $\mu(\mathring{\Xi}_A) > 0$, $|A| \geq 2$.
    From \eqref{1131}, for any compact set $K$ contained in $\mathring{\Xi}_A$, we have
    $$ \mu P_h(K) - \mu(K) \geq 0.$$
    Since $\mu(\mathring{\Xi}_A) > 0$, there exists a compact set $K$ contained in $\mathring{\Xi}_A$ such that $\mu(K) > 0$.
    We claim that $\mu P_h(K)$ converges to 0 as $h$ goes to infinity.
    For $\xi\in \mathring{\Xi}_A$, let $\sigma_1$ be a first hitting time of the process hitting $\Xi_A \setminus \mathring{\Xi}_A$.
    Then \cite[Theorem 7.12]{MZRP} implies that for all $\xi \in \mathring{\Xi}_A$,
    $$\bb E_{\xi} [\sigma_1] < \infty.$$
    Therefore, for $\xi \in \mathring{\Xi}_A$,
    \begin{align*}
        P_h 1_K(\xi) = \bb E_{\xi} [1_K(\xi_h)] \leq \bb P_{\xi} [\sigma_1 > h] \leq \frac{\bb E_{\xi} [\sigma_1]}{h}
    \end{align*}
    Therefore, $P_h 1_K$ converges to 0 pointwise as $h$ goes to infinity. From the dominated convergence theorem, we get
    $$\mu P_h(K) \rightarrow 0,$$
    which leads to a contradiction.
\end{proof}

\begin{proof}[Proof of Proposition \ref{otherscales}]
    For $\theta_N \prec N^2$, Lemma \ref{prescale} implies that $(\theta_N \mc I_N)$ $\Gamma$-converges to 0.
    For $N^2 \prec \theta_N \prec N^{1+\alpha}$, we show that $(\theta_N \mc I_N)$ $\Gamma$-converges to $X_S$ presented in \eqref{midconv}.
    By the definition of $X_S$, we need to show
    \begin{enumerate}
        \item For any $\mu \in \ms P(\Xi)$, $\mu$ supported on $\{\xi^x: x\in S\}$, there exists a sequence of measures $\mu_N$ such that $\mu_N$ weakly converges to $\mu$ and
              $$\lim_{N\rightarrow \infty} \theta_N \mc I_N(\mu_N) = 0.$$
        \item For any $\mu \in \ms P(\Xi)$, $\mu$ not supported on $\{\xi^x: x\in S\}$, for all sequence of measures $\mu_N$ weakly converges to $\mu$, we have
              $$\liminf_{N\rightarrow \infty} \theta_N \mc I_N(\mu_N) = \infty.$$
    \end{enumerate}
    (1) follows from Theorem \ref{27} and (2) follows from Lemma \ref{zerocond}.
    Finally, for $\theta_N \succ N^{1+\alpha}$, the $\Gamma$-convergence of $(\theta_N \mc I_N)$ to $U$ of \eqref{postconv} is straightforward.
\end{proof}

\subsection*{Acknowledgement}
This research is supported by the National Research Foundation of Korea (NRF) grant funded by the Korea government (MSIT) (No. 2023R1A2C100517311) and the undergraduate research program through the Faculty of Liberal Education,
Seoul National University. The author thanks to Insuk Seo introducing the problem and offering enlightening insights, and Claudio Landim, Jungkyoung Lee, and Seonwoo Kim regarding fruitful discussion on the problem.

\appendix

\section{Semigroup Theory}

In this section, we introduce basic definitions and lemmas related to semigroup theory, with a particular focus on the core of the infinitesimal generator of a semigroup.
Lemma \ref{corelem} serves as a practical tool for identifying the core of an infinitesimal generator.
Additionally, we elaborate on concepts pertinent to semigroup theory to facilitate the description of the lemma.

\begin{defn}
    Let $X$ be a Banach space and $B(X)$ be a space of bounded linear map to itself. $T:[0,\infty)\rightarrow B(X)$ is called a strongly continuous semigroup iff
    \begin{enumerate}
        \item [\textup{(1)}] $T(0) = I$.
        \item [\textup{(2)}] $T(t+s) = T(t)T(s)$ for all $t,s\geq 0$.
        \item [\textup{(3)}] $T(t)x \rightarrow x$ as $t\rightarrow 0$ for all $x\in X$.
    \end{enumerate}
\end{defn}

Let $S$ be a compact metric space. Let $T(t)$ be a semigroup on $C(S)$.

\begin{defn}
    The infinitesimal generator $A$ of a strongly continuous semigroup $T(t)$ is defined by
    \begin{equation*}
        A x = \lim_{t\rightarrow 0} \frac{T(t)x - x}{t}.
    \end{equation*}
    whenever the limit exists. The domain of $A$, $D(A)$, is the set of $x\in X$ for which this limit does exist.
\end{defn}

Now, we introduce the concept of a core of an unbounded operator.

\begin{defn}
    For any unbounded operator between Banach spaces $A:X\rightarrow Y$, defined on $D(A) \subset A$,
    we say $A$ is closed if for any sequence $\{x_n\}\subset D(A)$ such that $x_n \rightarrow x$ and $Ax_n \rightarrow y$, we have $x\in D(A)$ and $Ax = y$.
    If $A$ is closed, we say $D$ is a core of $A$ if for all $x\in Dom(A)$, there exists a sequence $\{x_n\}\subset D$ such that $x_n \rightarrow x$ and $Ax_n \rightarrow Ax$.
    Also, we say $A$ is densely defined if $D(A)$ is dense in $X$.
\end{defn}

\begin{lemma}\cite[Theorem 2.2.6]{Partington}
    \label{gencld}
    An infinitesimal generator $A$ of a strongly continuous semigroup is closed and densely defined.
\end{lemma}

The next lemma is a practical tool for identifying the core of an infinitesimal generator.

\begin{lemma}
    \label{corelem}
    $T(t)$ be a strongly continuous semigroup on a Banach space $X$.
    $U(t)$ be a corresponding resolvent operator.
    Let $D$ be a dense subset of $X$ such that one of the following holds.
    \begin{enumerate}
        \item [\textup{(1)}] $D$ is closed under $T(t)$ for all $t>0$.
        \item [\textup{(2)}] $D$ is closed under $U(t)$ for some $t>0$.
    \end{enumerate}
    Then $D$ is a core of infinitesimal generator of $T(t)$.
\end{lemma}

\begin{proof}
    For (1), see \cite[Theorem 1.9]{Davies}.
    For (2), follow the proof in \cite[Theorem 1.34]{Levy}.
\end{proof}

Now, for a compact metric space $S$, let $X = C(S)$ and $T(t)$ be a Feller semigroup on $X$. Let $\ms L$ be the infinitesimal generator of $T(t)$.
The next lemma belongs to a generator of a Feller process.

\begin{lemma}
    \label{reslem}
    The following holds.
    \begin{enumerate}
        \item [\textup{(1)}] $f = sU_s g$, $g\in C(S)$. Then $f\in D(\ms L)$ and $f-\frac{1}{\alpha}\ms L f = g$.
        \item [\textup{(2)}] For $f\in C(S) \text{ and } t>0$, $\exists c > 0$, $$\left\|\left(\frac{n}{t}U_{\frac{n}{t}}\right)^n f - T(t) f\right\|_\infty \leq \frac{ct}{\sqrt{n}} \|\ms L f\|_\infty.$$
    \end{enumerate}
\end{lemma}

\begin{proof}
    See \cite[Theorem 3.16]{liggett}.
\end{proof}

\section{Dirichlet Problem and Heat Kernel Estimate}

This section is devoted to providing a brief overview of the Dirichlet problem and heat kernel estimates related to the probabilistic perspective.

\subsection{A Dirichlet problem}
These results are from \cite[Remark 7.5]{KARA}.
Suppose we have an elliptic operator of the form
\begin{equation*}
    L u \coloneq a^{ij}(x) D_{ij}u + b^i D_i u
\end{equation*}
is given on $\bb R^n$. Consider a bounded continuous functions $k:\bb R^n\rightarrow [0,\infty)$ and $g:\bb R^n \rightarrow \bb R$.
Let $\Omega \subset \bb R^n$ be a bounded open domain.
Take a function $f:\partial \Omega \rightarrow \bb R$. The Dirichlet problem is to find a continuous function
$u:\bar{\Omega}\rightarrow \bb R$ such that such that $u$ is of class $C^2(\Omega)$ and satisfies the elliptic equation
\begin{equation}
    \begin{aligned}
        \label{diriprob}
            Lu - k u &=\; -g \text{ in } \Omega, \\
            u &=\; f \text{ on } \partial \Omega.
    \end{aligned}
\end{equation}
For $0<\beta<1$, we assume the following conditions. 
\begin{enumerate}
    \item $L$ is uniformly elliptic,
    \item $a^{ij}, b^i, k, g$ are $\beta$-H\"older continuous on $\bb R^n$, and
    \item every point $a\in \partial \Omega$ has the exterior sphere property; that is, there exists a ball $B(a)$ such that
    $B(a)\cap \Omega = \emptyset$ and $\bar{B}(a)\cap \partial \Omega = \{a\}$,
    \item $f$ is continuous.
\end{enumerate}
According to \cite[Theorem 6.13]{EPDE}, there exists a unique function $u$ of class $C^{2,\beta}(\Omega)\cap C(\bar{\Omega})$, which solves \eqref{diriprob}.

From \cite[Corollary 4.29]{KARA} and \cite[Remark 4.30]{KARA}, it is established that there exists a unique solution to the martingale problem associated with \( L \).
Subsequently, the following theorem is derived from \cite[Proposition 7.2]{KARA}.

\begin{thm}
    \label{probrepn}
    Let $\tau_\Omega \coloneq\inf \{t\geq 0 : X_t\notin \Omega\}$. Under the above conditions the unique solution $u$ of \eqref{diriprob} can be written as
    \begin{equation*}
        u(x) = \bb E_x\left[f(X_{\tau_\Omega})\exp\left(-\int_0^{\tau_\Omega} k(X_s)ds\right) + \int_0^{\tau_\Omega} g(X_t)\exp\left(-\int_0^t k(X_s)ds\right)dt\right].
    \end{equation*}
    where $\bb E_x$ is the expectation with respect to the probability measure $\bb P_x$ under which $X$ is a solution of martingale problem associated with $L$.
\end{thm}

\subsection{Heat kernel estimates}
    In this section, we provide estimates of heat kernel for the process $X$ on $\bb R^n$, which is the solution to the martingale problem associated with the operator $L$ as described in the previous section.
    We first define a heat kernel associated with the process $X$.

    \begin{defn}
        A function $p:(0,\infty)\times \bb R^n \times \bb R^n \rightarrow [0,\infty)$ is called the heat kernel of the process $X$ if for all $t>0$ and $x,y\in \bb R^n$,
        \begin{equation*}
            \bb E_x[f(X_t)] = \int_{\bb R^n} f(y)p(t,x,y)dy.
        \end{equation*}
    \end{defn}

    The following lemma, taken from \cite[display (1.2)]{esttrans}, provides an estimate for the heat kernel of the process $X$.

    \begin{lemma}
        \label{heatker}
        Under the assumptions above, there exists the heat kernel $p$ of the process $X$.
        Furthermore, there exists positive constants $c_1$ and $K_1$ depending only on $L$ such that
        \begin{equation*}
            p(t,x,y) \leq K_1 t^{-\frac{n}{2}}\exp\left(-c_1\frac{|x-y|^2}{t}\right) \; \text{ for all }\; x,y\in \Omega,\; t>0.
        \end{equation*}
    \end{lemma}

    Let $\Omega$ be a bounded open domain in $\bb R^n$. Consider a killed process $X^\Omega$ so that the process is killed upon exiting $\Omega$.
    Since a transition density for the killed process is less or equal to a transition density for the original process, we obtain the same heat kernel estimate for the killed process.
    Therefore, we have the following.
    \begin{cor}
        \label{heatker2}
        Fix a $\delta>0$. There exists the heat kernel $p_\Omega(t,x,y)$ associated with the killed process $X^\Omega$.
        Moreover, there exists a positive constant $C$ depending only on $\delta$ such that
        \begin{equation*}
            p_\Omega(t,x,y) \leq C,\; \text{ for all }\; x,y\in \Omega,\;\; |x-y|\geq\delta,\; t>0.
        \end{equation*}
    \end{cor}

\end{document}